\documentclass[letterpaper,11pt,reqno,twoside]{article}
\usepackage[english]{babel}
\usepackage[utf8]{inputenc}
\usepackage{amsmath,amsfonts,amssymb,amsthm,mathabx,secdot}
\usepackage{etoolbox}
\usepackage{enumerate}
\usepackage{graphicx}
\graphicspath{{images/}}
\usepackage{array,multirow,multicol,adjustbox}
\usepackage{algorithm}
\usepackage{algorithmic}
\usepackage{gnuplotikz}

\usepackage[margin=2cm]{geometry}
\usepackage[scr=pxtx]{mathalpha}
\usepackage{comment}

\newtheorem{thm}{Theorem}[section]
\newtheorem{lem}[thm]{Lemma}

\newtheorem{pro}[thm]{Proposition}

\theoremstyle{remark}
\newtheorem{rmk}[thm]{Remark}
\AtBeginEnvironment{rmk}{%
  \pushQED{\qed}%
}
\AtEndEnvironment{rmk}{%
  \popQED
}
\theoremstyle{definition}
\newtheorem{dfn}[thm]{Definition}

\usepackage{caption}  
\usepackage{subcaption}

\usepackage{xcolor}

\newcommand{\jg}[1]{\textcolor{violet}{[{\tt JG:} {#1}]}}

\newcommand{\mytitle}{%
Families of relative periodic orbits in the planar three-body problem via consecutive alignments
}

\newcommand{\myshorttitle}{%
RPO families in the PTBP%
}

\usepackage[hidelinks]{hyperref}
\hypersetup{
bookmarksnumbered=true,
bookmarksopen=true,
pdftitle={\myshorttitle{}},
pdfauthor={J. Gimeno, \`A. Jorba, M. Jorba-Cusc\'o, and B. Nicol\'as},
}
\usepackage{orcidlink}

\title{\mytitle{}}
\author{%
Joan Gimeno\,\orcidlink{0000-0002-8707-6379}\textsuperscript{(1)} \and %
\`Angel Jorba\textsuperscript{\textdagger} \and %
Marc Jorba-Cusc\'o\,\orcidlink{0000-0003-0308-3756}\textsuperscript{(2,3)} \and %
Bego\~{n}a Nicol\'as\,\orcidlink{0000-0001-7214-5864}\textsuperscript{(4,\textasteriskcentered)}%
}

\newcommand{\R}{{\mathbb R}}

\newcommand{\Q}{{\mathbb Q}}

\newcommand{\I}{\boldsymbol{i}}

\newcommand{\dop}{\mathrm{D}} 
\newcommand{\idmat}{Id}       

\newcommand{\bydef}{\,\stackrel{\mbox{\tiny\textnormal{\raisebox{0ex}[0ex][0ex]{def}}}}{=}\,}
\newcommand{\ve}[1]{\mathbf{#1}}
\newcommand{\qq}{\ve{q}} 
\newcommand{\pp}{\ve{p}} 
\newcommand{\ee}{\ve{e}} 
\newcommand{\LL}{\ve{L}} 
\newcommand{\llm}{\ve{l}} 
\newcommand{\zz}{\ve{z}} 
\newcommand{\JJ}{\ve{J}} 

\newcommand{\uu}{\ve{u}} 
\newcommand{\as}{\ve{a}} 
\newcommand{\0}{\ve{0}} 

\newcommand{\QQ}{\ve{Q}} 
\newcommand{\PP}{\ve{P}} 

\newcommand{\II}{\mathcal{I}} 
\newcommand{\SSS}{\mathcal{S}} 

\newlength{\mainfigw}
\newlength{\zoomfigw}
\newcommand{\mainplot}[3]{%
    \includegraphics[page=#1,width=\mainfigw]{#2/#3/#3-crop.pdf}%
}

\newcommand{\zoomplot}[3]{%
    \includegraphics[page=#1,width=\zoomfigw]{#2/#3/#3-zoomin-crop.pdf}%
}

\newcommand{\iterTplot}[3]{%
    \includegraphics[page=#1,width=\zoomfigw]{#2/#3/#3-10T-crop.pdf}%
}

\newlength{\maincurvefigw}
\newlength{\rightmaincurvefigw}
\newlength{\ecccurvefigw}
\newcommand{\maincurveplot}[3]{%
    \includegraphics[page=#1,width=\maincurvefigw]{#2/#3-crop.pdf}%
}
\newcommand{\rightmaincurveplot}[3]{%
    \includegraphics[page=#1,width=\rightmaincurvefigw]{#2/#3-crop.pdf}%
}
\newcommand{\eccentrycurveplot}[3]{%
    \includegraphics[page=#1,width=\ecccurvefigw]{#2/#3-crop.pdf}%
}

\newcommand{\rowlab}[1]{%
  \raisebox{10pt}{
    \rotatebox[origin=c]{90}{\footnotesize #1}%
  }
}

\newcommand{\rowlabten}[1]{%
  \raisebox{30pt}{
    \rotatebox[origin=c]{90}{\footnotesize #1}%
  }
}

\begin{document}
\maketitle

{\small
\begin{itemize}
\renewcommand{\itemsep}{.5pt}
\item[(1)] Departament de Matem\`atiques i Inform\`atica, Universitat
  de Barcelona, Gran Via de les Corts Catalanes 585, 08007 Barcelona,
  Spain, {\tt joan@maia.ub.es} 
\item[(2)] Departament de Matem\`atiques, Universitat Polit\`ecnica de Catalunya, 
Av. Diagonal 647, 08028, Barcelona, Spain, 
{\tt marc.jorba@upc.edu} 
\item[(3)] Centre de Recerca Matem\`atica (CRM), 
Edifici C, Campus UAB, 0 Floor, 08193 Bellaterra, Barcelona, Spain
\item[(4)] Departamento de Matem\'atica Aplicada, Universidade de Santiago de Compostela, Rua Lope G\'omez De Marzoa, s/n, 15705, Santiago De Compostela (A Coru\~na), Spain,
{\tt bego.nicolas@usc.es}
\item[\textsuperscript\textasteriskcentered] Corresponding Author

\end{itemize}
}

\begin{abstract}
Relative periodic orbits (RPOs) are solutions of the three-body problem that are periodic in a uniformly rotating reference frame and, in general, quasi-periodic in inertial coordinates. We present a numerical procedure for computing and continuing one-parameter families of RPOs of the planar Newtonian three-body problem. 

The method exploits consecutive syzygies, understood here as configurations in which the three bodies are aligned and their velocities satisfy the corresponding symmetry conditions. Matching the positions and momenta at two consecutive alignments reduces the computation of RPOs to a low-dimensional nonlinear problem. Its solutions are then numerically continued, and linear stability is determined from the nontrivial eigenvalues of the rotated monodromy matrix after removing the neutral directions associated with conserved quantities and continuous symmetries. 

\smallskip

The procedure is applied to several mass distributions and initial configurations, producing families of Poincar\'e, Hill, and binary-type solutions. These families exhibit transitions from nearly circular to highly eccentric motion, changes of stability near resonances and turning points, and absolute periodic solutions when the rotation angle is a rational multiple of $2\pi$. 
In the Hill families, the continuation connects satellite configurations with circumstellar motion as the smallest body loses its gravitational binding to the intermediate body. Circumbinary and circumstellar configurations are also obtained in the binary regime. 

The results illustrate the dynamical diversity of RPOs and provide coherent three-body motions that can be used as prescribed trajectories in restricted four-body models.
\end{abstract}

{\footnotesize {\it 2020 Mathematics Subject Classification:
70F07
; 70H12
; 70H14
; 37M20
}

\medskip

}

\phantomsection\pdfbookmark[1]{\contentsname}{toc}
\pagestyle{myheadings}
\markboth{
\myshorttitle{}}{J. Gimeno, M. Jorba-Cusc\'o, and B. Nicol\'as}
\newpage
{\small \tableofcontents}
\newpage

\section{Introduction}

The Newtonian Three-Body Problem (TBP) is a paradigmatic non-integrable Hamiltonian system 
in which first integrals coexist with continuous symmetries and a remarkably rich set of recurrent motions. Among TBP solutions, relative periodic orbits (RPOs) are particularly natural in the presence of rotational symmetry. If $\gamma(t)$ denotes a trajectory in phase space, an RPO satisfies
\[
    \gamma(t+T)=\ve{R}_\theta\gamma(t)
\]
for some return time $T>0$, $\ve{R}_{\theta}=\operatorname{diag}(\mathcal{R}_{\theta})$ and planar rotation $\mathcal R_\theta$, where $\theta$ is the net rotation accumulated during one relative period. Thus, a trajectory returns after a finite time to its initial state up to a rotation; equivalently, it is periodic in a suitably chosen uniformly rotating frame. If $\theta/(2\pi)\in\mathbb{Q}$, repeated application of the relative return produces an absolute periodic orbit (APO); otherwise the inertial motion is, generically, quasi-periodic. 

Relative periodic motion is not specific to celestial mechanics: RPOs are basic invariant objects in Hamiltonian systems with continuous symmetry, where reduction, momentum conservation, neutral directions, stability, and bifurcation are tightly connected \cite{OrtegaRatiu1999}. In the planar TBP, such solutions are organized into one-parameter families \cite{henon1974}, so numerical continuation provides a natural way to study how recurrent motions, resonances, and stability change beyond perturbative regimes.

\smallskip

Periodic and relative periodic solutions have been central to the qualitative study of the TBP since Poincar\'e. In the planetary regime, Poincar\'e classified solutions according to their eccentricities and inclinations and established the continuation of nearly Keplerian motions when the planetary interactions are sufficiently weak \cite[p.~97]{Poincare92}. Classical numerical studies later followed several planar families and analyzed their stability, including planetary-type solutions and continuations related to periodic orbits of restricted problems \cite{hadjidemetriou1975A,hadjidemetriou1976A,delibaltas1976,henon1976,message1980}. These works already demonstrated that families born near weakly coupled Keplerian motion may extend far from that limit, encounter resonances and turning points, change stability, and reach large eccentricities. In this work, we return to those families, but computing and continuing the recurrent motions directly in the full planar TBP and applying the same formulation to planetary, satellite, and binary configurations.

\smallskip

From a numerical point of view, an RPO is more delicate than an ordinary periodic orbit (i.e. when $\theta = 0$). Relative periodicity contains both an unknown return time $T$ and an unknown group action $\theta$, while the first integrals and continuous symmetries of the TBP produce neutral directions in the variational dynamics. Consequently, a naive periodic-orbit formulation carries degeneracies in both the nonlinear correction problem and the monodromy matrix. General theory describes the local structure and stability of Hamiltonian RPOs in the presence of symmetry \cite{OrtegaRatiu1999}. Our aim here is complementary: we exploit a geometric feature specific to the planar three-body problem ---consecutive alignments of the bodies--- to obtain a particularly direct low-dimensional formulation that can be continued far from the nearly circular regime.

The construction starts from two consecutive syzygies, understood here as configurations in which the three bodies are aligned and the velocities satisfy the corresponding symmetry conditions.
In the planar TBP, alignments define a hypersurface that serves as a local Poincar\'e section wherever the flow is transverse. We explicitly characterize the configurations where transversality fails (see Proposition~\ref{pro.nontransverse}). We show that suitable matching conditions between two consecutive alignments imply that the second state is a rigid rotation of the first and therefore lies on an RPO. After using the center-of-mass constraints, fixing the initial alignment, and exploiting the scaling symmetry of the Newtonian equations, the initial state is described by only three independent variables. The relative-periodicity conditions are solved with a Newton correction, and the resulting one-dimensional solution set is followed by pseudo-arclength continuation. This construction works directly with trajectories of the full TBP: it does not require prescribing a uniformly rotating model in advance, nor does it assume that the continued solutions remain close to circular Keplerian motion.

The symmetry structure is also incorporated into the stability computation. For an RPO, the standard monodromy matrix must be composed with the inverse accumulated rotation in order to compare perturbations at equivalent points on the group orbit. The resulting rotated monodromy matrix contains trivial multipliers associated with the flow, rotational symmetry, and conserved quantities. In our implementation, the corresponding neutral directions are explicitly constructed and removed through a singular-value-decomposition subspace projection. The remaining multipliers provide the nontrivial linear stability information. Thus the orbit computation, continuation, and stability analysis are handled within a common symmetry-adapted process rather than as separate numerical procedures.

The resulting continuations display several types of global behavior. For Poincar\'e-type, or planetary solutions, families initiated near quasi-circular configurations pass through resonant values of the relative rotation, develop turning points, and reach highly eccentric motions. Along these branches we observe changes of stability and a continuous transition from Poincar\'e's first-species solutions to second-species solutions. Some features recover the classical picture reported by Hadjidemetriou \cite{hadjidemetriou1976A}, whereas direct continuation in the full TBP allows the computation to proceed into highly eccentric regimes that were difficult to access in that earlier continuation work. Rational values of the accumulated rotation identify absolute periodic solutions embedded in the RPO families.

The Hill-type families exhibit a qualitatively different transition. The smallest body initially behaves as a satellite of an intermediate body while the latter revolves around the dominant mass. Along the same continuous family of RPOs, the satellite becomes progressively less bound to the intermediate body and eventually moves in a circumstellar configuration. In this sense, satellite and circumstellar recurrent motions appear as different regimes of a single global family rather than as unrelated constructions. The binary-type computations provide a third class of examples, containing both circumbinary and circumstellar configurations. 

\smallskip

A second motivation for studying these families comes from restricted four-body dynamics. The Circular Restricted Three-Body Problem (CRTBP) remains one of the basic models for the motion of a small particle under the attraction of two massive primaries. Its triangular equilibrium points are central to co-orbital dynamics and Trojan populations, and for many relevant mass ratios they are linearly stable in the ideal circular model. Real systems, however, are perturbed by effects absent from the CRTBP. Eccentric motion of the primaries can modify the stability near the triangular points \cite{Danby64}, and the gravitational influence of an additional massive body has long been considered in connection with the scarcity or destabilization of Trojan-like populations in some systems, including the influence of Jupiter on other planetary configurations \cite{Innanen1989,Holman1993,Marzari2000} and the influence of the Sun on the Earth--Moon triangular regions \cite{GomezJMS91a,SimoGJM95}. These examples motivate a basic modeling question: \emph{when a third primary is dynamically relevant, how should the motion of the three massive bodies themselves be prescribed?}


In the classical Bicircular Problem (BCP) \cite{Huang60,CroninRR64}, the additional massive body is assigned a circular motion around the original CRTBP configuration. This gives an analytically convenient time-dependent model, but the prescribed motion of the three massive bodies is not, in general, an exact solution of the TBP. A coherent restricted four-body model instead requires the three primaries to follow an actual three-body trajectory. One established route is the Quasi-Bicircular Problem (QBCP), in which weakly coupled Keplerian motions are refined to produce a coherent, nearly circular three-body solution. Such constructions have been used for the Sun--Jupiter--Saturn and Sun--Earth--Moon systems \cite{hadjidemetriou1980,Andreu98,Andreu02,GabernJ01}. The underlying quasi-bicircular solutions are particular RPOs of the TBP, and their existence near the weak-interaction regime is connected with the classical continuation picture developed by Poincar\'e \cite{Poincare92}.

This observation provides the link between the application-driven motivation and the dynamical problem studied here. Rather than treating a quasi-bicircular orbit as an isolated prescribed motion, we regard it as a point on a family of RPOs and ask how far that family can be continued as the interactions become stronger. In particular, what dynamical regimes are connected to the nearly circular solutions, where do stability changes occur, and what limits are encountered as the geometry departs from the perturbative picture? The Poincar\'e, Hill, and binary families computed show that this continuation can lead to highly eccentric configurations and to changes in the qualitative organization of the motion. Consequently, the available coherent primary motions are much richer than the near-circular examples that originally motivated quasi-bicircular models.


The main purpose of this paper is nevertheless the three-body dynamics itself. 
The restricted four-body motivation supplies one application of the RPO motions, but the computational and dynamical questions are independent of that application. In particular, the families provide a global view of how recurrent solutions in a classical Hamiltonian system connect orbital regimes that can appear very different when considered individually.

\smallskip

The paper is organized as follows. Section~\ref{sec:TBP} formulates the planar TBP, introduces absolute and relative periodic motion, and develops the stability analysis for RPOs. Section~\ref{sec:NumComp} derives the consecutive-alignment conditions and presents the continuation procedure. Section~\ref{sec:AC} studies the resulting Poincar\'e-, Hill-, and binary-type families, with emphasis on their stability, resonances, turning points, and changes of orbital regime. Finally, Section~\ref{sec:conclusions} summarizes the main conclusions and discusses extensions toward coherent restricted four-body models and the subsequent inclusion of a massless particle.

\section{The Three-Body Problem}\label{sec:TBP}
Let us consider the planar motion of three punctual bodies $B_0$, $B_1$ and $B_2$ (with 
masses $m_0 , m_1 , m_2 >0$) interacting to 
each other according to the universal gravitational law. Their trajectories and velocities 
of the masses are contained in the phase space of the planar Three-Body Problem (TBP), which is 
governed by the following Ordinary Differential Equation,
\begin{equation}\label{eq:tbp}
    m_i \ddot{\qq_i} = \mathcal{G} \sum_{i< j} \frac{m_i m_j}{\|\qq_{i,j}\|^{3}} \qq_{i,j}, \quad i,j\in \{0,1,2\} 
\end{equation}
where $\qq_i \bydef (x_i , y_i)$ is the position of each body with respect to the common center of mass, $\qq_{i,j} \bydef \qq_j - \qq_i$, and $\mathcal{G}$ is the gravitational constant.

The phase space of this system is twelve dimensional i.e. two positions and two velocities for 
each body, namely $(x_i , y_i , \dot{x}_i , \dot{y}_i)$ 
The subspace of all possible positions
(the configuration space, $\mathcal{C_{M}}$) is defined as 
$$\mathcal{C_{M}}= \left \{ (\qq_0 , \qq_1 , \qq_2 )  \in \mathbb{R}^6 ~|~  m_{0}\qq_{0} + m_{1}\qq_{1} + m_{2}\qq_{2}= 0 \right \} \setminus \Delta,$$
where $\Delta= \bigcup_{i,j} \Delta_{i,j}$ and $\Delta_{i,j} =\left \{ \qq = (\qq_0 , \qq_1 , \qq_2) \in \mathbb{R}^6 ~|~ \qq_i = \qq_j \right \}$ for $i,j\in \{0,1,2\}$.  

The flow can reach the space $\Delta$ in finite time. Namely, the flow of the 
TBP is \textbf{incomplete}. When such a situation happens, we say that the solution 
encounters a \textbf{collision}. In the context of the TBP, collisions 
may be double (these can be regularized analytically \cite{Siegel1941} and topologically
\cite{Easton1971}) or triple. When the number of bodies 
interacting is strictly larger than three, the flow can reach singularities 
without collision. The set of initial conditions that lead to collision has null
Lebesgue measure and it is of Baire first category. Classical works by Donald Saari 
tackle these questions. See \cite{Saari1971, Saari1973, Saari1975}. 
Triple collisions can almost never be analytically regularized 
and never topologically regularized \cite{Marchal1982}
Moreover, according to Marchal's lemma, action minimizing solutions  do not contain 
singularities \cite{Marchal2002}.  

\smallskip

It is usual to name the potential 
$$U(\qq) =  \mathcal{G} \sum_{j\neq i} \frac{m_i m_j}{\|\qq_{i,j}\|},$$
and the masses matrix $M=\operatorname{diag}( m_0 , m_0, m_1 , m_1 , m_2 , m_2)$, so 
equation \eqref{eq:tbp} can be written in compact form 
$$M \ddot{\qq} = \nabla U (\qq).$$
Moreover, defining the momenta as $\pp_i = m_i \dot{\qq}_i$, the Hamiltonian of 
the system is
\begin{equation}\label{eq:tbp_ham}
\mathcal{H}(\qq, \pp) = \frac{1}{2} \pp^{T} M^{-1} \pp - U (\qq).
\end{equation}
The Hamiltonian function is, of course, a conserved quantity also known as energy first integral.
Besides, there are two other known independent conserved quantities: The integral of the center of 
mass,
\begin{equation}\label{eq:tbp_cm}
\sum_{i=0}^{2} m_i \qq_i = 0,
\end{equation}
and the integral of the angular momentum 
\begin{equation}\label{eq:tbp_am}
    \LL=\sum_{i=0}^{2} \qq_{i} \times \pp_i.
\end{equation}
\begin{rmk}[Two dimensional cross product]\label{rmk:2dcp}
The operator $\times$, the cross product in two dimensions, is defined as follows: 
Given $\ve{a}=(a_1 , a_2),~\ve{b}=(b_1 , b_2) \in \mathbb{R}^{2}$ then 
$\ve{a}\times\ve{b}=a_{1}b_{2} - b_{1} a_{2}$.
\end{rmk}

The quantity 
\begin{equation}\label{eq:tbp_MI}
    I=\sum_{i=0}^{2} m_i \| \qq_i \|^2, 
\end{equation}
is called \textbf{moment of inertia}. It is a crucial quantity to address standard qualitative 
properties of the general many body problem. The variation of the moment of inertia can 
be studied via the \textbf{Lagrange-Jacobi} identity: 
\begin{equation}\label{eq:LJI}
    \ddot{I}=4\mathcal{H}+2U.
\end{equation}
Standard global qualitative behavior follows from \eqref{eq:LJI}, some results are: 
\begin{itemize}
    \item If $\mathcal{H}\geq 0$, $I$ is a strictly convex function (i.e. with a single minimum) and therefore the mutual distances of the bodies tend to infinity, therefore, 
    no bounded solutions are possible. In the limit case when $\mathcal{H}=0$, 
    the escape is parabolic, namely $I(t)\sim I_{0}+I_{1}t^{2/3}$.
    
    \item The case $\mathcal{H} < 0$ is the only case that allows bounded mutual distance 
    and, therefore, bounded solutions (notice that the converse is not true, there are unbounded 
    trajectories with negative energy \cite{chazy1922, chazy1929, chazy1932}). Moreover, solutions in which all bodies go 
    to infinity simultaneously are also prevented. In fact, from \eqref{eq:LJI}, it 
    follows 
    \begin{equation}\label{eq:tbp_GMA}
        \inf_{i,j} \qq_{i,j} \leq 2a \bydef - \frac{\mathcal{G}}{\mathcal{H}} \sum_{j\neq i} m_{i} m_{j}  .
    \end{equation}
     The quantity $2a$ is called \textbf{generalized major-axis}.
    
    \item Triple collisions occur when $I \to 0$. Moreover, if a triple collision 
    takes place, the angular momentum $\LL$, defined as \eqref{eq:tbp_am}, vanishes. This last statement is commonly 
    known as \textbf{Sundman's Theorem}. 
\end{itemize}

Summarizing, the planar TBP is a Hamiltonian system with six degrees of freedom and, hence, it  has 
a phase space of twelve dimensions described by a system of twelve ordinary differential 
equations. It is a well established fact that constants of motion can be used to remove some 
degrees of freedom from the original equation. 
The conservation of the center of mass \eqref{eq:tbp_cm} can be used 
to isolate the positions of one of the bodies (for instance $B_0$) with respect to the remaining two,
$$ \qq_0 =-\frac{m_1}{m_0} \qq_1 - \frac{m_2}{m_0} \qq_2, ~~~\pp_{0}=-(\pp_{1}+\pp_{2}).$$
Notice that the momenta are isolated from the conservation of linear momentum, another 
first integral that can be obtained deriving, with respect to the time, equation \eqref{eq:tbp_cm}.
Thus, a degree of freedom (or equivalently, two equations) can be removed. An additional 
degree of freedom can be removed by 
using the scalar equations of the energy \eqref{eq:tbp_ham} and the angular 
momentum \eqref{eq:tbp_am}. Therefore, the TBP can be reduced to a Hamiltonian system with 
three degrees of freedom (or six differential equations). 
\begin{rmk}[Symmetries of the TBP]\label{rmk:sym_TBP} 
A number of symmetries hold for the TBP. Given a solution of $(\qq , \pp)$  of \eqref{eq:tbp},
the following operations on $\qq$ produce a new solution $\tilde{\qq}$:  
\begin{enumerate}[i)]
    \item \label{it:sym1} $\tilde{\qq}=\qq+ \ve{a}+ \ve{b} t$, $\ve{a}, \ve{b} \in \mathbb{R}^{6}$, 
    \item \label{it:sym2} $\tilde{\qq}=\qq(t+T)$, $T\in \mathbb{R}$, 
    \item \label{it:sym3} $\tilde{\qq}_i = ( x_i , -y_i)$, $i=0,1,2$, 
    \item \label{it:sym4} $\tilde{\qq}(t)=\qq(2T-t)$, $T\in \mathbb{R}$,
    \item \label{it:sym5} $\tilde{\qq}(t)=\left (\mathcal{R}_{\alpha} \qq_0 (t), \mathcal{R}_{\alpha} \qq_1 (t), \mathcal{R}_{\alpha} \qq_2 (t) \right )$, $\alpha \in \mathbb{T}$ and $\mathcal{R}_{\alpha}$ 
    the $2\times 2$ rotation matrix of angle $\alpha$,
    \item \label{it:sym6} $\tilde{\qq}(t)= \kappa^{2} \qq ( t / \kappa^{3})$, $\kappa \in \mathbb{R}$. \qedhere
\end{enumerate}
\end{rmk}
The last symmetry implies that the equations of motion of the TBP are scale-invariant. Meaning
that the units of time and length are tied. Typically one takes the gravitational
constant $\mathcal{G}$ to be the unit. 
Also one can fix the sum of all the masses to be one. Then, there 
is freedom to pick the units of length or units of time. In the present work we will take 
advantage to fix the position of $B_1$ at time $0$. This will be clarified later, once 
the right context is set.

\subsection{Relative Periodic Orbits}\label{sec:pmot}
The study of periodic orbits is of paramount importance. This was already highlighted by 
Poincar\'e as he proposed them as the starting point for the systematic study of the 
three-body problem. In this context, there are two classes of periodic orbits: \textbf{Absolute
Periodic Orbits (APO)} and \textbf{Relative Periodic Orbits (RPO)}. 

Absolute periodic solutions (or just periodic orbits) to the TBP are those such that the 
positions and velocities of all the bodies return to their initial values after a certain time, 
the period. In the case of RPO, the bodies do not necessarily 
return to their absolute initial position, but their relative positions are periodic 
functions. More commonly, RPO are said to be periodic on a rotating 
frame, i.e. if $\gamma$ denotes a solution it is a relative periodic orbit if 
there exist a block-diagonal 
rotation $\ve{R}_{\theta}$ of angle $\theta$ and a value $T>0$ such that 
$\gamma(t+T)=\ve{R}_{\theta}\gamma(t)$. RPO are, in fact, APO 
if $\theta/(2\pi) \in \mathbb{Q}$ and quasi-periodic if not 
(see \cite[p. 41]{Montgomery2024} for a discussion). Notice that, in the literature,  it 
is usual for the authors to not distinguish between absolute and relative periodic orbits 
as knowing a single period is enough to determine the solution for all time. A standard 
mechanism for proving the existence of RPO is to apply the mirror theorem in two 
consecutive time-symmetries, see \cite{roy1955, shibayama2011} (this will be 
used later). 

\smallskip

APO are isolated in the three-body problems while RPO lie in one-parameter 
families \cite{henon1974}. The celebrated Broucke-Hadjidemetriou-H\'enon 
family is a continuation onto the 
planar space from the Schubart orbit \cite{schubart1956, broucke1975B, hadjidemetriou1975A,henon1976}. RPO families can be obtained also from 
continuation of families of periodic orbits of the CRTBP for positive 
(yet small enough) values of the massless particle (see \cite{hadjidemetriou1975B, bozis1976B, delibaltas1976, message1980}). Also, remarkable examples of APO are the well known choreographies, such 
as the celebrated figure eight, presented in \cite{chenciner2000} and computed numerically 
in \cite{Simo02a}. For a more general perspective on the subject, we 
reefer the interested reader to the recent work \cite{li2025} and references 
therein. 

When considering the study of the planetary solutions to the TBP (i.e. situations 
in which one of the masses is much larger than the other two), Poincar\'e studied 
RPO and categorized them into three different species (or kinds): 
\begin{enumerate}
\renewcommand{\theenumi}{\arabic{enumi}}
\renewcommand{\labelenumi}{\theenumi.)}
    \item Orbits of \textbf{first species} are orbits which have no inclinations 
    and very small eccentricities.
    \item Orbits of \textbf{second species} are orbits that are planar but have 
    large eccentricities. 
    \item Orbits of \textbf{third species} are orbits with small eccentricities 
    but with large mutual inclination. 
\end{enumerate}

Orbits of all the three kinds can be analytically continued from Keplerian orbits 
when the interaction of two of the bodies is ignored: The three-body problem is written as 
two two-body problems coupled by the (small) interaction of two bodies in suitable coordinates 
and  the Implicit Function Theorem is applied to obtain the existence of the periodic orbits 
(see \cite[p. 97]{Poincare92}).  We will see in Section~\ref{sec:AC} that, moreover, 
orbits of first kind can be also numerically continued to orbits of second kind. 

\medskip

A motivation for this work is 
to produce coherent restricted four body problems. In this context, orbits of first kind 
are called Quasi-Bicircular. In previous works, those solutions have been used to describe 
the motion of two planets around a star (for instance, Jupiter and Saturn revolving around 
the Sun in \cite{GabernJ01, Gabern03}) and also to describe the motion of a moon, 
a planet and a star (e.g. the Earth-Moon-Sun system \cite{AndreuS00a, Andreu02}). 
These configurations of orbits of the first kind are usually called planetary solutions and 
lunar solutions respectively. 

Acknowledging the contributions of their pioneering works, 
we shall refer to the planetary solutions as \textbf{Poincar\'e type} and to the lunar
solutions as \textbf{Hill type}, as they were first studied 
in \cite{Hill1878a, Hill1878b, Hill1878c}. Both, Poincar\'e and Hill solutions, are RPO whose 
angle $\theta$ corresponds to the angle between two lines that contain the bodies at two 
distinct times, see Figure~\ref{fig:qbs}. This is a characteristic that can be exploited to 
compute Poincar\'e, Hill and other type of RPO as we will see in Sections \ref{sec:NumComp} and
\ref{sec:AC}. In other texts, Poincar\'e solutions are called \emph{planetary solutions} and Hill solutions
may be called \emph{lunar solutions}. 

\begin{figure}[ht]
\centering
    \begin{subfigure}{0.45\linewidth}
        \includegraphics[width=\linewidth]{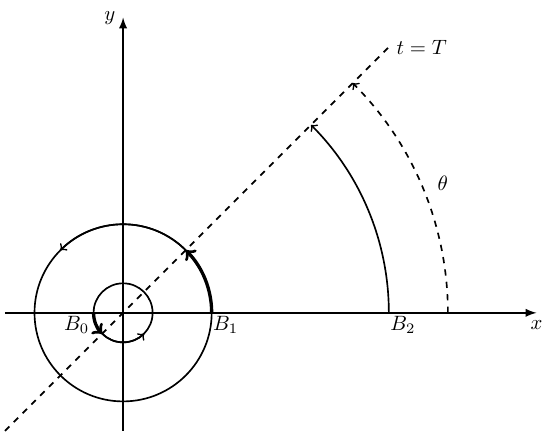}
        \caption{Poincar\'e type}
        \label{fig:qbs_1}
    \end{subfigure} \quad
    \begin{subfigure}{0.45\linewidth}
        \includegraphics[width=\linewidth]{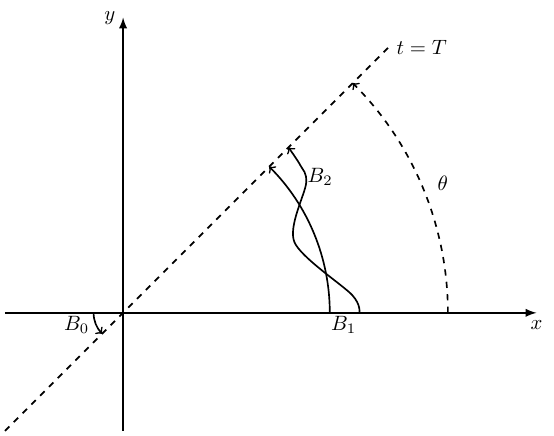}
        \caption{Hill type}
        \label{fig:qbs_2}
    \end{subfigure}
\caption{Quasi-bicircular solutions (solutions of first kind) of the planar Three-Body Problem. 
Panel \eqref{fig:qbs_1}: A Poincar\'e type solution in which two planets revolve around a star. 
Panel \eqref{fig:qbs_2}: A Hill type solution with a natural satellite revolving around a 
planet while this one revolves around a star. See text for more details.}\label{fig:qbs}
\end{figure}

\subsection{Stability of RPO}
In this section we handle stability of RPO. In classical theory of Ordinary Differential 
Equations, the stability of a periodic orbit is given by the monodromy matrix i.e. 
the first variational flow evaluated at the period. This approach can be applied 
directly to APO of the TBP but the case of RPO is slightly different. As RPO are periodic 
only in a rotating reference, this rotation is to be incorporated to the monodromy matrix. 
Recall, RPO are, generically, quasi-periodic orbits. 

Given an RPO $\gamma$ with angle $\theta$ and period $T$, we define the \textbf{rotated monodromy matrix}
as 
\begin{equation}\label{eq:romo}
\mathcal{M}=\ve{R}_{-\theta} \operatorname{D} \phi_{T}(\gamma(0)). 
\end{equation}
Here, $\ve{R}_{\theta}=\operatorname{diag}(\mathcal{R}_{\theta}, \dotsc, \mathcal{R}_{\theta}) \in \mathbb{M}_{12}(\mathbb{R})$.

\begin{lem}[Eigenvalues related to the RPO angle]\label{lem:JB}
Let $\gamma$ be an RPO of the system \eqref{eq:tbp}, its associated rotated monodromy matrix defined as in \eqref{eq:romo} has
two $2\times2$ Jordan blocks whose eigenvalues are $e^{\pm \I \theta}$. 
\end{lem}
\begin{proof}
    As it has been stated in previous section, one can use the conservation of the 
    center of mass and the conservation of the linear momentum to write 
    the vector $(\qq_0 , \pp_0 ) = - \left(\frac{m_{1}}{m_0} \qq_1 + \frac{m_{2}}{m_0} \qq_2, \pp_1 + \pp_2 \right)$.
    The corresponding $4\times 4$ block of the Jacobian is
    $$
    \begin{pmatrix}
        0_{2}& \frac{1}{m_0} \idmat_2\\
        0_2 & 0_2
    \end{pmatrix}.
    $$
    Therefore, integrating over $T$ units of time and rotating, 
    the block of the rotated monodromy matrix is: 
    $$
    \begin{pmatrix}
        \mathcal{R}_{-\theta}& \frac{T}{m_0}  \mathcal{R}_{-\theta}\\
        0_2 & \mathcal{R}_{-\theta} 
    \end{pmatrix}.
    $$
    Here, it is straightforward that both eigenvalues $e^{\pm \I \theta}$ have algebraic multiplicity two. The upper-right block makes their geometric multiplicity equal to 1. 
\end{proof}

\begin{lem}[Multiplicity of the eigenvalue 1]\label{lem:EV1}
    Let $\gamma$ be an RPO of the system \eqref{eq:tbp} but reduced by the conservation 
    laws of the center of mass and the linear momentum. The 
    associated rotated monodromy matrix $\mathcal{M}_{\rm red} \in \mathbb{M}_{8}(\mathbb{R})$ has 1 as an eigenvalue with multiplicity at least $4$. 
\end{lem}
\begin{proof}
    Let $\ve{z}=(\qq,\pp) \in \mathbb{R}^{8}$ a point on the reduced phase space. 
    To avoid cumbersome notation we will write the reduced vector field as  
    $\dot{\zz} = \JJ \nabla \mathcal{H}(\zz)$. Here, $\JJ$ is 
    the standard $8\times8$ symplectic matrix and $\mathcal{H}$ is the reduced hamiltonian. 
    Let us also denote $\phi_{t}(\zz)$ the flow map of the reduced system. 
    If $\zz= \gamma(0)$ then, there exist numbers $T$ and $\theta$
    such that $\phi_{T}(\zz)= \ve{R}_{\theta} \zz $. The rotated monodromy matrix is 
    $\mathcal{M}= \ve{R}_{-\theta} \dop \phi_{T}(\zz)$. Again, $\ve{R}_{\theta}$ stands 
    for the operator in the reduced space.
    Let us construct, explicitly, the four pairs eigenvector-eigenvalue.
    \begin{itemize}
        \item \textbf{The vector field:} As the system is autonomous, the vector field is an eigenvector of the 
        linearized system (this is a classical common fact that comes from direct 
        application of the Chain Rule. Therefore, 
        $$\dop \phi_T \JJ \mathcal{H}\nabla (\zz) = \JJ \nabla \mathcal{H}(\phi_{T}(\zz)).$$
        Moreover, the vector field commutes with the block-diagonal rotation operator, 
        $$\JJ \nabla \mathcal{H}(\ve{R}_{\theta} \zz) = \ve{R}_{\theta} \JJ \nabla \mathcal{H}(\zz).$$
        It follows, combining those expressions, that 
        $$\mathcal{M} \JJ \nabla \mathcal{H}(\zz)=\JJ \nabla \mathcal{H}(\zz). $$
    
        \item \textbf{Energy and Angular Momentum:} Let us assume that we are given a conserved quantity $\II$ which 
        is invariant under a rotation $\ve{R}_{\theta}$, i.e., 
        \begin{equation}\label{eq:InvRot}
        \II ( \ve{R}_{\theta} \zz) = \II(\zz).
        \end{equation}
        
        Notice that this is the case for both, the energy $\mathcal{H}$ and 
        the angular momentum $\LL$. Because those are conserved quantities, it 
        holds that 
        $$\II( \phi_{T}(\zz) ) = \II (\zz). $$
        By differentiating with respect to $\zz$, we get 
        \begin{equation}\label{eq:lev}
        \nabla \II ( \phi_{T} (\zz) )^{\top} \dop \phi_{T} (\zz) = \nabla \II (\zz)^{\top}.
        \end{equation}
        And using that $\zz$ lies in an RPO, so $\phi_{T}(\zz) = \ve{R}_{\theta} \zz$ and the 
        rotational invariance of the conserved quantity $\II$ it follows that $\nabla \II$
        is a left eigenvector of the matrix $\dop \phi_{T}$. Differentiating \eqref{eq:InvRot}
        with respect to $\zz$, we get: 
        $$\nabla \II ( \ve{R}_{\theta}(\zz))^{\top} \ve{R}_{\theta} = \nabla \II(\zz)^{\top}.$$
        And composing with $\ve{R}_{-\theta}$, we get 
        $$\nabla \II (\phi_T (\zz))^\top=\nabla \II (\zz)^{\top} \ve{R}_{-\theta}.$$
        Combining the last expression with \eqref{eq:lev} we get that $\nabla \II$ is also 
        a left eigenvector of the rotated monodromy matrix $\mathcal{M}$. 
    
        \item \textbf{Rotation invariance:} Let us consider the $2\times 2$ generator of the rotation 
        group 
        $$\mathcal{K}=\begin{pmatrix}
            0 & -1\\
            1 & 0
        \end{pmatrix},$$
        and define its extension to the reduced phase space $\ve{K}=\operatorname{diag}(\mathcal{K},\dots, \mathcal{K}) \in \mathbb{M}_{8}(\mathbb{R})$. Notice that, given any rotation 
        $\ve{R}_\alpha$, it commutes with the generator, namely $\ve{K}\ve{R}_{\alpha}= \ve{R}_{\alpha} \ve{K}$. Moreover, if $\ve{R}_{\alpha}$ is an arbitrary rotation, by 
        the rotational symmetry of the flow, 
        $$\phi_{T}( \ve{R}_{\alpha} \zz ) = \ve{R}_{\alpha} \phi_T (\zz),$$
        Differentiating the last equation with respect to $\alpha$, one gets
        $$\dop \phi_T \ve{K} \zz =\ve{K} \phi_T (\zz),$$
        and, the RPO condition $\phi_{T}(\zz)=\ve{R}_{\theta} \zz$ leads to 
        $$\dop \phi_T \ve{K} \zz =\ve{K} \ve{R}_{\theta} (\zz).$$
        Now, we can apply the commutativity between $\ve{K}$ and rotations and compose with $\ve{R}_{-\theta}$ on the left. It follows that 
        $$\mathcal{M} \ve{K} \zz = \ve{K} \zz,$$
        so the vector $\ve{K}\zz$ is also an eigenvector of eigenvalue 1.  \qedhere
    \end{itemize}
\end{proof}
As a consequence of Lemma~\ref{lem:EV1}, the rotated monodromy matrix in the full phase space has eight  
eigenvalues and four of them are equal to one. The two eigenvalues related to the conservation of mass
and linear momentum (See Lemma~\ref{lem:JB}) do not appear in the reduced phase space of 8 dimensions.
Usually, this reduction is straightforward and the numerical implementations start with 8 
differential equations. While the other conserved quantities can be used to produce further 
reductions (and study the equations of motions in the shape sphere), the conserved quantities 
can be used to deflate the $8\times 8$ matrix into a $4\times4$ matrix with no eigenvalues equal to 
1. This is especially useful from the numerical point of view, as the Jordan blocks associated
to those eigenvalues cause a remarkable accumulation of error (see, for instance \cite{GolubV96}). 

The deflation algorithm is based on using the four eigenvectors from Lemma~\ref{lem:EV1}
to obtain the span of the eigenspace related to the eigenvalue 1. Then, via Singular 
Value Decomposition (SVD), we obtain a generator of its orthogonal subspace and, 
finally, project the $8\times 8$ matrix onto the orthogonal subspace to get the fully 
reduced matrix with no trivial eigenvalues. This procedure is detailed in algorithm~\ref{alg:symmetry_deflation}.

\begin{algorithm}
\caption{Symmetry Deflation via Subspace Projection}\label{alg:symmetry_deflation}
\begin{algorithmic}[1]
\REQUIRE Initial state vector $\mathbf{z} \in \mathbb{R}^8$, Rotated Monodromy Matrix $\mathcal{M} \in \mathbb{M}_{8}(\mathbb{R})$
\ENSURE Non-trivial eigenvalues $\{\lambda_5, \lambda_6, \lambda_7, \lambda_8\}$

\STATE \COMMENT{Evaluate symmetry and first integral configurations at the initial state}
\STATE $\mathbf{v}_1 \leftarrow \mathbf{J}\mathcal{H}(\mathbf{z})$ \hfill \COMMENT{Vector-field}
\STATE $\mathbf{v}_2 \leftarrow \mathbf{K}\mathbf{z}$ \hfill \COMMENT{Rotation generator}
\STATE $\mathbf{w}_1 \leftarrow \nabla \mathcal{H}(\mathbf{z})$ \hfill \COMMENT{Gradient of the Hamiltonian}
\STATE $\mathbf{w}_2 \leftarrow \nabla \mathbf{L}(\mathbf{z})$ \hfill \COMMENT{Gradient of the Angular Momentum}

\STATE $W \leftarrow \begin{bmatrix} \mathbf{v}_1 & \mathbf{v}_2 & \mathbf{w}_1 & \mathbf{w}_2 \end{bmatrix} \in \mathbb{R}^{8 \times 4}$ \hfill \COMMENT{Assemble the unconstrained generalized eigenspace matrix}

\STATE $V \leftarrow \text{SVD}(W) \quad \text{where } V \in \mathbb{R}^{8 \times 4},\ V^{\top}V=\idmat _4$ \hfill 
\COMMENT{Extract the orthogonal complement via SVD}

\STATE $\mathcal{M}_{\text{red}} \leftarrow V^{\top} \mathcal{M} V \in \mathbb{M}_{4}(\mathbb{R})$ \hfill 
\COMMENT{Perform the subspace reduction projection}

\STATE $\{\lambda_5, \lambda_6, \lambda_7, \lambda_8\} \leftarrow \text{eig}(\mathcal{M}_{\text{red}})$ 
\hfill \COMMENT{Compute the non-trivial roots}

\RETURN $\{\lambda_5, \lambda_6, \lambda_7, \lambda_8\}$
\end{algorithmic}
\end{algorithm}

The stability type of each solution is determined by the location of its non-trivial roots. Owing to the Hamiltonian structure of the system, if $\lambda$ is an eigenvalue, then $\lambda^{-1}$ also is an eigenvalue. Consequently, the four computed eigenvalues occur in reciprocal pairs $(\lambda,\lambda^{-1})$. When all four eigenvalues lie on the unit circle in the complex plane, they form two conjugate pairs of the form $(\exp(\I \alpha), \exp (-\I\alpha))$, being $\I$ the imaginary unit, and then the solution is \emph{Totally Elliptic (TE)}. If one pair belongs to the unit circle and the other consists of real with moduli different from one, the solution is partially hyperbolic or \emph{Elliptic-Hyperbolic (EH)}. If all four eigenvalues lie off of the unit circle, then solution is classified as \emph{Totally Hyperbolic (TH)} when they all are real with moduli different from one, or as \emph{Complex Saddle (CS)} if we find a quartet of eigenvalues of $\lambda=(\rho \exp(\I\alpha), \rho \exp (-\I\alpha), \rho^{-1} \exp(\I\alpha), \rho^{-1} \exp (-\I\alpha))$ with $\rho>1$.

\section{Numerical computation of RPO}\label{sec:NumComp}
In this section we discuss numerical continuation of RPO. Our strategy 
consists of looking for two consecutive alignments of the bodies (i.e. 
syzygy conditions). In section~\ref{sec:syzygy} we describe the algebraic 
conditions we use to identify syzygies and in section~\ref{sec:continuation} we 
describe a scheme to numerically continue RPO.

\subsection{The syzygy condition}\label{sec:syzygy}
The term ``syzygy'', in astronomy, refers to a configuration of three or more  
bodies in which all of them lie in a straight-line for a certain time $t_0$. In 
more theoretically oriented works on the abundance of syzygy as a general feature 
of the trajectories of the three body problem. For instance \cite{diacu1989, diacu1992, montgomery2002}.
syzygies represent a key tool to designate a symbolic dynamics and proof chaotic 
motion in the TBP. However, in this work we are interested in them as a condition 
to build a Poincar\'e section. 

In \cite{diacu1989}, 
the following definition is given for a system of $N\geq 3$ bodies in the plane
\begin{dfn}
    A solution of the Equations~\eqref{eq:tbp} is called syzygy solution if there exists a date $t_0$ and a line $\Lambda_0$ which contains all $N$ bodies at this date.
\end{dfn}
In this work, 
it is proven that given $(\qq, \pp)$ a solution of \eqref{eq:tbp} is in   
a syzygy configuration for $t_0$ if and only if, 
\begin{equation}\label{eq:syzygy_condition}
    (f\circ \qq)(t_0)=0, 
\end{equation}
where, 
$$f(\qq) = \qq_0 \times \qq_1 + \qq_1 \times \qq_2 + \qq_2 \times \qq_0,$$
and the $2$-dimensional cross product is defined as in Remark~\ref{rmk:2dcp}. The latter condition can be also written in terms of coordinates, namely, 
$$
f(\qq) = x_0 y_1 - x_1 y_0 + x_1 y_2 - x_2 y_1 + x_2 y_0 - x_0 y_2=(x_1 -x_0)(y_2-y_0)-(x_2 -x_0)(y_1 -y_0). 
$$
Our computation of RPO is based on spatial Poincar\'e section. Consider the sets
\begin{equation*}
        \Sigma \bydef \{(\qq,\pp)\colon f(\qq) =0 \}, \qquad
        \Sigma_\parallel \bydef \Bigl\{(\qq,\pp)\colon \frac{d}{dt}f(\qq) =0 \Bigr\}.
\end{equation*}

A definition of a generalized syzygy solution for three bodies whose positions or velocities are collinear at some finite time, has been also addressed in recent work \cite{Tsygvintsev2023}. However, the kind of solutions we are interested in imply alignments on both, positions and velocities, at the same time.

Syzygy condition \eqref{eq:syzygy_condition} determines a section $\Sigma$ in the phase space of the TBP which is, moreover, transversal to the flow (that is, the syzygy condition
determines a Poincar\'e section) away from a set of degenerate collinear configurations like $\Sigma _\parallel$. A geometrical interpretation of $\Sigma$ is the signed 
area of the parallelogram formed by the vectors $\qq_{1,0}$ and $\qq_{2,0}$, i.e. 
the vector $(\qq, \pp)$ will be in syzygy if, and only if, 
$$\operatorname{det}(\qq_{1,0}, \qq_{2,0})=0.$$
Then, condition \eqref{eq:syzygy_condition} is uniquely referred to relative positions of the three bodies of the system. 

The syzygy condition \eqref{eq:syzygy_condition} can be also simplified as the following result shows
\begin{lem}\label{lem:simplesyzygy}
    The condition $\{x_{1} y_{2} = x_{2} y_{1}\}$ is also a characterization of the syzygy condition $\Sigma$.
\end{lem}
\begin{proof}
    The condition can be rewritten as $\operatorname{det}(\qq_1 , \qq_2)=0$. 
    Therefore, by the conservation of the center of masses, we have: 
    
    \begin{align*}
    \det(\qq_{0,1},\qq_{0,2})&= \det( \qq_1 - \qq_0 , \qq_2 - \qq_0)\\
    &=\det\left ( \left ( \frac{m_0 + m_1}{m_0} \right ) \qq_1 + \frac{m_2}{m_0}\qq_2 , \frac{m_1}{m_0}\qq_1 +   \left ( \frac{m_0 + m_2}{m_0} \right ) \qq_2\right)\\
    &=\left(\frac{m_0+m_1 +m_2}{m_{0}}\right) \det(\qq_1, \qq_2). 
    \qedhere
    \end{align*}    
\end{proof}

The characterization in Lemma~\ref{lem:simplesyzygy} is the one used in \cite{broucke1975B}. To characterize the size $\Sigma _\parallel$ we provide a geometric result consisting in each body's instantenous velocity has two components: a motion along a commong alignment with individual velocity speed; and a shared angular velocity $\omega$. Consequently, to first order, both bodies remain on the same rotating line. The trajectory is therefore tangent to the syzygy hypersurface.
\begin{pro}[Non-transversal set]\label{pro.nontransverse}
Let $\uu \in \R ^2$ be a unit vector along the common line of the three bodies. Then $(\qq, \pp) \in \Sigma_{\parallel}$ if, and only if, there exist reals $b _1, b _2$, and $\omega$ such that 
\[
\dot{\qq}_{i} = b_i \uu +\omega \mathcal{K} \qq_i \qquad \text{for }i=1,2
\]
\end{pro}
\begin{proof}
Since $(\qq,\pp) \in \Sigma$, then $\qq _1, \qq _2$ are collinear. Since triple collision is excluded, $\qq _i = a _i \uu$ for some $\as \bydef (a _1, a _2) \ne \0$.
We decompose $\dot \qq _i = b _i \uu + c _i \mathcal{K} \uu$. Let $h(\qq) \bydef \det (\qq _1, \qq _2)$. We have that 
\[
\frac{d}{dt} h(\qq) = \det (\qq _1, \dot \qq _2) + \det(\dot \qq _1, \qq _2) = a _1 c _2 - a _2 c _1.
\]
Since $\as \ne \0$, $\frac{d}{dt} h(\qq) = 0$ implies $c _i = \omega a _i$ for some $\omega \in \R$. Thus, $\dot{\qq}_{i} = b_i \uu +\omega \mathcal{K} \qq_i $ for $i = 1,2$.
\end{proof}

\begin{rmk}\label{rmk:Euler}
    Euler's collinear solutions are special solutions to the TBP in which all three bodies remain aligned along a straight line, maintaining fixed ratios between their distances. These solutions are examples of trajectories entirely contained in the 
    non-transversal set $\Sigma_{\parallel}$. 
\end{rmk}

\subsection{A continuation scheme}\label{sec:continuation}
The syzygy condition $\Sigma$ appearing in Lemma~\ref{lem:simplesyzygy} a Poincar\'e section, 
once the set $\Sigma_{\parallel}$ is removed. 
$\Sigma\setminus \Sigma_{\parallel}$ is the helpful section to compute RPO. The idea is 
to integrate a suitable initial condition (where the three bodies are aligned) 
until the next syzygy takes place (as a first return map to $\Sigma$).
Moreover, additional conditions can be imposed 
to ensure that the resulting orbit is an RPO and transverse. 
Let $(\qq, \pp)$ be an initial 
condition and $(\QQ, \PP)=\phi_{T}(\qq, \pp)$ where $T$ is the time to the 
next alignment so $f(\QQ)=0$ and define the following objective function 

\begin{equation}\label{eq:OF}
    G \colon \R^{8} \rightarrow \R^6 \qquad
    \zz \mapsto \bigl ( G_{1}(\zz), G_{2}(\zz), G_{3}(\zz), G_{4}(\zz), G_5(\zz), G_6(\zz) \bigr),
\end{equation}
where $\zz=(\qq, \pp)$, and

\vspace{5pt}

\noindent
\begin{minipage}[c]{0.48\textwidth}
\begin{enumerate}
\renewcommand{\theenumi}{\roman{enumi}}
\renewcommand{\labelenumi}{\theenumi)}
    \item \label{it:c1} $G_1(\zz)=\|\QQ_1\|^2-\|\qq_1\|^2$,
    \item \label{it:c2} $G_2(\zz)=\|\QQ_2\|^2-\|\qq_2\|^2$,
    \item \label{it:c3} $G_3(\zz)=\|\PP_1\|^2-\|\pp_1\|^2$,
\end{enumerate}
\end{minipage}\hfill
\begin{minipage}[c]{0.48\textwidth}
\begin{enumerate}
\setcounter{enumi}{3}
\renewcommand{\theenumi}{\roman{enumi}}
\renewcommand{\labelenumi}{\theenumi)}
    \item \label{it:c4} $G_4(\zz)=\|\PP_2\|^2-\|\pp_2\|^2$,
    \item \label{it:c5} $G_5(\zz)=\QQ_1\cdot\PP_1$,
    \item \label{it:c6} $G_6(\zz)=\QQ_2\cdot\PP_2$.
\end{enumerate}
\end{minipage}
\vspace{5pt}

Here, $\|\cdot\|$ is the standard euclidean norm and $\cdot$ the standard 
inner product. Equations \ref{it:c1}) and \ref{it:c2}) measure the discrepancy of 
the distance to the barycenter between the initial and the final alignment. 
Equations \ref{it:c3}) and \ref{it:c4}) measure the discrepancies of the individual 
momenta of the bodies and \ref{it:c5}) and \ref{it:c6}) measure orthogonality of 
the velocity with respect to the positions.

\smallskip

To ensure that a zero of the function $G$ is, in fact, an RPO we have to rely 
on the Mirror Theorem (see \cite{roy1955})
. Let us recall its setting. We define 
$$S=\begin{pmatrix}
    1 &  0 \\
    0 & -1
\end{pmatrix},$$
and define the operator 
\begin{equation*}
\mathcal{S}\colon  \mathbb{R}^{8} \to \mathbb{R}^{8}\qquad
(\qq, \pp) \mapsto\left (S\qq_1 , S\qq_2 , - S\pp_1, -S\pp_2 \right ). 
\end{equation*}
\begin{lem}\label{lem:invo}
    The following statements hold 
    \begin{enumerate}
        \item The operator $\mathcal{S}$ is an involution $\mathcal{S}^2 = \idmat_{8}$,
        \item $\SSS \circ \varphi_t = \varphi_{-t} \circ \SSS $, 
        \item $\SSS \circ \ve{R}_{\alpha}= \ve{R}_{-\alpha} \circ \SSS$.
        \item $\operatorname{Fix}(\SSS)=\{ (\qq, \pp) \in \Sigma ~:~ y_1 = \dot{x}_1 = y_2 = \dot{x}_2=0\}$. 
    \end{enumerate}
\end{lem}
\begin{proof}
    Points 1) 3), and 4) are standard verifications and point 2) comes from symmetries $iii)$ 
    and $iv)$ from Remark~\ref{rmk:sym_TBP}.
\end{proof}

\begin{lem}\label{lem:minimizer}
    Let $\zz=(\qq, \pp)\in \operatorname{Fix}(\SSS)$ in the transversal set 
    and $\ve{Z}=(\QQ, \PP)=\phi_{T}(\qq,\pp)$ 
    be two consecutive syzygies with alignment time $T$ and angle $\theta$. Suppose that $G(\zz)=0$, then $\zz$ lies in an RPO, namely $\phi_{2T} (\zz) = \ve{R}_{2\theta} \zz$. 
\end{lem}
\begin{proof}
    Let $\ve{z}\in \operatorname{Fix}(\SSS)\cap \left (\Sigma \setminus \Sigma_{\parallel} \right )$ and $\ve{z}$. Conditions $G_5(\ve{z})=G_6(\ve{z})=0$ impose the arrival momenta $\PP$ to be 
    transversal to the second syzygy line i.e. there exist $\theta$ such that 
    $\ve{R}_{-\theta} \ve{Z}\in \operatorname{Fix}(\SSS)$.  Now set, $\alpha=\theta/T$ and 
    introduce $\psi_{t}= \ve{R}_{-\alpha t} \varphi_{t}$. Due to property 3 from Lemma~\ref{lem:invo}, 
    $\psi_{t}$ is a diffeomorphism which is reversible with respect to $\SSS$. Notice also that
    \begin{equation*}
        \psi_{T} (\ve{z}) = \ve{R}_{-\alpha T} \varphi_{T} (\ve{z}) =\ve{R}_{-\theta} \varphi_{T}(\ve{z}) \in \operatorname{Fix}(\SSS).
    \end{equation*}
    Invoking the mirror theorem, we get $\psi_{2T}(\ve{z})= \ve{z}$, that is 
    $ \varphi_{2T}(\ve{z})= \ve{R}_{2\theta} \ve{z}$.
\end{proof}

\begin{dfn}
 A Syzygy Poincar\'e section is a $\Sigma _0 \subset \operatorname{Fix}(\SSS)\cap \left (\Sigma \setminus \Sigma_{\parallel} \right )$ for which each element has at least a flying back time to the section.
\end{dfn}

\begin{rmk}[Overdetermined system]\label{rmk:ods}
    If $\zz\in \Sigma_0$, then
    $\dot{x}_1=\dot{x}_2=0$. Thus, $G_5 (\zz)=0$ and $G_6 (\zz)= 0$. Time 
    reversibility implies $G_i (\zz)=0$ for $i\in\{1,2,3,4\}$. Hence, restricted to $\Sigma _0$, $\operatorname{dim} \bigl ( \operatorname{ker} \dop G(\zz) \bigr )=4$.
    Therefore, the purpose the extra conditions is to prevent solutions to go 
    to collision or to converge into the non-transversal set. 
\end{rmk}

Finding an RPO that is a zero of \eqref{eq:OF} can be formulated 
as a Fixed-Point problem if the second alignment is rotated back 
to $\Sigma_0$. Indeed, let us consider 
\begin{equation}\label{eq:SPM}
    \mathcal{P} \colon \Sigma_0 \rightarrow \Sigma_0 \qquad
    \zz \mapsto \ve{R}_{-\theta(\zz)} \phi_{T(\zz)}(\zz) ,
\end{equation}
where $T(\zz)$ is the travel time between the two syzygies and $\theta(\zz)$ the 
angle of the second alignment with respect to $\Sigma_0$. 

From symmetry \ref{it:sym6}) in Remark~\ref{rmk:2dcp} one can always consider 
an explicit time-rescaling so $x_1=1$ at the initial alignment. Thus, one extra 
variable can be removed. The previous observations imply that only three variables 
are required to determine the RPO: $x_2$, $\dot{y}_{1}$ and $\dot{y}_{2}$. Having 
an initial condition such that $G_5 ( x_2 , \dot{y}_{1}, \dot{y}_{2})=0$, 
$G_6 ( x_2 , \dot{y}_{1}, \dot{y}_{2})=0$ a standard application of the 
Implicit Function Theorem leads to the existence of a curve 
$$x_{2} \mapsto \left ( \dot{y}_{1} (x_{2}), \dot{y}_{2} (x_2) \right ).$$
This is a one-parameter family of RPO whose parameter is, essentially, the 
distance between bodies $1$ and $2$ (recall that the first body is fixed at 
$x_1$, so $d(B_2, B_1)=x_2 - 1$). In this sense, the closer $x_2$ is to one in 
the family, the larger the interaction of the three bodies is. 

By evaluating initial conditions on the syzygy section $\Sigma_0$ under 
the assumption of perpendicular velocity alignment, the full state 
vector is uniquely determined by the three components $(x_2, \dot{y}_1, \dot{y}_2)\in \mathbb{R}^3$. We can, hence, restrict $G$ to the 
subspace spanned by these three variables. A continuation algorithm for RPO together 
with the computation of their stability is sketched in Algorithm~\ref{alg:rpo_curve}. 
Notice that Remark~\ref{rmk:ods} is handled here by using a Gauss-Newton corrector 
step to avoid further reductions. This is similar to the approach suggested in \cite{wulff2008}.

\begin{algorithm}[h!]
\caption{Pseudo-Arclength Continuation and Stability of the RPO Curve}
\label{alg:rpo_curve}
\begin{algorithmic}[1]
\REQUIRE Initial seed RPO point $\zz^{(0)} \in \mathbb{R}^8$, step-size $\Delta s$, tolerance $\epsilon$, maximum steps $K_{\max}$
\ENSURE Curve of symmetric RPOs $\{\zz^{(k)}\}_{k=1}^K$ and their stability multipliers

\STATE $k \gets 1$
\STATE Compute initial approximated tangent vector $\mathbf{t}_0 \in \mathbb{R}^8$ with $\|\mathbf{t}_0\| = 1$.

\WHILE{$k \le K_{\max}$}

    \STATE \COMMENT{\textbf{Step 1: Predictor Step}}
    \STATE $\zz_{\text{pred}} \gets \zz^{(k-1)} + \Delta s \cdot \mathbf{t}_{k-1}$
    \STATE $\zz \gets \zz_{\text{pred}}$
    \STATE $\text{converged} \gets \text{False}$
    
    \STATE \COMMENT{\textbf{Step 2: Corrector Step (Gauss-Newton)}}
    \WHILE{NOT $\text{converged}$}
        \STATE Integrate flow $\Phi_t(\zz)$ to consecutive syzygy s to get $T(\zz), \theta(\zz)$
        \STATE Evaluate residual vector $G(\zz)  \in \mathbb{R}^6$
        
        \IF{$\|G(\zz)\| < \epsilon$}
            \STATE $\text{converged} \gets \text{True}$
        \ELSE
            \STATE $\Delta \zz \gets - \dop G(\zz)^\dagger G(\zz)$ \hfill\COMMENT{least squares solution, operator $\dagger$ denotes the pseudo-inverse}
            \STATE $\zz \gets \zz + \Delta \zz$
        \ENDIF
    \ENDWHILE
    
    \STATE \COMMENT{\textbf{Step 3: Save Point \& Update Tangent}}
    \STATE $\zz^{(k)} \gets \zz$
    \STATE Compute new tangent $\mathbf{t}_k = \pm (\zz^{k}-\zz^{k-1})/\|\zz^{k}-\zz^{k-1}\|$ such that $\mathbf{t}_k \cdot \mathbf{t}_{k-1} > 0$
    
    \STATE \COMMENT{\textbf{Step 4: Stability Analysis}}
    \STATE \textbf{Call} \text{Algorithm~\ref{alg:symmetry_deflation}}($~\zz^{(k)}, T(\zz^{(k)}), \theta(\zz^{(k)})$)
    
    \STATE $k \gets k + 1$
\ENDWHILE

\RETURN $\{\zz^{(k)}\}_{k=1}^K$
\end{algorithmic}
\end{algorithm}

\subsection{Technical details}\label{sec:tede}
Algorithms \ref{alg:symmetry_deflation} and \ref{alg:rpo_curve} have been implemented in C language 
from scratch. We have used the Taylor \cite{JorbaZ05,GimenoJZ22} package to integrate with a local threshold error of $10^{-16}$ in double-precision arithmetic and to produce the variational equations of the field through automatic differentiation. 

The Newton tolerance for the numerical computation of RPOs is set at $10^{-12}$. LAPACK's routines \cite{Anderson1999}, like \emph{dgeev} for eigenvalues are used, in the 
stability type a tolerance of $10^{-8}$ is used to determine whether an eigenvalue lies off the unit circle in the complex plane and, in that case, whether its imaginary part is non-zero. 



In all the experiments, we used the \texttt{gcc} compiler, version 14.2.0, on a computer running Ubuntu 25.04, equipped with an Intel(R) Core(TM) i5-12400 processor featuring 6 cores and 12 threads.

\section{Dynamical diversity of RPO families in the Three-Body Problem}\label{sec:AC}

The procedure presented in this work allows us to numerically compute families of RPOs to the TBP systematically given a set of masses. The nature of these families is sensitive to the initial relative positions and velocities among the three bodies, however it seems less sensitive to other parameters as the exact values of the masses. The richness of the problem prevents a complete generic analysis, for which in this section we examine a variety of academic examples.

In Section~\ref{sec:Poincare} we present some examples of families of Poincar\'e RPO. These families were firstly numerically computed in the late 70's by \cite{hadjidemetriou1976A} as an analytic continuation of symmetric circular periodic orbits of the CRTBP in a rotating reference frame. Due to the different methodology and the advancement of numerical techniques, we will find some discrepancies with their work among many similarities. Then, Section~\ref{sec:Hill} is devoted to Hill families of RPO, where as mentioned in the Introduction, the smallest body will transit from satellite regime to circumstellar. A last analysis, in Section~\ref{sec:Binary}, is concerned about Binary type of solutions. 

As we will see, many of these families are representative of real systems situations, capturing well-known typical characteristics and astrodynamical phenomena. However, some of the computed solutions of the TBP do not correspond to any known astronomical situation. They are included since they constitute an academic interest as solutions to the Three-Body Problem.

As several dynamical features and analysis steps are shared among the different families of relative periodic orbits, the first case study is presented in greater detail to illustrate the recurring phenomena and general characteristics of the solutions. The following sections will therefore emphasize the particular features of each family, highlighting their specific dynamical behavior while avoiding repetition of aspects already discussed.

\subsection{Poincar\'e solutions}\label{sec:Poincare}
Let's first start studying some cases of the Poincar\'e (or planetary) solutions, see Figure~\ref{fig:qbs_1}. In this kind of solutions, each of the bodies $1$ and $2$, $B_1$ and $B_2$, is expected to behave as in a two-body problem with respect to the body $0$, $B_0$, where the mass of $B_0$ is significantly larger than the masses of the other two bodies, and whereas bodies $1$ and $2$ do not approach. Then, if those two conditions hold, we could study the motion of the smaller bodies as in two decoupled two-body problems; one for the bodies $0$ and $1$, and other one for the bodies $0$ and $2$. In this situation, the possible gravitational interaction between $B_1$ and $B_2$ is not relevant, as if they do not seem to ``see each other''.

Following for example \cite{Pollard66}, for each Kepler (two-body) problem, the eccentricity vector of body $i\in \{1,2\}$ with respect to body $0$ can be defined as:
\begin{equation}
    \ee_i = \cfrac{1}{\mu_i} (\dot{\qq}_{i,0} \times \llm_{i,0}) - \cfrac{\qq_{i,0}}{|\qq_{i,0}|},
\end{equation}
where $\mu_i=G(m_0+m_i)$ is the reduced mass of the Kepler problem for bodies $0$ and $i$, and $\qq_{i,0}$, $\dot{\qq}_{i,0}$ and $\llm_{i,0}$ are the relative position, velocity and angular momentum per unit of mass of body $i$ with respect to body $0$, respectively. 
Since we are working on the planar case, the angular momentum is entirely contained in the vertical axis, perpendicular to the plane of motion. Besides, in the initial condition, at $t=0$, the three bodies are aligned in the $x-$axis, with null $x$-velocity component. Then $\llm_{i,0}=\qq_{i,0} \times \dot{\qq}_{i,0} = [0, 0, x_{i,0} \dot{y}_{i,0} ],$
what results in the eccentricity vector for each of the two Kepler problems to be contained in the $x-$axis:
\begin{equation}
    \ee_i = \Bigg[ \cfrac{x_{i,0} \dot{y}_{i,0}^2}{\mu_i} -1, 0, 0 \Bigg]. \label{eq:ecvec}  
\end{equation}
This vector is known for being a conserved quantity of the Kepler problem and for pointing towards the periapsis of the orbit. Obviously, the keplerian eccentricity orbital element corresponds to $e_i=|\ee_i|$ for each body $i\in \{1,2\}$ orbiting around $B_0$. 

Both, $B_1$ and $B_2$, are placed at the positive $x-$axis at $t=0$. 
Then, if the evaluation of Equation~\eqref{eq:ecvec} in the initial coordinates of a computed solution results in a positive value of the $x$-component of the eccentricity vectors, $e_{x_i}>0$, the bodies would be occupying the periapsis of their respective orbits, so their semimajor axis would be approximated by $a_i=x_i/(1-e_i^2)$. Otherwise, they would be occupying the apoapsis of their respective orbits, and then $a_i=x_i/(1+e_i^2)$. 
Also notice that if the described orbit is circular, $e=0$, the following relation holds 
\begin{equation}
\cfrac{x_{i,0} \dot{y}_{i,0}^2}{\mu_i} = 1.    
\label{eq:rel-x-vy}
\end{equation}

Regarding $m_0\gg m_1,m_2$, then in normalized units $\mu_1\approx\mu_2\approx1$. Besides, for $B_1$, we have $x_{1,0}=x_1-x_0\approx1$. Therefore, $B_1$ will describe circular, or almost circular, orbits where its initial velocity is close to the unity. As $\dot{y}_{1,0}$ gets larger than $1$, the orbit of body $1$ gets more eccentric. As we will see, this will occur when $B_2$ approaches $B_1$ and the interaction between them is not negligible any more.

Being $T_1$ and $T_2$ the periods of the two smaller bodies around body $0$, since $B_1$ is placed closer to $B_0$ than $B_2$, it is expected to move faster than $B_2$, and therefore $T_1\leq T_2$, meaning that body $1$ may complete more than one revolution before the alignment happens, where body $2$ may not complete any. 

\begin{rmk}[Classical period of syzygy and resonances]\label{rmk:resonances} 
The time for a classical syzygy (an alignment) can be predicted using the following expression:
\begin{equation}
    T_s = \cfrac{1}{\cfrac{1}{T_1} - \cfrac{1}{T_2}}.
    \label{eq:Ts}
\end{equation}
Resonances among bodies can be recognized in terms of the angle between alignments $\theta$ (Figure~\ref{fig:qbs_1}):
\begin{itemize}
    \item when $\theta=\pi$, $B_1$ and $B_2$ are on the negative part of $x-$axis in the second alignment, while $B_0$ is on the positive part. This corresponds to the period of the syzygy to be equal to $T_2/2$, and following \eqref{eq:Ts} to $T_2 = 3T_1$.    
    \item when  $\theta=2\pi$, $B_1$ and $B_2$ are again on the positive part of $x-$axis in the second alignment, while $B_0$ is on the negative part. This corresponds to the period of the syzygy to be equal to $T_2$, and following~\eqref{eq:Ts} to $T_2 = 2T_1.$
\end{itemize}

While bodies $1$ and $2$ describe circular orbits, their cut on the $x-$axis corresponds to their semi-major axis. Then, for Kepler's third law, $T^2_i=4\pi^2a^3_i/\mu_i$, for $i\in \{1,2\}$, resonance $\theta=\pi$ is expected to take place when $B_2$ occupies an initial position $x_2=\sqrt[3]{3^2}\approx2.08$, while resonance $\theta=2\pi$ is expected at $x_2=\sqrt[3]{2^2}\approx1.59$.
\end{rmk}

Statements introduced in this preamble will be useful along the explanations of the following two sections. For both, the configuration and initial values are the ones collected in Figure~\ref{fig:ic-poincare}. In Section~\ref{subsec:AE0} we focus on the case in which the bodies are ordered by mass $m_0\gg m_1\geq m_2$, where in Section~\ref{subsec:AE1} the intermediate body is the smaller one,  $m_0\gg m_2\geq m_1$. In both situations we have considered bodies $1$ and $2$, the two planets, in prograde orbit, since in planetary systems it is natural that all planets revolve in the same direction. 

\begin{figure}[h]
\centering
\begin{minipage}{0.55\textwidth}
\centering
 \includegraphics[width=\linewidth]{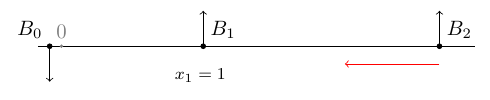}
\end{minipage}
\hfill
\begin{minipage}{0.35\textwidth}
\centering
\begin{tabular}{|c|c|c|c|}
\hline
$x_1$ & $x_2$ & $\dot{y}_1$ & $\dot{y}_2$\\
\hline \hline
$1$ & $8$ & $1$ & $0.7$ \\ \hline
\end{tabular}
\end{minipage}
\caption{Left, scheme of the initial positions and directions of motion for the three bodies in Poincar\'e regime; gray $0$ marks the origin of coordinates and the red arrow shows the sense of the continuation. Right, initial conditions taken for the continuation. Notice that $x_0$ and $\dot{y}_0$ are given by conservation of the center of mass.}
\label{fig:ic-poincare}
\end{figure}

Notice that, $B_2$ starts at a large distance from $B_1$, whose initial velocity corresponds to a circular orbit according to Equation~\eqref{eq:rel-x-vy} and, finally, we have taken velocity of $B_2$ equal to $0.7$.

\subsubsection{Planets ordered by mass}\label{subsec:AE0}
In this first section of academic cases of Poincar\'e type of solutions, we  analyse four cases in which the size of the masses is descending as they are further from the origin. This could correspond to systems like Sun-Jupiter-Saturn or Sun-Earth-Mars. To cover different situations and analyse the influence of the less massive body, we take normalized masses $m_1=10^{-3}$, $m_2\in \{10^{-6},10^{-5},10^{-4},10^{-3}\}$ and $m_0=1-m_1-m_2$, such that the total mass is $1$. Notice that these situations could also correspond to a planet and two satellites.

\begin{figure}[h!]
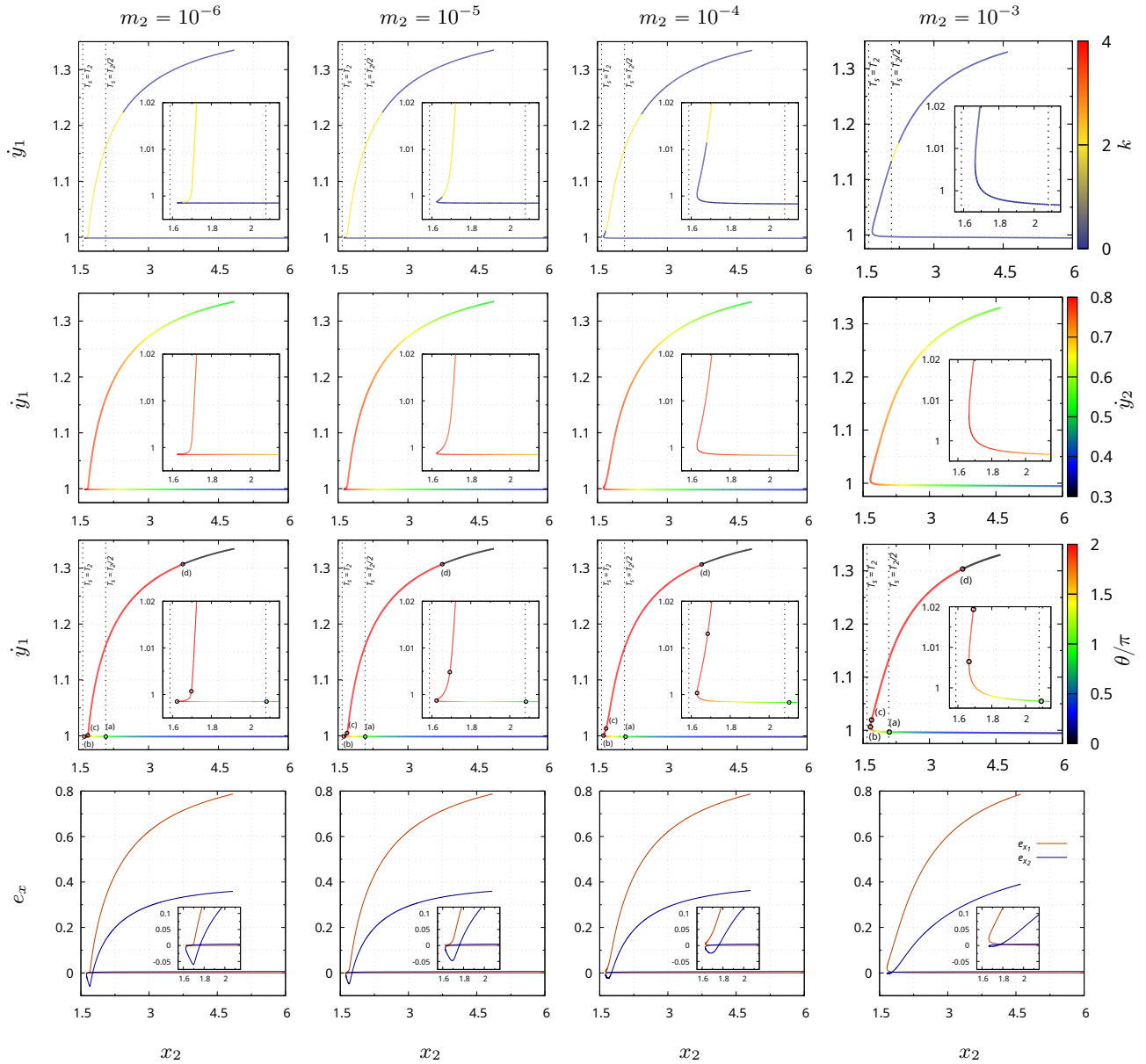

    \centering
    \footnotesize
    \setlength{\tabcolsep}{3pt}
    \renewcommand{\arraystretch}{1.5}

    \begin{adjustbox}{max totalsize={\textwidth}{0.98\textheight},center}
    \begin{tabular}{@{}
        >{\centering\arraybackslash}m{0.025\textwidth}
        *{3}{>{\centering\arraybackslash}m{\maincurvefigw}}
        >{\centering\arraybackslash}m{\rightmaincurvefigw}
        @{}
        >{\centering\arraybackslash}m{0.025\textwidth}}

        &
        {\footnotesize $m_2=10^{-6}$} &
        {\footnotesize $m_2=10^{-5}$} &
        {\footnotesize $m_2=10^{-4}$} &
        {\footnotesize $m_2=10^{-3}$}
        \\

        \rowlab{$\dot{y}_1$}
        
        &
        \maincurveplot{5}{AE0}{AE0_mapa_all}
        &
        \maincurveplot{12}{AE0}{AE0_mapa_all}
        &
        \maincurveplot{19}{AE0}{AE0_mapa_all}
        &
        \rightmaincurveplot{26}{AE0}{AE0_mapa_all}
        &
        
        \\

        \rowlab{$\dot{y}_1$}
        &
        \maincurveplot{1}{AE0}{AE0_mapa_all}
        &
        \maincurveplot{8}{AE0}{AE0_mapa_all}
        &
        \maincurveplot{15}{AE0}{AE0_mapa_all}
        &
        \rightmaincurveplot{22}{AE0}{AE0_mapa_all}
        & \rowlab{$\dot{y}_2$}
        
        \\
        
        \rowlab{$\dot{y}_1$}
        &
        \maincurveplot{4}{AE0}{AE0_mapa_all}
        &
        \maincurveplot{11}{AE0}{AE0_mapa_all}
        &
        \maincurveplot{18}{AE0}{AE0_mapa_all}
        &
        \rightmaincurveplot{25}{AE0}{AE0_mapa_all}
        &
        \rowlab{$\theta/\pi$}
        
        \\

        \rowlab{$e_x$}
        &
        \eccentrycurveplot{7}{AE0}{AE0_mapa_all}
        &
        \eccentrycurveplot{14}{AE0}{AE0_mapa_all}
        &
        \eccentrycurveplot{21}{AE0}{AE0_mapa_all}
        &
        \eccentrycurveplot{28}{AE0}{AE0_mapa_all}
        \\
        
        &
        {\footnotesize $x _2$} &
        {\footnotesize $x _2$} &
        {\footnotesize $x _2$} &
        {\footnotesize $x _2$} &
        \\
    \end{tabular}
    \end{adjustbox}

    \caption{Poincar\'e families of solutions computed for $m_0=1-m_1-m_2$, $m_1=10^{-3}$ and $m_2 \in \{10^{-6},10^{-5},10^{-4},10^{-3}\}$ for each of the four columns. First three rows show the curve continuation in $(x_2,\dot{y}_1)$, colored according to the stability type, to the velocity of the second body $\dot{y}_2$, and to the angle $\theta/\pi$. Last row shows the initial $x$ component of the eccentricity vectors of $B_1$ and $B_2$, with respect to $B_0$, as the second approaches the first. Insets show a zoom of the bottom-left region.}
    
    \label{fig:AE0_Poincare}
\end{figure}

\begin{figure}[h!]
    \centering
    \footnotesize
    \setlength{\tabcolsep}{3pt}
    \renewcommand{\arraystretch}{1.5}

    \begin{adjustbox}{max totalsize={\textwidth}{0.98\textheight},center}
    \begin{tabular}{@{}
        >{\centering\arraybackslash}m{0.025\textwidth}
        *{2}{>{\centering\arraybackslash}m{\maincurvefigw}}
        >{\centering\arraybackslash}m{\rightmaincurvefigw}
        >{\centering\arraybackslash}m{\rightmaincurvefigw}
        @{}
        >{\centering\arraybackslash}m{0.025\textwidth}}

        &
        {\footnotesize $m_2=10^{-6}$} &
        {\footnotesize $m_2=10^{-5}$} &
        {\footnotesize $m_2=10^{-4}$} &
        {\footnotesize $m_2=10^{-3}$}
        \\

        \rowlab{$\dot{y}_1$}
        &
        \maincurveplot{2}{AE0}{AE0_mapa_all}
        &
        \maincurveplot{9}{AE0}{AE0_mapa_all}
        &
    \multicolumn{1}{
        >{\centering\arraybackslash}m{\rightmaincurvefigw}||
    }{
        \rightmaincurveplot{16}{AE0}{AE0_mapa_all}
    }
        &
        \rightmaincurveplot{23}{AE0}{AE0_mapa_all}
        & 
        \rowlabten{$\mathcal{H}\times10^{-4}$}
        
        \\
        
        \rowlab{$\dot{y}_1$}
        &
        \maincurveplot{3}{AE0}{AE0_mapa_all}
        &
        \maincurveplot{10}{AE0}{AE0_mapa_all}
        &
    \multicolumn{1}{
        >{\centering\arraybackslash}m{\rightmaincurvefigw}||
    }{
        \rightmaincurveplot{17}{AE0}{AE0_mapa_all}
    }
        &
        \rightmaincurveplot{24}{AE0}{AE0_mapa_all}
        &
        \rowlabten{$L\times10^{-3}$}
        \\
        
        &
        {\footnotesize $x _2$} &
        {\footnotesize $x _2$} &
        {\footnotesize $x _2$} &
        {\footnotesize $x _2$} &
        \\
    \end{tabular}
    \end{adjustbox}

    \caption{Poincar\'e families of solutions computed for $m_0=1-m_1-m_2$, $m_1=10^{-3}$ and $m_2 \in \{10^{-6},10^{-5},10^{-4},10^{-3}\}$ for each of the four columns. The rows show the curve continuation in $(x_2,\dot{y}_1)$, first colored according to the value of the total energy and second according to the value of the total angular momentum.}
    \label{fig:AE0_Poincare1}
\end{figure}

\begin{table}[h!]
\centering
{\footnotesize
\begin{tabular}{|c|c|c|c|c||c|c|c|c|c|c|}
\hline\multicolumn{5}{|c||}{\textbf{Solution of TBP}} & \multicolumn{6}{|c|}{\textbf{Approximation of keplerian elements}} \\ \hline
\multicolumn{11}{|c|}{\hspace{-3 cm}(a) \textbf{Resonance} $\theta=\pi$} \\ \hline
Case & $m_2$  &   $x_2$  & $\theta$     & $T$     & $e_{x_1}$      & $a_1$ & $T_1$  & $e_{x_2}$      & $a_2$  & $T_2$\\ \hline
1    & $10^{-6}$ & $2.0834$  & $\pi$    & $9.43655$ & $9.3\times10^{-7}$ & $1.0000$ &  $6.2832$ & $4.1\times10^{-3}$ & $2.0834$ & $18.9041$\\ \hline
2    & $10^{-5}$ & $2.0834$  & $\pi$    & $9.43704$ & $9.3\times10^{-6}$ & $1.0000$ &  $6.2832$ & $4.1\times10^{-3}$ & $2.0834$ & $18.9041$\\ \hline
3    & $10^{-4}$ & $2.0834$  & $\pi$    & $9.44197$ & $9.3\times10^{-5}$ & $1.0000$ &  $6.2835$ & $4.1\times10^{-3}$ & $2.0834$ & $18.9040$\\ \hline
4    & $10^{-3}$ & $2.0932$  & $\pi$    & $9.45777$ & $9.1\times10^{-4}$ & $1.0000$ &  $6.2863$ & $4.1\times10^{-3}$ & $2.0932$ & $19.0380$\\ \hline \hline

\multicolumn{11}{|c|}{\hspace{-3 cm}(b) \textbf{Turning point}} \\ \hline
Case & $m_2$  &   $x_2$  & $\theta$     & $T$     & $e_{x_1}$      & $a_1$ & $T_1$  & $e_{x_2}$      & $a_2$  & $T_2$\\ \hline
1    & $10^{-6}$ & $1.6235$  & $6.1116$ & $12.4109$ & $5.0\times10^{-5}$ & $1.0000$ &  $6.2832$  & $-8.9\times10^{-3}$& $1.6234$ & $13.0027$\\ \hline
2    & $10^{-5}$ & $1.6241$  & $6.1058$ & $12.4151$ & $4.9\times10^{-4}$ & $1.0000$ &  $6.2832$  & $-8.6\times10^{-3}$& $1.6240$ & $13.0101$ \\ \hline
3    & $10^{-4}$ & $1.6297$  & $6.0743$ & $12.4499$ & $4.1\times10^{-3}$ & $1.0000$ &  $6.2837$  & $-6.7\times10^{-3}$& $1.6297$ & $13.0782$ \\ \hline
4    & $10^{-3}$ & $1.6655$  & $5.8833$ & $12.5681$ & $2.0\times10^{-2}$ & $1.0004$ &  $6.2902$  & $-1.3\times10^{-3}$& $1.6655$ & $13.5117$ \\ \hline \hline

\multicolumn{11}{|c|}{\hspace{-3 cm}(c) \textbf{Minimum of} $e_{x_2}$} \\ \hline
Case & $m_2$  &   $x_2$  & $\theta$     & $T$     & $e_{x_1}$      & $a_1$ & $T_1$  & $e_{x_2}$      & $a_2$  & $T_2$\\ \hline
1    & $10^{-6}$ & $1.6971$  & $6.2816$ & $12.6680$ & $4.4\times10^{-3}$ & $1.0000$ &  $6.2834$  & $-6.0\times10^{-2}$& $1.6911$ & $13.8244$\\ \hline
2    & $10^{-5}$ & $1.6922$  & $6.2775$ & $12.8240$ & $1.3\times10^{-2}$ & $1.0002$ &  $6.2847$  & $-4.8\times10^{-2}$& $1.6884$ & $13.7909$\\ \hline
3    & $10^{-4}$ & $1.6845$  & $6.2578$ & $13.1427$ & $3.0\times10^{-2}$ & $1.0009$ &  $6.2918$  & $-2.4\times10^{-2}$& $1.6836$ & $13.3725$\\ \hline
4    & $10^{-3}$ & $1.6919$  & $6.1174$ & $13.3491$ & $4.7\times10^{-2}$ & $1.0022$ &  $6.3069$  & $-3.3\times10^{-3}$& $1.6919$ & $13.8347$\\ \hline \hline

\multicolumn{11}{|c|}{\hspace{-3 cm}(d) \textbf{Resonance} $\theta=2\pi$} \\ \hline
Case & $m_2$  &   $x_2$  & $\theta$     & $T$     & $e_{x_1}$      & $a_1$ & $T_1$  & $e_{x_2}$      & $a_2$  & $T_2$\\ \hline
1    & $10^{-6}$ & $3.7366$  & $2\pi$       & $82.1064$ & $7.1\times10^{-1}$ & $2.0375$ &  $18.2740$  & $3.3\times10^{-1}$ & $4.2031$ & $54.1692$ \\ \hline
2    & $10^{-5}$ & $3.7368$  & $2\pi$       & $82.1151$ & $7.1\times10^{-1}$ & $2.0376$ &  $18.2751$  & $3.3\times10^{-1}$ & $4.2034$ & $54.1740$ \\ \hline
3    & $10^{-4}$ & $3.7376$  & $2\pi$       & $82.1489$ & $7.1\times10^{-1}$ & $2.0376$ &  $18.2761$ & $3.3\times10^{-1}$ & $4.2042$ & $54.1902$  \\ \hline
4    & $10^{-3}$ & $3.7493$  & $2\pi$       & $82.6345$ & $7.1\times10^{-1}$ & $2.0397$ &  $18.3129$ & $3.3\times10^{-1}$ & $4.2172$ & $54.4412$  \\ \hline
\end{tabular}
}
    \caption{Some relevant data for the Poincar\'e type of RPOs computed for the four sets of masses shown in Figure~\ref{fig:AE0_Poincare} at resonances $\theta\in \{\pi,2\pi\}$, at the turning point and at the minimum of the initial $x-$component of eccentricity vector of body $2$. Solution data corresponds to the initial position of body $2$, $x_2$, angle between the two alignments, $\theta$ and period of the solution, $T$. Last six columns contain eccentricity, semi-major axis and period for body $1$ ($e_{x_1}$, $a_1$, $T_1$) and for body $2$ ($e_{x_2}$, $a_2$, $T_2)$, both by approximating their motion to a Keplerian orbit around body $0$.}
    \label{tab:AE0}
\end{table}

We start the computation of the solutions placing $B_2$ far away from $B_1$, whose initial position is always $x_1=1$, and then, we approach body $2$ to body $1$ simply reducing the value of $x_2$ coordinate. In Figures~\ref{fig:AE0_Poincare} and \ref{fig:AE0_Poincare1} we can see the computed curves of solutions in terms of the initial position of $B_2$, $x_2$, and initial $y$ component of velocity of $B_1$, $\dot{y}_1$, found for these 4 sets of masses. Notice that, for a given set of masses, each point in the continuation curve $(x_2,\dot{y}_1)$ is a different solution of the general Three-Body Problem.

Here, it is worth to mention that our fourth set of masses corresponds to one of family $I_1$ computed by \cite{hadjidemetriou1976A}. His and our curve of solutions present the same profile in spite of using different representing coordinates ($(x_2,x_1)$ instead of $(x_2,\dot{y}_1)$), reference frame (rotating instead of inertial) and methodology (analytic continuation from CRTBP instead of direct computation in TBP).

Figure~\ref{fig:AE0_Poincare} contains sixteen images, each of the four columns of images corresponds to a given set of masses, while the four rows include different information of the obtained solutions. In the first row, the continuation curve $(x_2,\dot{y}_1)$ is colored according to the stability type; solutions in blue are totally elliptic (TE), those in green are partially hyperbolic (EH), orange is used for totally hyperbolic (TH) solutions and red for complex saddle (CS) type of solutions. In these plots two vertical lines are included to mark the expected $x_2$ position for solutions corresponding to the described resonances $\theta=\pi$ and $\theta=2\pi$, see Remark~\ref{rmk:resonances}. In the second and third rows, the same continuation curve is colored according to the value of the initial $y$ component of velocity of body $2$, $\dot{y}_2$, and according to 
the normalized 
angle between alignments, 
$\theta/\pi\in[0,2]$. 
And in the fourth row, variations of the $x$ component of the Keplerian eccentricity vectors of bodies $B_1$ and $B_2$ in their orbits with respect to $B_0$ are shown. 

Similarly, Figure~\ref{fig:AE0_Poincare1} shows the same continuation curves for the four cases, colored according to the total energy $\mathcal{H}$, Equation~\eqref{eq:tbp_ham}, and to the total angular momentum $L$,  Equation~\eqref{eq:tbp_am}. In order to support the explanation of the phenomena found, Table~\ref{tab:AE0} collects key information of the computed solutions at special locations of the continuation curves for the four sets of masses, marked with letters (a)-(d) in Figure~\ref{fig:AE0_Poincare}. Table~\ref{tab:AE0} also includes the approximation through two-body problem of the Keplerian orbital elements (eccentricity, semi-major axis and period) of $B_1$ and $B_2$ around $B_0$. Information provided in this table is developed progressively in the following paragraphs.

First of all, let us notice that for all solutions the total angular momentum is not null and that all the solutions own negative total energy, see Figure~\ref{fig:AE0_Poincare1}, 
which is mandatory when the motion of the three bodies is bounded.

Now, we start the explanations looking at Figure~\ref{fig:AE0_Poincare}, for the four sets of masses. There, we observe the same phenomenon: at the beginning of the continuation, when bodies $1$ and $2$ are far from each other, the initial velocity for $B_1$ is close to the unity, what, as we have seen, is related to $e_1=0$ (circular orbits). This can be also observed by looking at the last row of graphs for both bodies. For this range of solutions, with circular orbits for $B_1$, we find that its Keplerian period ($T_1$) is close to $2\pi$ and that all the computed RPOs for the three bodies are stable for the four cases considered.

Then, as the $B_2$ approaches body $B_1$, i.e. $x_2$ is reduced, its velocity gets larger when $\dot{y}_1\approx1$. When the \textbf{resonance} $\theta=\pi$ (special solution (a)) is crossed, we find some unstable solutions for the cases in which the mass of the body $2$ is comparable with that of body $1$, $m_2=10^{-4}$ and $m_2=10^{-3}$. Regardless the resonance and the unstable solutions, the eccentricities of both bodies are still close to zero, meaning that the bodies still describe circular or almost circular orbits. Right after the resonance is crossed, we find stable solutions again.

Looking at the data collected in Table~\ref{tab:AE0} for this special situation, $\theta=\pi$, we check that the relations  $T_s=T_2/2$ and $T_2=3T_1$ explained in Remark~\ref{rmk:resonances} hold for the computed solutions, which are found for the expected value $x_2\approx2.08$. In addition, data in the table confirm that the eccentricities of bodies $1$ and $2$ at resonance $\theta=\pi$ are very small.

After the resonance $\theta=\pi$ is crossed, as $x_2$ gets smaller, the four curves of solutions present a \textbf{turning point} (special solution (b)), that is sharper as $m_2$ is smaller, as one can check in Figure~\ref{fig:AE0_Poincare} and looking at the $x_2$ value of solutions corresponding to the turning point collected in Table~\ref{tab:AE0}. This turning point occurs when the solutions are close to the expected resonance among circular orbits when $\theta=2\pi$, corresponding to $T_s=T_2$, $T_2=2T_1$ and $x_2\approx1.59$, Remark~\ref{rmk:resonances}. Looking at these values in the table, we find a good agreement between the expected values and the computed ones. However, this does not happen exactly for $\theta=2\pi$. In fact, we observe how the variation of $\theta$ starts fluid in the continuation curve, see third row of Figure~\ref{fig:AE0_Poincare}, until the turning point is reached. 

Besides, at the turning point we find a minimum, really close to zero, of the total angular momentum, see second row of Figure~\ref{fig:AE0_Poincare1}. Notice that if the total angular momentum of a Three-Body Problem becomes zero, a binary or even triple collision becomes possible.

Going through the turning point produces a change in the stability of the solutions. In the first three cases, right after the turning point we find unstable solutions with 2 real eigenvalues (Elliptic-Hyperbolic type), until $B_1$ and $B_2$ get apart again and we recover stable solutions. For the last case, in which $m_1=m_2$, the turning point is less sharp and the range of unstable solutions is quite narrow.

Right after the turning point, eccentricity vector of the body $2$ gets negative, meaning that $B_2$'s orbit is not circular any more and that its periapsis takes place at the negative $x-$axis. This effect is again stronger as $m_2$ is smaller, as one can check at the third type of special solutions, (c) \textbf{minimum value of} $e_{x_2}$, collected in Table~\ref{tab:AE0}. Consequently, we conclude that the bodies $1$ and $2$ can not come closer describing still almost circular orbits with respect to the body $0$.

Notice that the profile of the curves $(x_2,\dot{y}_1)$ and $(x_2,e_{x_1})$ in Figure~\ref{fig:AE0_Poincare}, look exactly equal, due to the relation derived in previous section, Equation~\eqref{eq:ecvec}, and outlining that after the turning point, both the velocity and the eccentricity of $B_1$ increase. Differently as it happens with $\dot{y}_1$, the velocity of $B_2$ starts increasing before the turning point, gets its maximum there, and then it decreases again as $B_1$ and $B_2$ get apart.

Then, following the continuation after the turning point, we see that the velocity of the body $1$, as well as the eccentricity of both $B_1$ and $B_2$, increase significantly. This effect combined with the slow variation of $\theta$, makes that the solution of Three-Body Problem corresponding at \textbf{resonance} $\theta=2\pi$ (special solution (d)) occurs for very eccentric orbits. As a result, our families of RPOs of the Three-Body Problem transit from orbits of first species to orbits of second species described in Section~\ref{sec:pmot}.

This transition among first and second species of Poincar\'e solutions along this planetary family was already observed by \cite{hadjidemetriou1976A}. They also found the instability related with the resonance $T_2=3T_1$ and the relation between the turning point and the resonance $T_2=2T_1$. However, after the turning point they claimed not to be able to compute the solution corresponding to the resonance $\theta=2\pi$, that $B_2$ remains on circular orbit and that there are no other unstable solutions. In our work, we are able to compute the solution corresponding to such resonance for very eccentric orbits of $B_1$ and $B_2$ and we find a range of unstable solutions after the turning point. It must be taken into account that Hadjidemetriou's solutions come from continuation of circular orbits in the CRTBP, for this they also assume to have difficulties to properly continue solutions with high eccentricity. A difficulty that our methodology does not find.

Hadjidemetriou also explains that since there is an infinity of resonances, in the continuation of one family of periodic orbits of the CRTBP, one might expect an infinity of separate families of planetary-type orbits. This is very important since it is another evidence of the dynamical diversity of the TBP. Along with the families of RPOs we are computing, other families can bifurcate giving rise to different configurations with their specific dynamical features.

Recall that those solutions in families of relative periodic orbits, that occur for $\theta/(2\pi)\in \Q$, like the ones at resonances $\theta=\pi$ and $\theta=2\pi$, are in fact absolute periodic orbits. Besides, notice that the computed solutions at resonance $\theta=2\pi$ are not the classical Euler solutions of collinear motion in the TBP, see Remark~\ref{rmk:Euler}. In Euler solutions the three bodies are forced to move with the same and constant angular velocity, what makes them belong to the same line for all time. On the contrary, in our solutions, each of the three bodies moves with different and non-constant angular velocity, making $B_1$ complete two rounds in the time $B_2$ completes one round for periodic solution with $\theta=2\pi$.

The slow variation of $\theta$, results in a slow variation of the period of the computed RPOs, $T$. Looking at Figure~\ref{fig:Ts}, we check that the prediction of time for a syzygy given by classical Equation~\eqref{eq:Ts} agrees with the period of the computed RPO before the turning point. After it, when orbits are not quasi-bicircular, Equation~\eqref{eq:Ts} is not suitable any more.

\begin{figure}[ht]
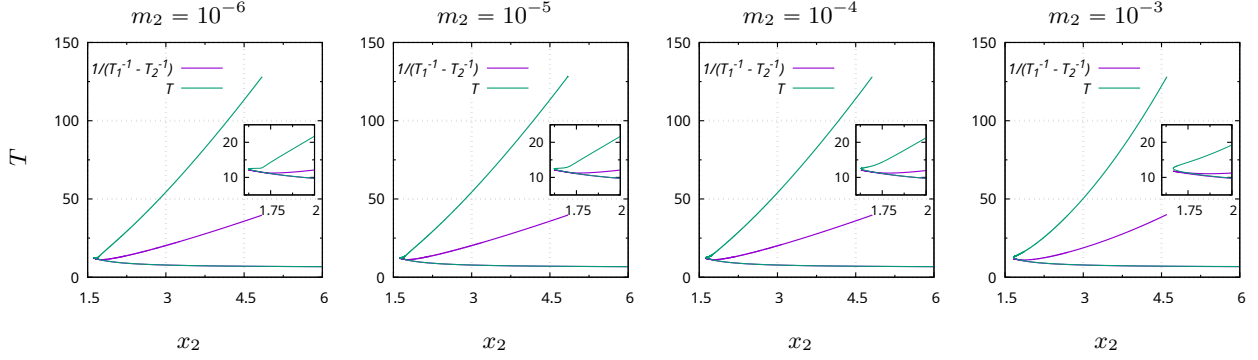

    \centering
    
    \centering
    \footnotesize
    \setlength{\tabcolsep}{5pt}
    \renewcommand{\arraystretch}{1.5}

    \begin{adjustbox}{max totalsize={\textwidth}{1\textheight},center}
    \begin{tabular}{@{}
        >{\centering\arraybackslash}m{0.01\textwidth}
        *{4}{>{\centering\arraybackslash}m{\maincurvefigw}}
        @{}}

        &
        {\footnotesize $m_2=10^{-6}$} &
        {\footnotesize $m_2=10^{-5}$} &
        {\footnotesize $m_2=10^{-4}$} &
        {\footnotesize $m_2=10^{-3}$}
        \\

        \rowlab{$T$} &        
        \maincurveplot{1}{AE0}{pinto-Ts}
        &
        \maincurveplot{2}{AE0}{pinto-Ts}
        &
        \maincurveplot{3}{AE0}{pinto-Ts}
        &
        \maincurveplot{4}{AE0}{pinto-Ts}
        
        \\
        &
        {\footnotesize $x _2$} &
        {\footnotesize $x _2$} &
        {\footnotesize $x _2$} &
        {\footnotesize $x _2$} 
    \end{tabular}
    \end{adjustbox}
    
    \caption{In green, variation of the period $T$ of the computed relative periodic solutions as $B_2$ approaches $B_1$ for the four cases in Figure~\ref{fig:AE0_Poincare}. In purple the expected period for a syzygy according to Equation~\eqref{eq:Ts}.}
    \label{fig:Ts}
\end{figure}

The continuations vanish when bodies $1$ and $2$ are far from each other, describing significantly eccentric orbits, and the total energy of the problem is close to zero, see Figure~\ref{fig:AE0_Poincare1}. Notice that if a null energy value is reached the motion of the three bodies is not bounded, then it is natural not to be able to find RPOs to the TBP for those values of total energy.

\paragraph{Some trajectories}

\begin{figure}[h!]
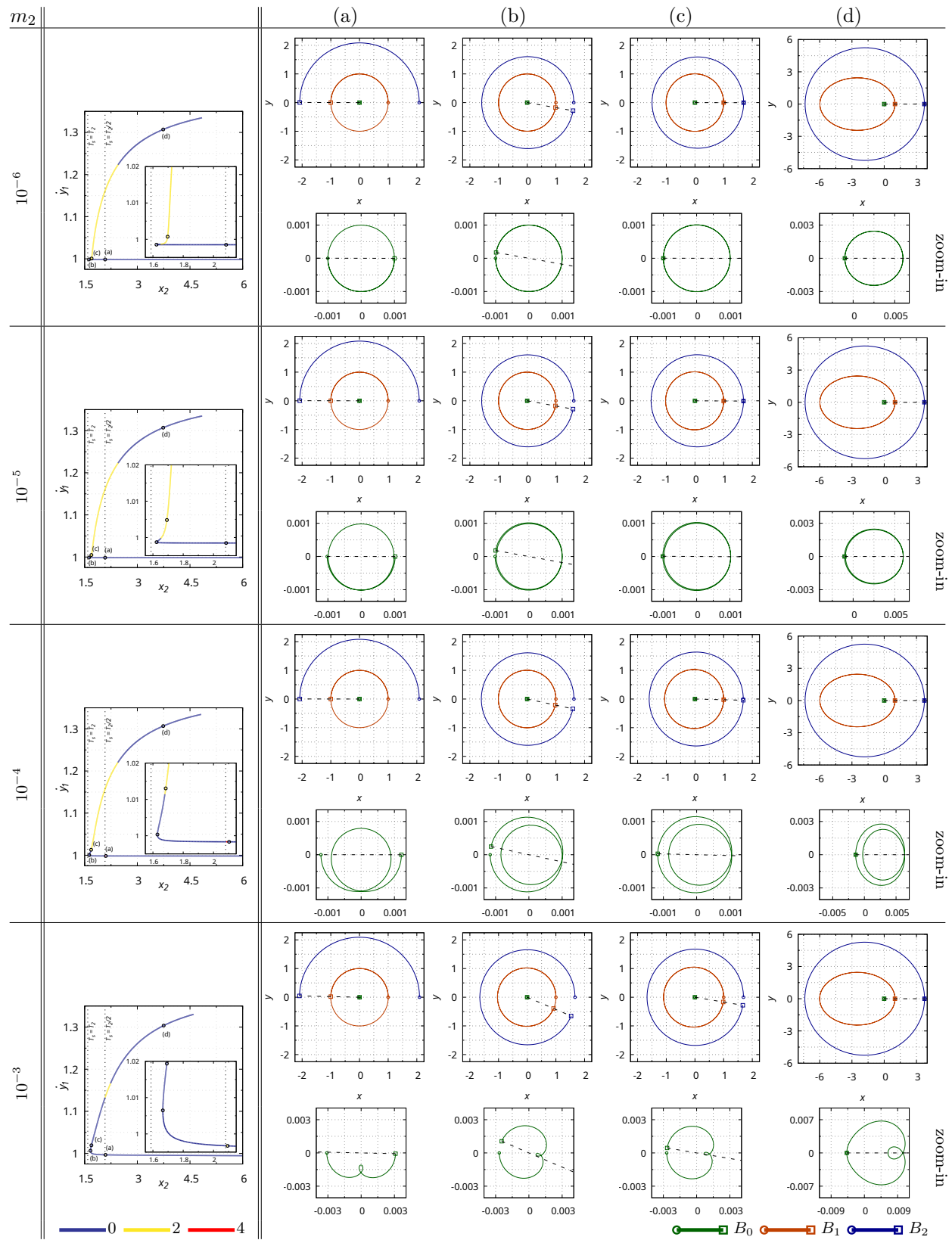

    \centering
    \setlength{\tabcolsep}{2pt}
    \renewcommand{\arraystretch}{0.75}
    \setlength{\arrayrulewidth}{0.4pt}
    \setlength{\doublerulesep}{1.2pt}

    \begin{adjustbox}{max totalsize={\textwidth}{0.9\textheight},center}
    \begin{tabular}{@{}
        >{\centering\arraybackslash}m{0.03\textwidth}||
        >{\centering\arraybackslash}m{0.225\textwidth}||
        *{4}{>{\centering\arraybackslash}m{\mainfigw}}
        >{\centering\arraybackslash}m{0.01\textwidth}
        @{}}

        $m _2$ & &
        (a) & (b) & (c) & (d)
        \\ \hline

         \multirow{2}{*}[-3em]{\rowlab{$10^{-6}$}}
        &
        \multirow{2}{*}[0em]{\maincurveplot{6}{AE0}{AE0_mapa_all}}
        &
        \mainplot{6}{AE0}{coppeq0}
        &
        \mainplot{3}{AE0}{coppeq0}
        &
        \mainplot{1}{AE0}{coppeq0}
        &
        \mainplot{7}{AE0}{coppeq0}
        \\

        &
        &
        \zoomplot{6}{AE0}{coppeq0}
        &
        \zoomplot{3}{AE0}{coppeq0}
        &
        \zoomplot{1}{AE0}{coppeq0}
        &
        \zoomplot{7}{AE0}{coppeq0}
        & \rotatebox[origin=r]{-90}{\footnotesize zoom-in}
        \\

        \hline

        \multirow{2}{*}[-3em]{\rowlab{$10^{-5}$}}
        &
        \multirow{2}{*}[0em]{\maincurveplot{13}{AE0}{AE0_mapa_all}}
        &
        \mainplot{6}{AE0}{coppeq1}
        &
        \mainplot{3}{AE0}{coppeq1}
        &
        \mainplot{2}{AE0}{coppeq1}
        &
        \mainplot{8}{AE0}{coppeq1}
        \\
        &
        &
        \zoomplot{6}{AE0}{coppeq1}
        &
        \zoomplot{3}{AE0}{coppeq1}
        &
        \zoomplot{2}{AE0}{coppeq1}
        &
        \zoomplot{8}{AE0}{coppeq1}
        & \rotatebox[origin=r]{-90}{\footnotesize zoom-in}
        \\

        \hline

         \multirow{2}{*}[-3em]{\rowlab{$10^{-4}$}}
        &
        \multirow{2}{*}[0em]{\maincurveplot{20}{AE0}{AE0_mapa_all}}
        &
        \mainplot{6}{AE0}{coppeq2}
        &
        \mainplot{4}{AE0}{coppeq2}
        &
        \mainplot{1}{AE0}{coppeq2}
        &
        \mainplot{8}{AE0}{coppeq2}
        \\

        &
        &
        \zoomplot{6}{AE0}{coppeq2}
        &
        \zoomplot{4}{AE0}{coppeq2}
        &
        \zoomplot{1}{AE0}{coppeq2}
        &
        \zoomplot{8}{AE0}{coppeq2}
        & \rotatebox[origin=r]{-90}{\footnotesize zoom-in}
        \\

        \hline

         \multirow{2}{*}[-3em]{\rowlab{$10^{-3}$}}
        &
        \multirow{2}{*}[0em]{\maincurveplot{27}{AE0}{AE0_mapa_all}}
        &
        \mainplot{6}{AE0}{coppeq3}
        &
        \mainplot{4}{AE0}{coppeq3}
        &
        \mainplot{2}{AE0}{coppeq3}
        &
        \mainplot{8}{AE0}{coppeq3}
        \\
        
        &
        &
        \zoomplot{6}{AE0}{coppeq3}
        &
        \zoomplot{4}{AE0}{coppeq3}
        &
        \zoomplot{2}{AE0}{coppeq3}
        &
        \zoomplot{8}{AE0}{coppeq3} 
        & \rotatebox[origin=r]{-90}{\footnotesize zoom-in}
        \\

        & 
    {\footnotesize
      \gpshortline[gpzero]{TE }%
      \gpshortline[gptwo]{EH }%
      \gpshortline[gpfour]{TH }%
      \gpshortline[gpsix]{CS}%
    }%
    & \multicolumn{4}{r}{\footnotesize\gplinemarked[gpsun]{$B_0$ }\gplinemarked[gpjup]{$B_1$ }\gplinemarked[gpsat]{$B_2$ }}
    \end{tabular}
    \end{adjustbox}

    \caption{Poincar\'e solutions for the four cases shown in Figure~\ref{fig:AE0_Poincare}, with $m_2$ in the first column. Columns (a)--(d) correspond to the selected solutions marked in the continuation curves of the second row. Bigger images show the trajectories in the $(x,y)$ plane of the three bodies; in blue body $2$, in orange body $1$, and in green body $0$. Small images are zoom-in body $0$.}
    \label{fig:AE0_tras}
\end{figure}

In order to give a more detailed explanation of the phenomena found, Figure~\ref{fig:AE0_tras} includes the trajectories of the three bodies for the four sets of masses considered and for the four special solutions marked with letters in the third row of Figure~\ref{fig:AE0_Poincare} and collected in Table~\ref{tab:AE0}. Notice that transition among these particular solutions is soft along the continuation curve.

For these RPOs, we propagate in time from $t=0$ to $t=T$ the initial conditions of the three bodies according to Equation~\eqref{eq:tbp}. Corresponding $(x,y)$ representation are included in Figure~\ref{fig:AE0_tras}, where the orbits of the bodies $1$ and $2$ are plotted in orange and blue, respectively, in a larger image that is accompanied by other smaller one showing the orbit of $B_0$ in green. For clarity, Figure~\ref{fig:AE0_tras} is designed as a table, where each row corresponds to each of the four cases. First column collects $m_2$, second one the continuation curve colored according to the stability type and the third one contains the trajectories at the four situations (a)-(d).  

Then, plots at situation (a) contain a solution at resonance $\theta=\pi$, where one can observe how $B_2$ completes half of its orbit, while $B_1$ completes one and a half revolutions. Notice that orbit of body $B_0$ seems to perform one and a half revolutions when $m_2$ is very small (case $m_2=10^{-6}$, Figure~\ref{fig:AE0_tras}), however if we look at it carefully -and look at the following examples with larger values of $m_2$- it is clear that $B_0$ does not close its orbit for this solution. The fact that the most massive body cannot complete one revolution until the lightest body does, reveals the sensitivity of the interactions among the three bodies, despite some effects can be understood through the point of view of Keplerian (2-body) problem.

Plots (b) and (c) of Figure~\ref{fig:AE0_tras} show the closest computed solution to the turning point, and the solution for the minimum of $e_{x_2}$. Notice that in neither of these two solutions bodies $B_0$ and $B_2$ complete a revolution, in spite of the fact that the value of $\theta$ is very close to $2\pi$, see Table~\ref{tab:AE0}. In fact, after the turning point, it takes many points in the continuation curve to actually reach the resonance corresponding to $\theta=2\pi$, as already explained. This resonance is achieved for a very eccentric orbit of $B_1$, see plots (d) of Figure~\ref{fig:AE0_tras}. It is remarkable that orbits of $B_1$ and $B_2$ at resonance $\theta=2\pi$ look similar regardless the different values of the masses considered. Body $0$ seems to be the only one that behaves differently. 

Also, analysing the Keplerian elements in Table~\ref{tab:AE0} for bodies $1$ and $2$, 
we see that their values are similar for the four sets of masses at resonance $\theta=2\pi$. However, we must notice that the approximation of the semi-major axis after the turning point does not present a good agreement with the trajectories propagated in the three-body problem, meaning that approximations given by two-body problem, after interaction among bodies $1$ and $2$ is relevant, is not suitable. 



\subsubsection{Planets not ordered by mass}\label{subsec:AE1}

In this second section of academic cases of Poincar\'e type solutions, we consider the case in which the farther body is the second most massive body, like in the Sun-Mercury-Venus system. 
In order to cover different situations we take $m_2=10^{-3}$, $m_1\in \{10^{-6},10^{-5},10^{-4},10^{-3}\}$ and $m_0=1-m_1-m_2$, in such a way that last cases in previous and present sections correspond to the same set of masses. Works by Hadjidemetriou and Delibaltas \cite{hadjidemetriou1976A,delibaltas1976} argued that this situation does not 
need to be studied separately, since they consider it is the same as the previous one just interchanging the roles of $B_1$ and $B_2$. It is noteworthy that they consider the two planets with the same mass, like in our last example, if $m_1\neq m_2$ the statement is not true.

\begin{figure}[h!]
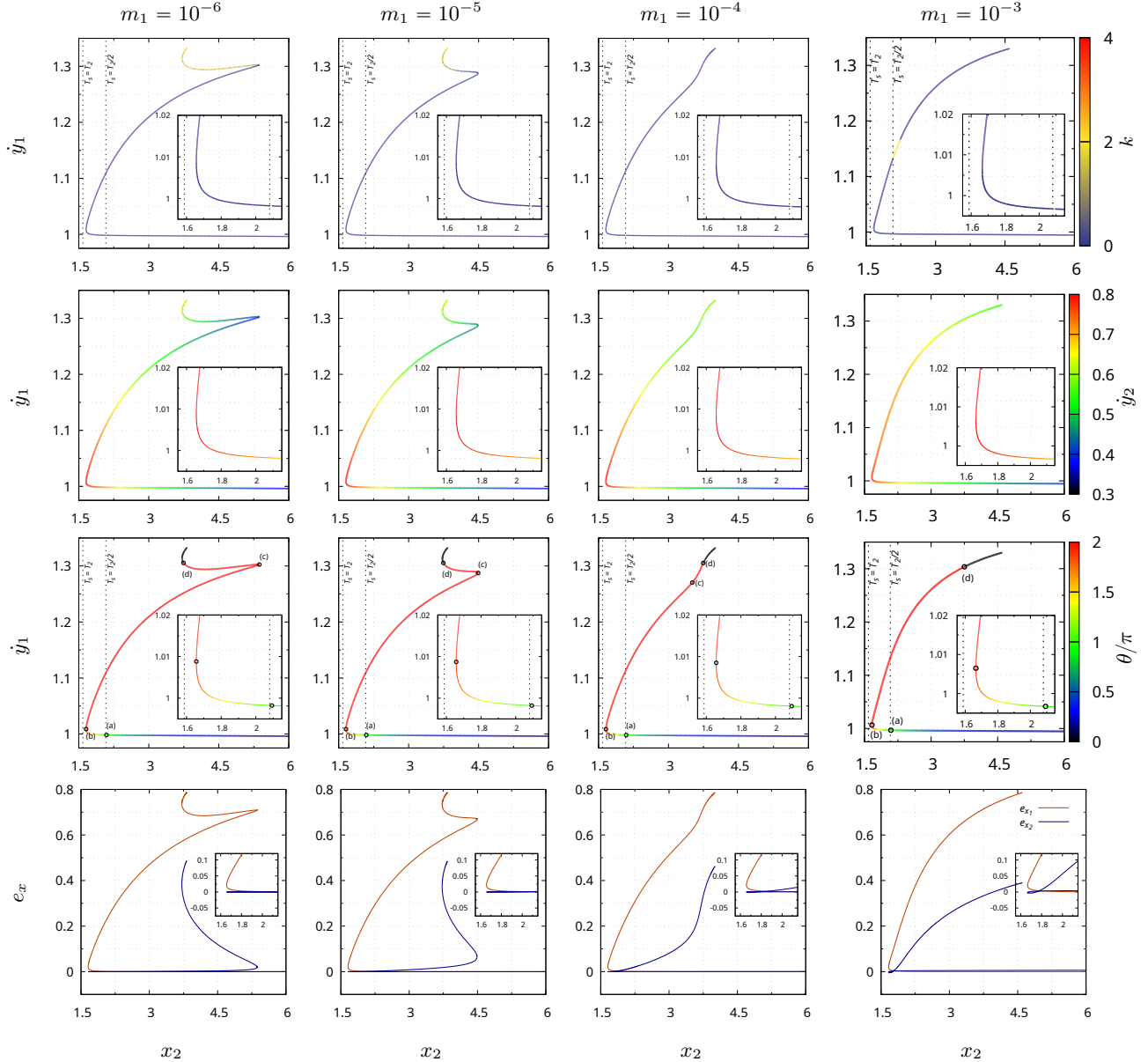

    \centering
    \footnotesize
    \setlength{\tabcolsep}{3pt}
    \renewcommand{\arraystretch}{1.5}

    \begin{adjustbox}{max totalsize={\textwidth}{0.98\textheight},center}
    \begin{tabular}{@{}
        >{\centering\arraybackslash}m{0.025\textwidth}
        *{3}{>{\centering\arraybackslash}m{\maincurvefigw}}
        >{\centering\arraybackslash}m{\rightmaincurvefigw}
        @{}
        >{\centering\arraybackslash}m{0.025\textwidth}}

        &
        {\footnotesize $m_1=10^{-6}$} &
        {\footnotesize $m_1=10^{-5}$} &
        {\footnotesize $m_1=10^{-4}$} &
        {\footnotesize $m_1=10^{-3}$}
        \\

        \rowlab{$\dot{y}_1$}
        
        &
        \maincurveplot{5}{AE1}{AE1_mapa_all}
        &
        \maincurveplot{12}{AE1}{AE1_mapa_all}
        &
        \maincurveplot{19}{AE1}{AE1_mapa_all}
        &
        \rightmaincurveplot{26}{AE1}{AE1_mapa_all}
        &
        
        \\

        \rowlab{$\dot{y}_1$}
        &
        \maincurveplot{1}{AE1}{AE1_mapa_all}
        &
        \maincurveplot{8}{AE1}{AE1_mapa_all}
        &
        \maincurveplot{15}{AE1}{AE1_mapa_all}
        &
        \rightmaincurveplot{22}{AE1}{AE1_mapa_all}
        & \rowlab{$\dot{y}_2$}
        
        \\
        
        \rowlab{$\dot{y}_1$}
        &
        \maincurveplot{4}{AE1}{AE1_mapa_all}
        &
        \maincurveplot{11}{AE1}{AE1_mapa_all}
        &
        \maincurveplot{18}{AE1}{AE1_mapa_all}
        &
        \rightmaincurveplot{25}{AE1}{AE1_mapa_all}
        &
        \rowlab{$\theta/\pi$}
        
        \\

        \rowlab{$e_x$}
        &
        \eccentrycurveplot{7}{AE1}{AE1_mapa_all}
        &
        \eccentrycurveplot{14}{AE1}{AE1_mapa_all}
        &
        \eccentrycurveplot{21}{AE1}{AE1_mapa_all}
        &
        \eccentrycurveplot{28}{AE1}{AE1_mapa_all}
        \\
        
        &
        {\footnotesize $x _2$} &
        {\footnotesize $x _2$} &
        {\footnotesize $x _2$} &
        {\footnotesize $x _2$} &
        \\
    \end{tabular}
    \end{adjustbox}

    \caption{Poincar\'e families of solutions computed for  $m_0=1-m_1-m_2$, $m_2=10^{-3}$ and $m_1\in\{10^{-6},10^{-5},10^{-4},10^{-3}\}$ for each of the four columns. First three rows show the curve continuation in $(x_2,\dot{y}_1)$, colored according to the stability type, to the velocity of the second body $\dot{y}_2$, and to the angle $\theta/\pi$. Last row shows the initial $x$ component of the eccentricity vectors of $B_1$ and $B_2$, with respect to $B_0$, as the second approaches the first. Insets show a zoom of the bottom-left region.}
    \label{fig:AE1_Poincare}
\end{figure}

\begin{figure}[h!]
    \centering
    \footnotesize
    \setlength{\tabcolsep}{3pt}
    \renewcommand{\arraystretch}{1.5}

    \begin{adjustbox}{max totalsize={\textwidth}{0.98\textheight},center}
    \begin{tabular}{@{}
        >{\centering\arraybackslash}m{0.025\textwidth}
        *{2}{>{\centering\arraybackslash}m{\maincurvefigw}}
        >{\centering\arraybackslash}m{\rightmaincurvefigw}
        >{\centering\arraybackslash}m{\rightmaincurvefigw}
        @{}
        >{\centering\arraybackslash}m{0.025\textwidth}}

        &
        {\footnotesize $m_2=10^{-6}$} &
        {\footnotesize $m_2=10^{-5}$} &
        {\footnotesize $m_2=10^{-4}$} &
        {\footnotesize $m_2=10^{-3}$}
        \\

        \rowlab{$\dot{y}_1$}
        &
        \maincurveplot{2}{AE1}{AE1_mapa_all}
        &
        \maincurveplot{9}{AE1}{AE1_mapa_all}
        &
    \multicolumn{1}{
        >{\centering\arraybackslash}m{\rightmaincurvefigw}||
    }{
        \rightmaincurveplot{16}{AE1}{AE1_mapa_all}
    }
        &
        \rightmaincurveplot{23}{AE1}{AE1_mapa_all}
        & 
        \rowlabten{$\mathcal{H}\times10^{-4}$}
        
        \\
        
        \rowlab{$\dot{y}_1$}
        &
        \maincurveplot{3}{AE1}{AE1_mapa_all}
        &
        \maincurveplot{10}{AE1}{AE1_mapa_all}
        &
    \multicolumn{1}{
        >{\centering\arraybackslash}m{\rightmaincurvefigw}||
    }{
        \rightmaincurveplot{17}{AE1}{AE1_mapa_all}
    }
        &
        \rightmaincurveplot{24}{AE1}{AE1_mapa_all}
        &
        \rowlabten{$L\times10^{-3}$}
        \\
        
        &
        {\footnotesize $x _2$} &
        {\footnotesize $x _2$} &
        {\footnotesize $x _2$} &
        {\footnotesize $x _2$} &
        \\
    \end{tabular}
    \end{adjustbox}

 \caption{Poincar\'e families of solutions computed for $m_0=1-m_1-m_2$, $m_2=10^{-3}$ and $m_1\in\{10^{-6},10^{-5},10^{-4},10^{-3}\}$ for each of the four columns. The rows show the curve continuation in $(x_2,\dot{y}_1)$, first colored according to according to the value of the total energy and second according to the value of the total angular momentum.}
    \label{fig:AE1_Poincare1}
\end{figure}

\begin{table}[h!]
\centering
{\footnotesize
\begin{tabular}{|c|c|c|c|c||c|c|c|c|c|c|}
\hline\multicolumn{5}{|c||}{\textbf{Solution of TBP}} & \multicolumn{6}{|c|}{\textbf{Approximation of keplerian elements}} \\ \hline
\multicolumn{11}{|c|}{\hspace{-3 cm}(a) \textbf{Resonance} $\theta=\pi$} \\ \hline
Case & $m_1$  &   $x_2$  & $\theta$     & $T$     & $e_{x_1}$      & $a_1$ & $T_1$  & $e_{x_2}$      & $a_2$  & $T_2$\\ \hline
0 & $10^{-6}$ & 2.0932 & $\pi$ & 9.43482 & $9.0\times10^{-4}$ & 1.0000 & 6.2863 & $4.1\times10^{-6}$ & 2.0932 & 19.0282 \\ \hline
1 & $10^{-5}$ & 2.0932 & $\pi$ & 9.43502 & $9.0\times10^{-4}$ & 1.0000 & 6.2863 & $4.1\times10^{-5}$ & 2.0932 & 19.0283 \\ \hline
2 & $10^{-4}$ & 2.0932 & $\pi$ & 9.43709 & $9.1\times10^{-4}$ & 1.0000 & 6.2863 & $4.1\times10^{-4}$ & 2.0932 & 19.0292 \\ \hline
\hline

\multicolumn{11}{|c|}{\hspace{-3 cm}(b) \textbf{Turning point 1}} \\ \hline
Case & $m_1$  &   $x_2$  & $\theta$     & $T$     & $e_{x_1}$      & $a_1$ & $T_1$  & $e_{x_2}$      & $a_2$  & $T_2$\\ \hline
0 & $10^{-6}$ & 1.6566 & 5.9063 & 12.6139 & $2.22\times10^{-2}$ & 1.0005 & 6.2910 & $-2.0\times10^{-6}$ & 1.6566 & 13.3975 \\ \hline
1 & $10^{-5}$ & 1.6567 & 5.9058 & 12.6117 & $2.22\times10^{-2}$ & 1.0005 & 6.2910 & $-2.0\times10^{-5}$ & 1.6567 & 13.3986 \\ \hline
2 & $10^{-4}$ & 1.6576 & 5.9027 & 12.6040 & $2.19\times10^{-2}$ & 1.0005 & 6.2909 & $-1.9\times10^{-4}$ & 1.6576 & 13.4096 \\ \hline
\hline

\multicolumn{11}{|c|}{\hspace{-3 cm}(c) \textbf{Turning point 2}} \\ \hline
Case & $m_1$  &   $x_2$  & $\theta$     & $T$     & $e_{x_1}$      & $a_1$ & $T_1$  & $e_{x_2}$      & $a_2$  & $T_2$\\ \hline
0 & $10^{-6}$ & 5.3705 & $6.2837$ & 80.6666 & $7.09\times10^{-1}$ & 2.0093 & 17.9043 & $1.95\times10^{-2}$ & 5.3726 & 78.2443 \\ \hline
1 & $10^{-5}$ & 4.4899 & $6.2828$ & 66.2322 & $6.68\times10^{-1}$ & 1.8066 & 15.2642 & $6.52\times10^{-2}$ & 4.5091 & 60.1621 \\ \hline
2 & $10^{-4}$ & 3.5097 & $6.2821$ & 54.6052 & $6.23\times10^{-1}$ & 1.6349 & 13.1409 & $1.70\times10^{-1}$ & 3.6136 & 43.1638 \\ \hline
\hline

\multicolumn{11}{|c|}{\hspace{-3 cm}(d) \textbf{Resonance} $\theta=2\pi$} \\ \hline
Case & $m_1$  &   $x_2$  & $\theta$     & $T$     & $e_{x_1}$      & $a_1$ & $T_1$  & $e_{x_2}$      & $a_2$  & $T_2$\\ \hline
0 & $10^{-6}$ & 3.7440 & $2\pi$ & 82.4959 & $7.14\times10^{-1}$ & 2.0395 & 18.3103 & $3.27\times10^{-1}$ & 4.1911 & 53.9098 \\ \hline
1 & $10^{-5}$ & 3.7440 & $2\pi$ & 82.4967 & $7.14\times10^{-1}$ & 2.0395 & 18.3103 & $3.27\times10^{-1}$ & 4.1913 & 53.9138 \\ \hline
2 & $10^{-4}$ & 3.7442 & $2\pi$ & 82.5317 & $7.14\times10^{-1}$ & 2.0399 & 18.3147 & $3.27\times10^{-1}$ & 4.1939 & 53.9661 \\ \hline
\end{tabular}
}
\caption{Some relevant data for the Poincar\'e type of RPOs computed for the first three sets of masses shown in Figure~\ref{fig:AE1_Poincare} at resonances $\theta\in \{\pi,2\pi\}$ and at the two turning points found. Solution data corresponds to the initial position of body $2$, $x_2$, angle between the two alignments, $\theta$ and period of the solution, $T$. Last six columns contain eccentricity, semi-major axis and period for body $1$ ($e_{x_1}$, $a_1$, $T_1$) and for body $2$ ($e_{x_2}$, $a_2$, $T_2)$, both by approximating their motion to a Keplerian orbit around body $0$.}
    \label{tab:AE1}
\end{table}

Again, we start the computation of the solutions placing $B_2$ far away from $B_1$, whose initial position is always $x_1=1$, and then, $B_2$ increasingly approaches to $B_1$, as described in Figure~\ref{fig:ic-poincare}. Figures~\ref{fig:AE1_Poincare} and \ref{fig:AE1_Poincare1} are equivalent to Figures~\ref{fig:AE0_Poincare} and \ref{fig:AE0_Poincare1} explained in the first section, in the sense that they provide the information about the same magnitudes. Recall that for a given set of masses, each point in the continuation curve is a different solution of the Three-Body Problem. Also, similarly as in previous section, Table~\ref{tab:AE1} collects key information of the computed solutions at special locations of the continuation curves, and the approximation given by two-body problem for Keplerian elements of bodies $1$ and $2$ with respect to $B_0$. Last case, for $m_1=m_2$, is not included in the Table since it also corresponds to the last case of previous section.

At the beginning of the continuation curves in $(x_2,\dot{y}_1)$, when $B_1$ and $B_2$ are far from each other, we observe the same kind of solution as in the previous examples: stable, circular solutions for the three bodies and for the four sets of masses considered, see first row of Figure~\ref{fig:AE1_Poincare}. 

As $x_2$ gets smaller, the \textbf{resonance} $\theta=\pi$ (special solution (a)), related to $T_s=T_2/2$, is crossed and we find unstable solutions with 2 real eigenvalues for all the set of masses considered. Notice that when the bodies are ordered according to their masses (as in Section~\ref{subsec:AE0}), unstable solutions related to this resonance only appeared for the cases in which the smallest mass was large enough. Now, in spite of using the same sets of masses, but being the exterior body the second most massive one, $m_2>m_1$, crossing the resonance $\theta=\pi$ always results in unstable solutions. These solutions are circular, as one can check by looking that the last row of images in Figure~\ref{fig:AE1_Poincare} and at the values of $e_x$ for this special case collected in Table~\ref{tab:AE1}.

After this resonance, we found stable, circular or almost circular, solutions as $B_2$ keeps approaching $B_1$ and their velocities start to increase. When $x_2$ is close to the value $1.59$ at which the resonance $\theta=2\pi$ would occur for circular orbits of bodies $1$ and $2$ around body $0$, see Remark~\ref{rmk:resonances}, we find again a \textbf{turning point} (special solution (b)). Similarly as in Section~\ref{subsec:AE0}, at this turning point the velocity of $B_2$ gets its maximum (see row two of Figure~\ref{fig:AE1_Poincare}) and, the total angular momentum its minimum (see row two of Figure~\ref{fig:AE1_Poincare1}), meaning that $B_1$ and $B_2$ are close to a collision.

Notice that in this case, approaching the turning point does not produce a change in the stability of the solutions, nor in the eccentricity of the orbit of $B_2$. Before and after the turning point, the solutions for $B_2$ correspond to nearly circular orbits. However, eccentricity of $B_1$ is significantly affected by this approach, see last row of Figure~\ref{fig:AE1_Poincare}. Again, we observe the similarity between the profiles of the curves $(x_2,\dot{y}_1)$ and $(x_2,e_{x_1})$ due to Equation~\eqref{eq:ecvec}.  After this first turning point, $\dot{y}_2$ starts decreasing, while the variation of $\theta$ gets slow as the continuation curve evolves. 

We find a \textbf{second turning point} (special solution (c)) when $B_2$ is getting far away from $B_1$, that is sharper as lower is the value of $m_1$, check values at Table~\ref{tab:AE1}. In this second turning point, $\dot{y}_2$ reaches a minimum and the total energy its highest values, being up to the order of $-10^{-5}$. Again, we recall that if the energy reaches a null value, the motion is not bounded any more. Approaching this second turning point affects the eccentricity of body $2$, that starts growing. In fact, the solution of the Three-Body Problem based on an alignment at the (special solution (d)) \textbf{resonant angle} $\theta=2\pi$ is given again for very eccentric orbits. Therefore, again, our families of Poincar\'e RPOs transit from orbits of first species to orbits of the second species.

Finally, the continuations vanish when bodies $1$ and $2$ describe very eccentric orbits being far away from each other and the solutions approach null values of the total energy of the system.

\paragraph{Some trajectories}

\begin{figure}[h!]
    \centering
    \setlength{\tabcolsep}{2pt}
    \renewcommand{\arraystretch}{0.75}
    \setlength{\arrayrulewidth}{0.4pt}
    \setlength{\doublerulesep}{1.2pt}

    \begin{adjustbox}{max totalsize={\textwidth}{0.9\textheight},center}
    \begin{tabular}{@{}
        >{\centering\arraybackslash}m{0.03\textwidth}||
        >{\centering\arraybackslash}m{0.225\textwidth}||
        *{4}{>{\centering\arraybackslash}m{\mainfigw}}
        >{\centering\arraybackslash}m{0.01\textwidth}
        @{}}

        $m _1$ & &
        (a) & (b) & (c) & (d)
        \\ \hline

         \multirow{2}{*}[-3em]{\rowlab{$10^{-6}$}}
        &
        \multirow{2}{*}[0em]{\maincurveplot{6}{AE1}{AE1_mapa_all}}
        &
        \mainplot{1}{AE1}{coppeq0}
        &
        \mainplot{4}{AE1}{coppeq0}
        &
        \mainplot{5}{AE1}{coppeq0}
        &
        \mainplot{7}{AE1}{coppeq0}
        \\

        &&
        \zoomplot{1}{AE1}{coppeq0}
        &
        \zoomplot{4}{AE1}{coppeq0}
        &
        \zoomplot{5}{AE1}{coppeq0}
        &
        \zoomplot{7}{AE1}{coppeq0}
        & \rotatebox[origin=r]{-90}{\footnotesize zoom-in}
        \\

        \hline

        \multirow{2}{*}[-3em]{\rowlab{$10^{-5}$}}
        &
        \multirow{2}{*}[0em]{\maincurveplot{13}{AE1}{AE1_mapa_all}}
        &
        \mainplot{1}{AE1}{coppeq1}
        &
        \mainplot{4}{AE1}{coppeq1}
        &
        \mainplot{5}{AE1}{coppeq1}
        &
        \mainplot{7}{AE1}{coppeq1}
        \\

        &&
        \zoomplot{1}{AE1}{coppeq1}
        &
        \zoomplot{4}{AE1}{coppeq1}
        &
        \zoomplot{5}{AE1}{coppeq1}
        &
        \zoomplot{7}{AE1}{coppeq1}
        & \rotatebox[origin=r]{-90}{\footnotesize zoom-in}
        \\

        \hline

         \multirow{2}{*}[-3em]{\rowlab{$10^{-4}$}}
        &
        \multirow{2}{*}[0em]{\maincurveplot{20}{AE1}{AE1_mapa_all}}
        &
        \mainplot{1}{AE1}{coppeq2}
        &
        \mainplot{4}{AE1}{coppeq2}
        &
        \mainplot{5}{AE1}{coppeq2}
        &
        \mainplot{6}{AE1}{coppeq2}
        \\

        &&
        \zoomplot{1}{AE1}{coppeq2}
        &
        \zoomplot{4}{AE1}{coppeq2}
        &
        \zoomplot{5}{AE1}{coppeq2}
        &
        \zoomplot{6}{AE1}{coppeq2}
        & \rotatebox[origin=r]{-90}{\footnotesize zoom-in}
        \\

        & 
    {\footnotesize
      \gpshortline[gpzero]{TE }%
      \gpshortline[gptwo]{EH }%
      \gpshortline[gpfour]{TH }%
      \gpshortline[gpsix]{CS}%
    }%
    & \multicolumn{4}{r}{\footnotesize\gplinemarked[gpsun]{$B_0$ }\gplinemarked[gpjup]{$B_1$ }\gplinemarked[gpsat]{$B_2$ }}
    \end{tabular}
    \end{adjustbox}

    \caption{Poincar\'e solutions for the four cases shown in Figure~\ref{fig:AE1_Poincare}, with $m_1$ in the first column. Columns (a)--(d) correspond to the selected solutions marked in the continuation curves of the second row. Bigger images show the trajectories in the $(x,y)$ plane of the three bodies; in blue body $2$, in orange body $1$, and in green body $0$. Small images are zoom-in body $0$.}
    \label{fig:AE1_tras}
\end{figure}

Figure~\ref{fig:AE1_tras} includes the trajectories of the three bodies for the four special solutions marked with letters in the third row of Figure~\ref{fig:AE1_Poincare} and in Table~\ref{tab:AE1}. 
The (a)--(d) columns show the time propagation of the trajectories of the three bodies during one period, $T$, for the cases of masses $m_1\in\{10^{-6},10^{-5},10^{-4}\}$; we skip last case ($m_1=m_2=10^{-3}$) since it was already included in previous section. 

Again, we observe that for solution (a) resonance $\theta=\pi$, $B_1$ completes one round and a half, while $B_0$ and $B_2$ complete only a half one. For solutions corresponding to the first and second turning points, (b) and (c), orbits of bodies are almost closed, especially in the second turning point. Then, the corresponding angle $\theta$ is close to $2\pi$, as one can check at Table~\ref{tab:AE1}. In the second turning point, the eccentricity of body $1$ is so high that crosses the orbit of body $2$. This close encounter, typically avoided in classical Poincar\'e solutions, may be the cause of the change in the sense of the continuation curve. 

As observed in previous section, after the first turning point, where the interaction among bodies $1$ and $2$ becomes more relevant, the approximation of some Keplerian elements for their orbits around $B_0$ (collected in Table~\ref{tab:AE1}) differs from the real orbits computed in the frame of the Three-Body Problem.

Finally, notice that the computed solutions of the TBP with different set of masses, at the solution (d) resonance $\theta=2\pi$ present the same configuration for bodies $1$ and $2$, while the orbit of body $0$ differs among cases. That is the same effect observed previously in Figure~\ref{fig:AE0_tras}. This tendency limit for body $1$ to a very eccentric orbit inside a nearly circular one for $B_2$ with $a_2\rightarrow \infty$ was theoretically described by Hadjidemetriou \cite{hadjidemetriou1976A}. Notice that $B_2$ does not describe a circular orbit around the center of mass, in fact its Keplerian eccentricity ($e_{x_2}\approx0.33$) agrees with its TBP orbit, however it \textit{seems} close to being circular due to its size with a center different from the center of mass nor $B_0$.

\subsection{Hill solutions}\label{sec:Hill}
Let's now study Hill type of solutions, also known as satellite or Lunar solutions. Configuration of these solutions is based on a hierarchical scheme, see Figure~\ref{fig:qbs_2}, where body $2$ orbits body $1$ that, at the same time, orbits body $0$, as in a star-planet-satellite system.

Then, we could think about this problem of three bodies as two decoupled two-body problems; one for the star-planet system ($B_0$ and $B_1$) and other for the planet-satellite system ($B_1$ and $B_2$). This way of considering the decouple is acceptable as a first approximation where masses of the bodies satisfies the hierarchy ($m_0\gg m_1 \gg m _2$) and bodies $1$ and $2$ are sufficiently close.

Different from the procedure followed in the previous section for Poincar\'e type of solutions, we now begin the continuation of the families of RPOs to the TBP by placing $B_1$ and $B_2$ very close to each other, as shown in Figure~\ref{fig:ic-hill} where the initial conditions are included.

\begin{figure}[h]
\centering
\begin{minipage}{0.55\textwidth}
\centering
 \includegraphics[width=\linewidth]{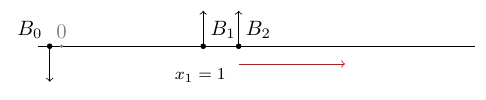}
\end{minipage}
\hfill
\begin{minipage}{0.35\textwidth}
\centering
\begin{tabular}{|c|c|c|c|}
\hline
$x_1$ & $x_2$ & $\dot{y}_1$ & $\dot{y}_2$\\
\hline \hline
$1$ & $1.001$ & $1.003$ & $0.7$ \\ \hline
\end{tabular}
\end{minipage}
\caption{Left, scheme of the initial positions and directions of motion for the three bodies in Hill regime; gray $0$ marks the origin of coordinates and the red arrow shows the sense of the continuation. Right, initial conditions taken for the continuation. Notice that $x_0$ and $\dot{y}_0$ are given by conservation of the center of mass.}
\label{fig:ic-hill}
\end{figure}

In this configuration, with bodies $1$ and $2$ sufficiently close, $B_2$ is gravitationally bounded to $B_1$, meaning that the energy for their Keplerian motion must be negative, see for example \cite{Pollard66}. Energy per unit of mass for their relative Keplerian motion is given by,
\begin{equation}
    h_{1,2} = \frac{\dot{q}_{1,2}^2}{2} - \frac{m_1+m_2}{q_{1,2}},
    \label{eq:h12}
\end{equation}
being $q_{1,2}$ and $\dot{q}_{1,2}$ their relative position and velocity, respectively. Notice that gravitational constant $\mathcal{G}$ is omitted in this expression since we take its value as the unity. 

Although $h_{1,2}$ is a constant of motion of the Two-Body Problem, this is not true for the Three-Body Problem. However, as we will explain, it helps us to understand how $B_2$ (the satellite) that starts orbiting $B_1$ (the planet), eventually gets apart from it and orbits directly $B_0$ (the star), becoming a circumstellar body.

\subsubsection{From satellite to circumstellar solutions}
In this section of Hill solutions, starting with the typical configuration of star-planet-satellite, 
we take $m_1=10^{-4}$, $m_2\in \{10^{-7},10^{-6},10^{-5},10^{-4}\}$ and $m_0=1-m_1-m_2$.

\begin{figure}[h!]
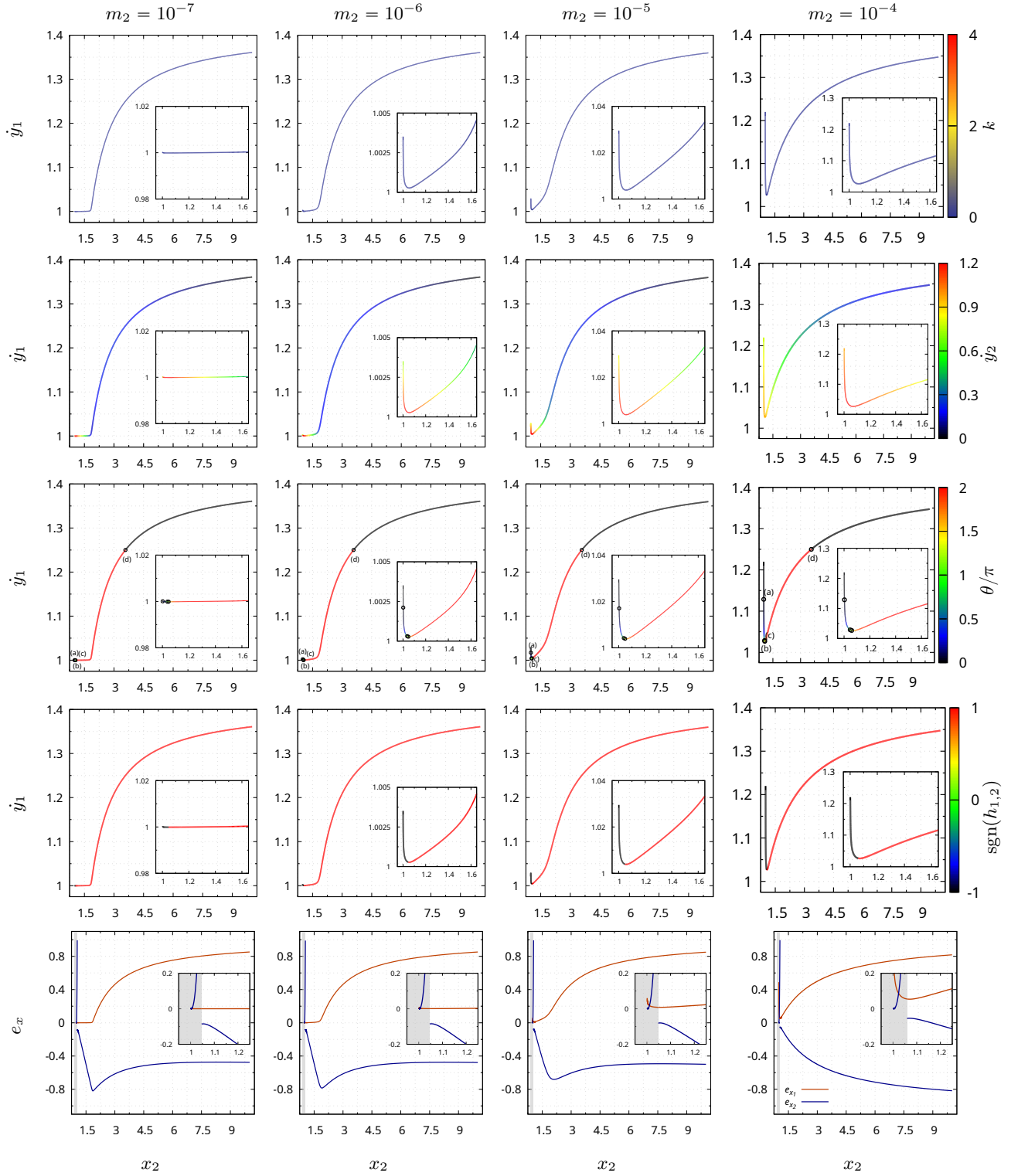

    \centering
    \footnotesize
    \setlength{\tabcolsep}{3pt}
    \renewcommand{\arraystretch}{1.5}

    \begin{adjustbox}{max totalsize={\textwidth}{0.98\textheight},center}
    \begin{tabular}{@{}
        >{\centering\arraybackslash}m{0.025\textwidth}
        *{3}{>{\centering\arraybackslash}m{\maincurvefigw}}
        >{\centering\arraybackslash}m{\rightmaincurvefigw}
        @{}
        >{\centering\arraybackslash}m{0.025\textwidth}}

        &
        {\footnotesize $m_2=10^{-7}$} &
        {\footnotesize $m_2=10^{-6}$} &
        {\footnotesize $m_2=10^{-5}$} &
        {\footnotesize $m_2=10^{-4}$}
        \\

        \rowlab{$\dot{y}_1$}
        
        &
        \maincurveplot{5}{AE3}{AE3_mapa_all}
        &
        \maincurveplot{12}{AE3}{AE3_mapa_all}
        &
        \maincurveplot{19}{AE3}{AE3_mapa_all}
        &
        \rightmaincurveplot{26}{AE3}{AE3_mapa_all}
        &
        
        \\

        \rowlab{$\dot{y}_1$}
        &
        \maincurveplot{1}{AE3}{AE3_mapa_all}
        &
        \maincurveplot{8}{AE3}{AE3_mapa_all}
        &
        \maincurveplot{15}{AE3}{AE3_mapa_all}
        &
        \rightmaincurveplot{22}{AE3}{AE3_mapa_all}
        & \rowlab{$\dot{y}_2$}
        
        \\
        
        \rowlab{$\dot{y}_1$}
        &
        \maincurveplot{4}{AE3}{AE3_mapa_all}
        &
        \maincurveplot{11}{AE3}{AE3_mapa_all}
        &
        \maincurveplot{18}{AE3}{AE3_mapa_all}
        &
        \rightmaincurveplot{25}{AE3}{AE3_mapa_all}
        &
        \rowlab{$\theta/\pi$}
        
        \\
        
        \rowlab{$\dot{y}_1$}
        &
        \maincurveplot{29}{AE3}{AE3_mapa_all}
        &
        \maincurveplot{30}{AE3}{AE3_mapa_all}
        &
        \maincurveplot{31}{AE3}{AE3_mapa_all}
        &
        \rightmaincurveplot{32}{AE3}{AE3_mapa_all}
        & \rowlab{$\mathrm{sgn}(h _{1,2})$}
        \\

        \rowlab{$e_x$}
        &
        \eccentrycurveplot{7}{AE3}{AE3_mapa_all}
        &
        \eccentrycurveplot{14}{AE3}{AE3_mapa_all}
        &
        \eccentrycurveplot{21}{AE3}{AE3_mapa_all}
        &
        \eccentrycurveplot{28}{AE3}{AE3_mapa_all}
        \\
        
        &
        {\footnotesize $x _2$} &
        {\footnotesize $x _2$} &
        {\footnotesize $x _2$} &
        {\footnotesize $x _2$} &
        \\
    \end{tabular}
    \end{adjustbox}

    \caption{Hill families of solutions computed for $m_0=1-m_1-m_2$, $m_1=10^{-4}$ and $m_2\in\{10^{-7},10^{-6},10^{-5},10^{-4}\}$ for each of the four columns. The first four rows show the curve continuation in $(x_2,\dot{y}_1)$, colored according to the stability type, to velocity of the second body $\dot{y}_2$, to the angle $\theta/\pi$ and to the sign of the energy of relative system for bodies $1$ and $2$, Equation~\eqref{eq:h12}. Last row shows the variation of the $x$ component of the eccentricity vectors of the bodies $1$ and $2$, with respect to body $0$, as the second gets apart from the first. Shadowed area in last row cover the solutions for which $h_{1,2}<0$. Insets show a zoom of the bottom-left region. }
    \label{fig:AE3_Hill}
\end{figure}

\begin{figure}[h!]
    \centering
    \footnotesize
    \setlength{\tabcolsep}{3pt}
    \renewcommand{\arraystretch}{1.5}

    \begin{adjustbox}{max totalsize={\textwidth}{0.98\textheight},center}
    \begin{tabular}{@{}
        >{\centering\arraybackslash}m{0.025\textwidth}
        *{2}{>{\centering\arraybackslash}m{\maincurvefigw}}
        >{\centering\arraybackslash}m{\rightmaincurvefigw}
        >{\centering\arraybackslash}m{\rightmaincurvefigw}
        @{}
        >{\centering\arraybackslash}m{0.025\textwidth}}

        &
        {\footnotesize $m_2=10^{-7}$} &
        {\footnotesize $m_2=10^{-6}$} &
        {\footnotesize $m_2=10^{-5}$} &
        {\footnotesize $m_2=10^{-4}$}
        \\

        \rowlab{$\dot{y}_1$}
        &
        \maincurveplot{2}{AE3}{AE3_mapa_all}
        &
        \maincurveplot{9}{AE3}{AE3_mapa_all}
        &
    \multicolumn{1}{
        >{\centering\arraybackslash}m{\rightmaincurvefigw}||
    }{
        \rightmaincurveplot{16}{AE3}{AE3_mapa_all}
    }
        &
        \rightmaincurveplot{23}{AE3}{AE3_mapa_all}
        & 
        \rowlabten{$\mathcal{H}\times10^{-5}$}
        
        \\
        
        \rowlab{$\dot{y}_1$}
        &
        \maincurveplot{3}{AE3}{AE3_mapa_all}
        &
        \maincurveplot{10}{AE3}{AE3_mapa_all}
        &
    \multicolumn{1}{
        >{\centering\arraybackslash}m{\rightmaincurvefigw}||
    }{
        \rightmaincurveplot{17}{AE3}{AE3_mapa_all}
    }
        &
        \rightmaincurveplot{24}{AE3}{AE3_mapa_all}
        &
        \rowlabten{$L\times10^{-4}$}
        \\
        
        &
        {\footnotesize $x _2$} &
        {\footnotesize $x _2$} &
        {\footnotesize $x _2$} &
        {\footnotesize $x _2$} &
        \\
    \end{tabular}
    \end{adjustbox}
 \caption{Hill families of solutions computed for $m_0=1-m_1-m_2$, $m_1=10^{-4}$ and $m_2\in\{10^{-7},10^{-6},10^{-5},10^{-4}\}$ for each of the four columns. The rows show the curve continuation in $(x_2,\dot{y}_1)$, first colored according to according to the value of the total energy and second according to the value of the total angular momentum.}
    \label{fig:AE3_Hill1}
\end{figure}

For each of these sets of masses we compute the relative periodic solutions satisfying the conditions described in Section~\ref{sec:NumComp}. In this case we start with $x_2$ close to --but larger than-- $x_1=1$. In Figure~\ref{fig:AE3_Hill} we show the continuation curves in $(x_2,\dot{y}_1)$ colored according to different characteristics or magnitudes; the stability type (totally elliptic, TE, partially hyperbolic, EH, totally hyperbolic, TH, or complex saddle, CS), the velocity $\dot{y}_2$, the normalized angle between alignments $\theta/\pi$, the sign of relative energy among bodies $1$ and $2$, $h_{1,2}$. As a last row, initial $x$-component of the eccentricity vectors for bodies $1$ and $2$ as approximated by their Keplerian motions is included; for $B_1$ with respect to $B_0$, and for $B_2$ with respect to $B_1$ meanwhile $h_{1,2}<0$, and with respect to $B_0$ when $h_{1,2}>0$. Figure~\ref{fig:AE3_Hill1} shows the continuation curves colored by the total energy $\mathcal{H}$ and the total angular momentum $L$. The understanding of particular elements of the image explanations will be aid by information collected in Table~\ref{tab:AE3}. Again, the table contains data of the computed solutions at specific locations of the continuation curves, as well as, the approximation through two-body problem of Keplerian elements of bodies $1$ and $2$.

Let's first notice that the shape of the continuation curves has changed. Now, for the first two cases, for which $B_2$ (the satellite) has the smallest masses of the ones considered, $B_1$ starts with velocity close to the unity, what, recalling Equation~\eqref{eq:ecvec} means that $B_1$ is describing a circular or almost circular orbit around $B_0$. When the mass of the satellite, $B_2$, is comparable to the one of the planet, $B_1$, we still find satellite type of solutions, but the velocity of $B_1$ is larger than $1$. Regardless of this, vast majority of solutions computed in this configuration are stable, i.e. they have no real eigenvalues.

The similarities with Poincar\'e type of solutions are basically that when bodies $1$ and $2$ are close, velocity of $B_2$ gets its maximum. Also, at the beginning of the continuation, the variation of the angle $\theta$ is fluid until the solution is close to happen at $\theta=2\pi$, ($\theta/\pi=2$), that happens for high values of velocity and eccentricity of $B_1$. Of course, total angular momentum is not null for any solution, reaching its minimum at the closest approaches between bodies $1$ and $2$, and total energy is always negative, what is required for the motion of the three bodies to be bounded.

In spite of the fact that the total energy must be negative, we find a change in the sign of the energy of relative system for bodies $1$ and $2$, $h_{1,2}$ given by Equation~\eqref{eq:h12} and shown in the fourth row of Figure~\ref{fig:AE3_Hill}. Meanwhile these two bodies are close enough, and given the difference of masses, $B_2$ orbits $B_1$ resulting in $h_{1,2}<0$; this situation is marked by a gray shadow in the last row of Figure~\ref{fig:AE3_Hill}, where we see that the eccentricity of body 2 in its orbit around body 1, starts being small.

However, as $x_2$ increases, orbit of $B_2$ around $B_1$ gets more and more eccentric rapidly and energy $h_{1,2}$ eventually gets positive. As a result, through this continuation of RPOs to the Three-Body Problem, $B_2$ that started being a satellite of $B_1$ gets free of its gravitational capture and begins to orbit $B_0$ while intersecting orbit of $B_1$, as will be shown later. If body $0$ is a star, bodies $1$ and $2$ are said to be circumstellar objects. From this moment, orbital elements of the motion of $B_2$ are computed with respect to $B_0$, what explains the jump in the graph for $e_{x_2}$. When $B_2$ begins its circumstellar orbit, it describes a low eccentric orbit that gets more eccentric as $x_2$ continues increasing. 


For these families of Hill solutions, special solutions collected in Table~\ref{tab:AE3} correspond to (a) some of the \textbf{first computed solutions} of the continuation curves, with bodies $1$ and $2$ are very close. Then, (b) the computed solutions at \textbf{resonance} $\theta=\pi$. Notice that in these two situations the value of $x_2$ is very small, and $B_2$ is still gravitationally bounded to $B_1$, what corresponds to the typical satellite behavior. The third special solution we include in the table corresponds to (c) the \textbf{first unbounded solution} among bodies $1$ and $2$, identified by the first solution with $h_{1,2}>0$. Notice that this solution is very close to the minimum reached by $\dot{y}_1$.

At this point, it is worth recalling the approximation of Hill radius for $B_1$,
\begin{equation}
    r_{H_1} \approx a_1 \bigg(\frac{m_1}{3m_0}\bigg)^{1/3},
    \label{eq:rHill}
\end{equation}
that for our cases corresponds to an approximated value of $0.032$. This radius gives an approximation of the region around body $1$ where its own gravity dominates over the tidal pull of the much larger body, $B_0$. Looking at the $x_2$ value in Table~\ref{tab:AE3} for the first unbounded solution among bodies $1$ and $2$, we find a reasonable agreement among these two results, since these values correspond to distance to $B_1$ of $\sim 0.04-0.05$. Keep in mind that Equation~\eqref{eq:h12} is not a real constant of motion, nor Equation~\eqref{eq:rHill} is an exact value, but both are useful constructions to understand the limits of the gravitational bounded motion among the planet, $B_1$, and the satellite, $B_2$.

Last special solution included in the table corresponds, as in previous examples, to the (d) \textbf{resonance} $\theta=2\pi$; that is, when the three bodies align in positions and velocities at their respective initial positions and velocities. Again, this situation is given for very eccentric orbits.

\begin{table}[ht]
\centering
{\footnotesize
\begin{tabular}{|c|c|c|c|c||c|c|c|c|c|c|}
\hline\multicolumn{5}{|c||}{\textbf{Solution of TBP}} & \multicolumn{6}{|c|}{\textbf{Approximation of keplerian elements}} \\ \hline
\multicolumn{11}{|c|}{\hspace{-3 cm}(a) \textbf{Closest solutions among bodies} $1$ \textbf{and} $2$} \\ \hline
Case & $m_2$  &   $x_2$  & $\theta$     & $T$     & $e_{x_1}$      & $a_1$ & $T_1$  & $e_{x_2}$      & $a_2$  & $T_2$\\ \hline
1 & $10^{-7}$ & 1.0021 & 0.0587 & 0.05868 & $4.4\times10^{-4}$ & 1.0000 & 6.2832 & $1.6\times10^{-4}$ & 0.0021 & 0.05905 \\ \hline
2 & $10^{-6}$ & 1.0019 & 0.0520 & 0.05199 & $4.5\times10^{-3}$ & 1.0000 & 6.2834 & $1.2\times10^{-4}$ & 0.0019 & 0.05227 \\ \hline
3 & $10^{-5}$ & 1.0030 & 0.0983 & 0.09837 & $3.5\times10^{-2}$ & 1.0012 & 6.2946 & $4.5\times10^{-4}$ & 0.0030 & 0.09951 \\ \hline
4 & $10^{-4}$ & 1.0030 & 0.0710 & 0.07121 & $2.8\times10^{-1}$ & 1.0819 & 7.0707 & $2.5\times10^{-4}$ & 0.0030 & 0.07174 \\ \hline \hline

\multicolumn{11}{|c|}{\hspace{-3 cm}(b) \textbf{Resonance} $\theta=\pi$} \\ \hline
Case & $m_2$  &   $x_2$  & $\theta$     & $T$     & $e_{x_1}$      & $a_1$ & $T_1$  & $e_{x_2}$      & $a_2$  & $T_2$\\ \hline
1 & $10^{-7}$ & 1.0371 & 3.1562 & 3.13597 & $9.3\times10^{-5}$ & 1.0000 & 6.2832 & $5.4\times10^{-1}$ & 0.0521 & 7.4657 \\ \hline
2 & $10^{-6}$ & 1.0372 & 3.1522 & 3.13359 & $9.2\times10^{-4}$ & 1.0000 & 6.2832 & $5.3\times10^{-1}$ & 0.0521 & 7.4334 \\ \hline
3 & $10^{-5}$ & 1.0384 & 3.1593 & 3.15611 & $8.7\times10^{-3}$ & 1.0001 & 6.2839 & $5.4\times10^{-1}$ & 0.0544 & 7.5932 \\ \hline
4 & $10^{-4}$ & 1.0475 & 3.1425 & 3.23996 & $5.9\times10^{-2}$ & 1.0035 & 6.3165 & $5.5\times10^{-1}$ & 0.0686 & 7.9763 \\ \hline \hline

\multicolumn{11}{|c|}{\hspace{-3 cm}(c) \textbf{First unbounded solution} $h_{1,2}>0$} \\ \hline
Case & $m_2$  &   $x_2$  & $\theta$     & $T$     & $e_{x_1}$      & $a_1$ & $T_1$  & $e_{x_2}$      & $a_2$  & $T_2$\\ \hline

1 & $10^{-7}$ & 1.0467 & 3.9575 & 3.95837 & $8.7\times10^{-5}$ & 1.0000 & 6.2832 & $-8.6\times10^{-2}$ & 1.0390 & 6.6545 \\ \hline
2 & $10^{-6}$ & 1.0468 & 3.9515 & 3.95481 & $8.6\times10^{-4}$ & 1.0000 & 6.2832 & $-8.5\times10^{-2}$ & 1.0392 & 6.6561\\ \hline
3 & $10^{-5}$ & 1.0482 & 3.9469 & 3.97342 & $8.1\times10^{-3}$ & 1.0001 & 6.2838 & $-8.1\times10^{-2}$ & 1.0414 & 6.6776 \\ \hline
4 & $10^{-4}$ & 1.0593 & 3.8991 & 4.07422 & $5.5\times10^{-2}$ & 1.0030 & 6.3121 & $-5.4\times10^{-2}$ & 1.0561 & 6.8198 \\ \hline \hline

\multicolumn{11}{|c|}{\hspace{-3 cm}(d) \textbf{Resonance} $\theta=2\pi$} \\ \hline
Case & $m_2$  &   $x_2$  & $\theta$     & $T$     & $e_{x_1}$      & $a_1$ & $T_1$  & $e_{x_2}$      & $a_2$  & $T_2$\\ \hline
1 & $10^{-7}$ & 3.5806 & $2\pi$ & 21.76990 & $5.6\times10^{-1}$ & 1.4646 & 11.1368 & $-5.6\times10^{-1}$ & 2.7179 & 28.1548 \\ \hline
2 & $10^{-6}$ & 3.5805 & $2\pi$ & 21.76970 & $5.6\times10^{-1}$ & 1.4646 & 11.1367 & $-5.6\times10^{-1}$ & 2.7179 & 28.1545 \\ \hline
3 & $10^{-5}$ & 3.5814 & $2\pi$ & 21.77920 & $5.6\times10^{-1}$ & 1.4649 & 11.1399 & $-5.6\times10^{-1}$ & 2.7187 & 28.1671 \\ \hline
4 & $10^{-4}$ & 3.5832 & $2\pi$ & 21.80230 & $5.6\times10^{-1}$ & 1.4653 & 11.1455 & $-5.6\times10^{-1}$ & 2.7205 & 28.1945 \\ \hline

\end{tabular}
}
\caption{Some relevant data for the Hill type of RPOs found for the sets of masses shown in Figure~\ref{fig:AE3_Hill} for bodies $1$ and $2$ very proximate, at resonances $\theta=\pi,2\pi$ and when body $2$ gets free from gravitational bound of body $1$. Solution data corresponds to the initial position of body $2$, $x_2$, angle between the two alignments, $\theta$ and period of the solution, $T$. Last six columns contain Keplerian elements eccentricity, semi-major axis and period for body $1$ ($e_{x_1}$, $a_1$, $T_1$) with respect to $B_0$, and for body $2$ ($e_{x_2}$, $a_2$, $T_2)$, with respect to $B_1$ meanwhile $h_{1,2}<0$ and with respect to $B_0$ otherwise.}
\label{tab:AE3}
\end{table}

\paragraph{Some trajectories}
\begin{figure}[h!]
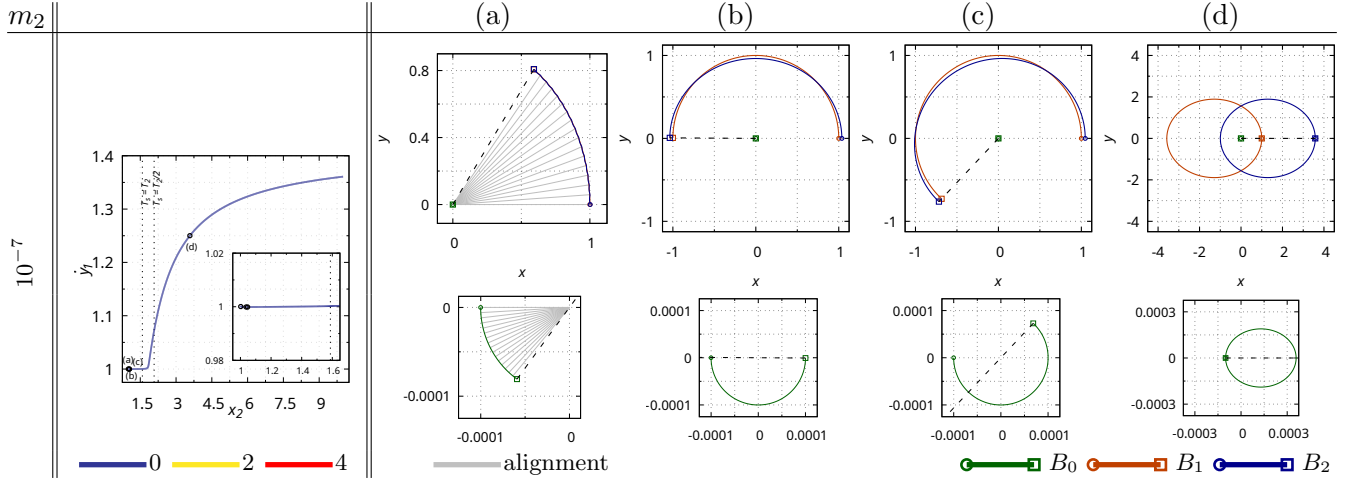

    \centering
    \setlength{\tabcolsep}{2pt}
    \renewcommand{\arraystretch}{0.75}
    \setlength{\arrayrulewidth}{0.4pt}
    \setlength{\doublerulesep}{1.2pt}

    \begin{adjustbox}{max totalsize={\textwidth}{0.9\textheight},center}
    \begin{tabular}{@{}
        >{\centering\arraybackslash}m{0.03\textwidth}||
        >{\centering\arraybackslash}m{0.225\textwidth}||
        *{4}{>{\centering\arraybackslash}m{\mainfigw}}
        @{}}

        $m _2$ & &
        (a) & (b) & (c) & (d)
        \\ \hline

         \multirow{2}{*}[-3em]{\rowlab{$10^{-7}$}}
        &
        \multirow{2}{*}[0em]{\maincurveplot{6}{AE3}{AE3_mapa_all}}
        &
        \mainplot{1}{AE3}{copHill0}
        &
        \mainplot{4}{AE3}{copHill0}
        &
        \mainplot{3}{AE3}{copHill0}
        &
        \mainplot{6}{AE3}{copHill0}
        \\

        &
        &
        \zoomplot{1}{AE3}{copHill0}
        &
        \zoomplot{4}{AE3}{copHill0}
        &
        \zoomplot{3}{AE3}{copHill0}
        &
        \zoomplot{6}{AE3}{copHill0}
        \\

        & 
    {\footnotesize
      \gpshortline[gpzero]{TE }%
      \gpshortline[gptwo]{EH }%
      \gpshortline[gpfour]{TH }%
      \gpshortline[gpsix]{CS}%
    }%
    & \multicolumn{1}{r}{\footnotesize\gpline[gptheta]{alignment}}&\multicolumn{3}{r}{\footnotesize\gplinemarked[gpsun]{$B_0$ }\gplinemarked[gpjup]{$B_1$ }\gplinemarked[gpsat]{$B_2$ }}
    \end{tabular}
    \end{adjustbox}

    \caption{Hill solutions for the first case in Figure~\ref{fig:AE3_Hill}; $m_1=10^{-4}$, $m_2=10^{-7}$, $m_0=1-m_1-m_2$. Columns (a)--(d) correspond to the selected solutions marked in the continuation curves of the second row. Bigger images show the trajectories in the $(x,y)$ plane of the three bodies; in blue body $2$, in orange body $1$, and in green body $0$. Small images are zoom-in body $0$.}
    \label{fig:AE3_cop0}
\end{figure}

In Figure~\ref{fig:AE3_cop0} four RPOs to the TBP are shown; images below each large image contains a zoom of the trajectory for body $0$. They correspond to special locations of the first case in Figure~\ref{fig:AE3_Hill}. Column (a) contains one of the first solutions in the continuation, when $B_2$ is gravitationally bounded to $B_1$. Notice that, in this case the value of the period of the solution, as well as the value of $\theta$, is very small. For this reason we have marked the angle corresponding to the relative periodic solution (in light gray line), and then we include the trajectory for 15 more periods, in order to show the coupling among these two bodies. In the following cases, trajectories correspond to time propagation from $t=0$ to $t=T$. Column (b) corresponds to the solution at resonance $\theta=\pi$, for which the three bodies describe half an orbit. Column (c) corresponds to the first computed solution for which $h_{1,2}>0$, that means, the gravitational decoupling of bodies $1$ and $2$. Finally, column (d) includes the trajectories for the $\theta=2\pi$ resonance, at which bodies $1$ and $2$ describe almost symmetric orbits with respect to body $0$, that is slightly displaced from the center of mass, where the origin is set. Notice that this solution agrees with the values estimated for the eccentricities of the bodies, but not for their semi-major axis and periods. As explained previously Keplerian approximation of the orbital elements fails for very eccentric Three-Body Problem solutions.

It is remarkable that along the continuation family of Hill relative periodic solutions of the Three‑Body Problem, the configuration can transit from a typical satellite regime (as in solution (a)) to a central body with two circumstellar companions orbiting around it, as in solution (d). The latter configuration should not be confused with planetary Poincar\'e solutions. As we have explained, such a continuation from Hill‑type  solutions to Poincar\'e‑type  solutions cannot be achieved without an appropriate regularization of the system.

Here we must recall that Poincar\'e solutions of the Three-Body Problem are typically constructed from a Keplerian (two body) problem that is perturbed by a third one. Classical proofs 
assume the bodies to be far away from collision.
In other words, the orbits of the two less massive bodies in classical Poincar\'e solutions do not cross, while the orbits computed in this section, continued from typical Hill configuration, cross. We believe these solutions should not be interpreted as typical planetary solutions, but as a natural dynamical drift of a small body (like an asteroid) that eventually can be trapped by a more massive one (like a planet). What is more, this transit, as explained in the Introduction, has been observed in our Solar System \cite{Granvik2012,delaFuenteMarcos2020}.

\subsection{Binary solutions}\label{sec:Binary}
In the last two sections, we find different kinds of solutions that have a common characteristic; a hierarchy among the masses. This resulted in just one central body around which the other two revolve, directly or through a combined orbit (planet+satellite). In this section we deal with configurations in which two of the bodies have large masses (bodies $0$ and $1$) meanwhile the mass of the third one (body $2$) is small, comparatively. This situation would correspond, for example, to a binary star system and a planet.

In this configuration, schematically described in Figure~\ref{fig:ic-binary}, we will distinguish between two possible situations. If the velocity of $B_2$ is comparable to that of $B_1$, we find the configuration in which the planet revolves around the orbit defined by the two stars, performing what is called a circumbinary orbit. This situation is studied in Section~\ref{subsec:AE2}. However, if $B_2$ has small velocity, the configuration is very different, in fact we find that the sense of rotation of body $2$ eventually changes and at some point in the continuation begins to describe a retrograde orbit with respect to the binary system. This situation is addressed in Section~\ref{subsec:AE4}, which covers solutions where the planet ($B_2$) orbits only one of the stars ($B_1$, for proximity). This situation would correspond again to a circumstellar orbit. 

\begin{figure}[h]
\centering
\begin{minipage}{0.55\textwidth}
\centering
 \includegraphics[width=\linewidth]{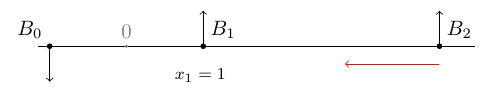}
\end{minipage}
\hfill
\begin{minipage}{0.35\textwidth}
\centering
\begin{tabular}{|c|c|c|c|c|}
\hline
$x_1$ & $x_2$ & $\dot{y}_1$ & \multicolumn{2}{c|}{$\dot{y}_2$}\\
\hline \hline
$1$ & $8$ & $0.7$ & $0.7$ & $0.3$\\ \hline
\end{tabular}
\end{minipage}
\caption{Left, scheme of the initial positions and directions of motion for the three bodies in Binary regime; gray $0$ marks the origin of coordinates and the red arrow shows the sense of the continuation. Right, initial conditions taken for the continuation. Each value for $\dot{y}_2$ corresponds to a different configuration; see text. Notice that $x_0$ and $\dot{y}_0$ are given by conservation of the center of mass.}
\label{fig:ic-binary}
\end{figure}

\subsubsection{Circumbinary planet}\label{subsec:AE2}


In this situation we consider the configuration and conditions given in Figure~\ref{fig:ic-binary} with $\dot{y}_2=0.7$. As we will see, we find relative periodic solutions describing a binary star system and a circumbinary planet. Notice that this situation is often treated as two decoupled two-body problems; one for the binary ($B_1$ relative to $B_0$) and other for the planet ($B_2$ relative to the orbit of the binary).

The masses considered for the less massive body are  $m_2\in\{10^{-6},10^{-5},10^{-4},10^{-3}\}$, while the masses of bodies $0$ and $1$ are computed as $m_0=m_1=(1-m_2)/2$, such that the total mass is one.

\begin{figure}[h!]
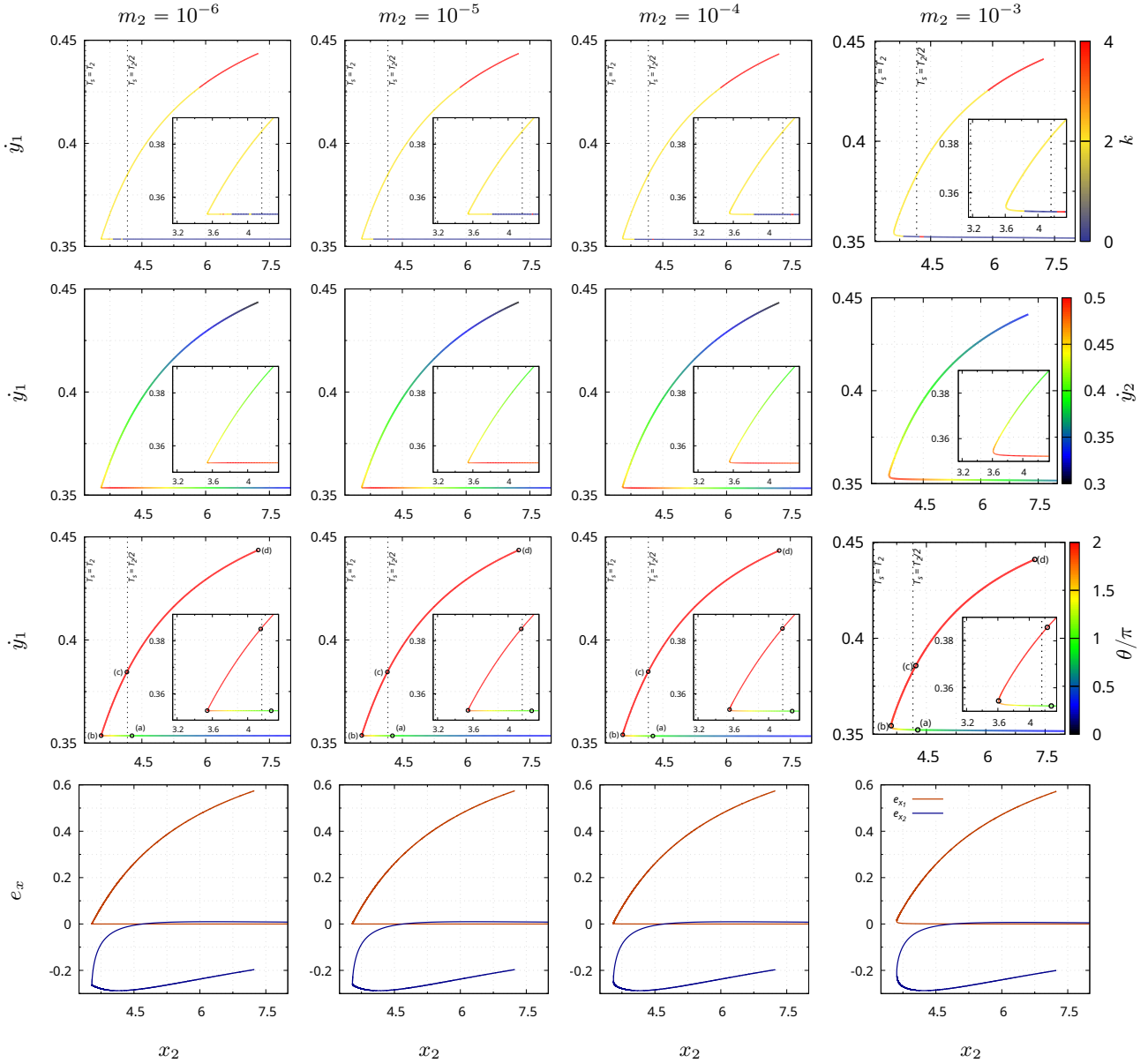

    \centering
    \footnotesize
    \setlength{\tabcolsep}{3pt}
    \renewcommand{\arraystretch}{1.5}

    \begin{adjustbox}{max totalsize={\textwidth}{0.98\textheight},center}
    \begin{tabular}{@{}
        >{\centering\arraybackslash}m{0.025\textwidth}
        *{3}{>{\centering\arraybackslash}m{\maincurvefigw}}
        >{\centering\arraybackslash}m{\rightmaincurvefigw}
        @{}
        >{\centering\arraybackslash}m{0.025\textwidth}}

        &
        {\footnotesize $m_2=10^{-6}$} &
        {\footnotesize $m_2=10^{-5}$} &
        {\footnotesize $m_2=10^{-4}$} &
        {\footnotesize $m_2=10^{-3}$}
        \\

        \rowlab{$\dot{y}_1$}
        
        &
        \maincurveplot{5}{AE2}{AE2_mapa_all}
        &
        \maincurveplot{12}{AE2}{AE2_mapa_all}
        &
        \maincurveplot{19}{AE2}{AE2_mapa_all}
        &
        \rightmaincurveplot{26}{AE2}{AE2_mapa_all}
        &
        
        \\

        \rowlab{$\dot{y}_1$}
        &
        \maincurveplot{1}{AE2}{AE2_mapa_all}
        &
        \maincurveplot{8}{AE2}{AE2_mapa_all}
        &
        \maincurveplot{15}{AE2}{AE2_mapa_all}
        &
        \rightmaincurveplot{22}{AE2}{AE2_mapa_all}
        & \rowlab{$\dot{y}_2$}
        
        \\
        
        \rowlab{$\dot{y}_1$}
        &
        \maincurveplot{4}{AE2}{AE2_mapa_all}
        &
        \maincurveplot{11}{AE2}{AE2_mapa_all}
        &
        \maincurveplot{18}{AE2}{AE2_mapa_all}
        &
        \rightmaincurveplot{25}{AE2}{AE2_mapa_all}
        &
        \rowlab{$\theta/\pi$}
        
        \\

        \rowlab{$e_x$}
        &
        \eccentrycurveplot{7}{AE2}{AE2_mapa_all}
        &
        \eccentrycurveplot{14}{AE2}{AE2_mapa_all}
        &
        \eccentrycurveplot{21}{AE2}{AE2_mapa_all}
        &
        \eccentrycurveplot{28}{AE2}{AE2_mapa_all}
        \\
        
        &
        {\footnotesize $x _2$} &
        {\footnotesize $x _2$} &
        {\footnotesize $x _2$} &
        {\footnotesize $x _2$} &
        \\
    \end{tabular}
    \end{adjustbox}

    \caption{Binary families of solutions computed for $m_0=m_1=(1-m_2)/2$ and $m_2\in\{10^{-6},10^{-5},10^{-4},10^{-3}\}$ for each of the four columns. First three rows show the curve continuation in $(x_2,\dot{y}_1)$, colored according to the stability type, to the velocity of the second body $\dot{y}_2$, and to the angle $\theta/\pi$. Last row shows the initial $x$ component of the eccentricity vectors of $B_1$ and $B_2$, with respect to $B_0$ and to $B_0+B_1$, respectively. Insets show a zoom of the bottom-left region.}
\label{fig:AE2_Binary}
\end{figure}

\begin{figure}[h!]
    \centering
    \footnotesize
    \setlength{\tabcolsep}{3pt}
    \renewcommand{\arraystretch}{1.5}

    \begin{adjustbox}{max totalsize={\textwidth}{0.98\textheight},center}
    \begin{tabular}{@{}
        >{\centering\arraybackslash}m{0.025\textwidth}
        *{3}{>{\centering\arraybackslash}m{\maincurvefigw}}
        >{\centering\arraybackslash}m{\rightmaincurvefigw}
        @{}
        >{\centering\arraybackslash}m{0.025\textwidth}}

        &
        {\footnotesize $m_2=10^{-6}$} &
        {\footnotesize $m_2=10^{-5}$} &
        {\footnotesize $m_2=10^{-4}$} &
        {\footnotesize $m_2=10^{-3}$}
        \\

        \rowlab{$\dot{y}_1$}
        &
        \maincurveplot{2}{AE2}{AE2_mapa_all}
        &
        \maincurveplot{9}{AE2}{AE2_mapa_all}
        &
        \maincurveplot{16}{AE2}{AE2_mapa_all}
        &
        \rightmaincurveplot{23}{AE2}{AE2_mapa_all}
        & 
        \rowlabten{$\mathcal{H}\times10^{-2}$}
        
        \\
        
        \rowlab{$\dot{y}_1$}
        &
        \maincurveplot{3}{AE2}{AE2_mapa_all}
        &
        \maincurveplot{10}{AE2}{AE2_mapa_all}
        &
        \maincurveplot{17}{AE2}{AE2_mapa_all}
        &
        \rightmaincurveplot{24}{AE2}{AE2_mapa_all}
        &
        \rowlab{$L$}
        \\
        
        &
        {\footnotesize $x _2$} &
        {\footnotesize $x _2$} &
        {\footnotesize $x _2$} &
        {\footnotesize $x _2$} &
        \\
    \end{tabular}
    \end{adjustbox}

\caption{Binary families of solutions computed for $m_0=m_1=(1-m_2)/2$ and $m_2\in\{10^{-6},10^{-5},10^{-4},10^{-3}\}$ for each of the four columns. The rows show the curve continuation in $(x_2,\dot{y}_1)$, first colored according to according to the value of the total energy and second according to the value of the total angular momentum.}
    \label{fig:AE2_Binary1}
\end{figure}

\begin{table}[h!]
\centering
{\footnotesize
\begin{tabular}{|c|c|c|c|c||c|c|c|c|c|c|}
\hline\multicolumn{5}{|c||}{\textbf{Solution of TBP}} & \multicolumn{6}{|c|}{\textbf{Approximation of keplerian elements}} \\ \hline
\multicolumn{11}{|c|}{\hspace{-3 cm}(a) \textbf{Resonance} $\theta=\pi$} \\ \hline
Case & $m_2$  &   $x_2$  & $\theta$     & $T$     & $e_{x_1}$      & $a_1$ & $T_1$  & $e_{x_2}$      & $a_2$  & $T_2$\\ \hline
0 & $10^{-6}$ & 4.2672 & $\pi$ & 26.6316 & $7.3\times10^{-7}$ & 1.0000 & 6.2832 & $-1.60\times10^{-2}$ & 4.2661 & 55.3639 \\ \hline
1 & $10^{-5}$ & 4.2672 & $\pi$ & 26.6343 & $7.3\times10^{-6}$ & 1.0000 & 6.2832 & $-1.60\times10^{-2}$ & 4.2661 & 55.3638 \\ \hline
2 & $10^{-4}$ & 4.2672 & $\pi$ & 26.6609 & $7.3\times10^{-5}$ & 1.0000 & 6.2835 & $-1.64\times10^{-2}$ & 4.2660 & 55.3628 \\ \hline
3 & $10^{-3}$ & 4.2822 & $\pi$ & 26.8511 & $7.3\times10^{-4}$ & 1.0000 & 6.2863 & $-1.92\times10^{-2}$ & 4.2806 & 55.6459 \\ \hline
\hline

\multicolumn{11}{|c|}{\hspace{-3 cm}(b) \textbf{Turning point}} \\ \hline
Case & $m_2$  &   $x_2$  & $\theta$     & $T$     & $e_{x_1}$      & $a_1$ & $T_1$  & $e_{x_2}$      & $a_2$  & $T_2$\\ \hline
0 & $10^{-6}$ & 3.5414 & $6.2590$ & 35.4910 & $2.75\times10^{-4}$ & 1.0000 & 6.2832 & $-2.60\times10^{-1}$ & 3.3167 & 37.9530 \\ \hline
1 & $10^{-5}$ & 3.5449 & $6.2131$ & 35.3946 & $9.17\times10^{-4}$ & 1.0000 & 6.2832 & $-2.54\times10^{-1}$ & 3.3293 & 38.1686 \\ \hline
2 & $10^{-4}$ & 3.5574 & $6.0898$ & 35.1834 & $3.25\times10^{-3}$ & 1.0000 & 6.2836 & $-2.40\times10^{-1}$ & 3.3640 & 38.7668 \\ \hline
3 & $10^{-3}$ & 3.6094 & $5.8010$ & 34.9828 & $1.22\times10^{-2}$ & 1.0001 & 6.2877 & $-2.10\times10^{-1}$ & 3.4569 & 40.3839 \\ \hline
\hline

\multicolumn{11}{|c|}{\hspace{-3 cm}(c) \textbf{Minimum of} $e_{x_2}$} \\ \hline
Case & $m_2$  &   $x_2$  & $\theta$     & $T$     & $e_{x_1}$      & $a_1$ & $T_1$  & $e_{x_2}$      & $a_2$  & $T_2$\\ \hline
0 & $10^{-6}$ & 4.1482 & $6.2832$ & 48.0826 & $1.82\times10^{-1}$ & 1.0344 & 6.6105 & $-2.88\times10^{-1}$ & 3.8313 & 47.1200 \\ \hline
1 & $10^{-5}$ & 4.1491 & $6.2829$ & 48.1009 & $1.83\times10^{-1}$ & 1.0345 & 6.6113 & $-2.88\times10^{-1}$ & 3.8322 & 47.1356 \\ \hline
2 & $10^{-4}$ & 4.1574 & $6.2803$ & 48.2670 & $1.84\times10^{-1}$ & 1.0352 & 6.6180 & $-2.88\times10^{-1}$ & 3.8399 & 47.2779 \\ \hline
3 & $10^{-3}$ & 4.2313 & $6.2574$ & 49.7608 & $1.99\times10^{-1}$ & 1.0413 & 6.6793 & $-2.87\times10^{-1}$ & 3.9088 & 48.5556 \\ \hline

\multicolumn{11}{|c|}{\hspace{-3 cm}(d) \textbf{Closest to resonance} $\theta=2\pi$} \\ \hline
Case & $m_2$  &   $x_2$  & $\theta$     & $T$     & $e_{x_1}$      & $a_1$ & $T_1$  & $e_{x_2}$      & $a_2$  & $T_2$\\ \hline
0 & $10^{-6}$ & 7.2380 & $6.2832$ & 128.027 & $5.74\times10^{-1}$ & 1.4925 & 11.456 & $-1.97\times10^{-1}$ & 6.9686 & 115.584 \\ \hline
1 & $10^{-5}$ & 7.2382 & $6.2832$ & 128.034 & $5.74\times10^{-1}$ & 1.4924 & 11.456 & $-1.97\times10^{-1}$ & 6.9687 & 115.588 \\ \hline
2 & $10^{-4}$ & 7.2383 & $6.2829$ & 128.034 & $5.74\times10^{-1}$ & 1.4918 & 11.449 & $-1.97\times10^{-1}$ & 6.9678 & 115.565 \\ \hline
3 & $10^{-3}$ & 7.2387 & $6.2804$ & 128.039 & $5.72\times10^{-1}$ & 1.4854 & 11.380 & $-2.00\times10^{-1}$ & 6.9590 & 115.346 \\ \hline
\end{tabular}
}
\caption{Some relevant data for the Binary type of RPOs found for the four sets of masses shown in Figure~\ref{fig:AE2_Binary} at resonance $\theta=\pi$, at the turning point, at the minimum of the $x-$component of eccentricity vector of body $2$ and at the closest computed solution to resonance $\theta=2\pi$. Solution data corresponds to the initial position of body $2$, $x_2$, angle between the two alignments, $\theta$ and period of the solution, $T$. Last six columns contain eccentricity, semi-major axis and period for body $1$ ($e_{x_1}$, $a_1$, $T_1$) and for body $2$ ($e_{x_2}$, $a_2$, $T_2)$, by approximating their motion to a Keplerian orbit around $B_0$ and around the binary formed by $B_0$ and $B_1$, respectively.}
    \label{tab:AE2}
\end{table}

Figure~\ref{fig:AE2_Binary} shows the continuation curves $(x_2,\dot{y}_1)$ again colored according to the stability type (totally elliptic, TE, partially hyperbolic, EH, totally hyperbolic, TH, or complex saddle, CS), to the value of $\dot{y}_2$, to the value of the normalized angle between alignments $\theta/\pi$ and the curves $(x_2,e_{x_1})$ and $(x_2,e_{x_2})$. Notice that in this configuration, Keplerian orbital elements of $B_1$ are computed with respect to $B_0$ and the ones of $B_2$ are computed with respect to the binary formed by bodies $0$ and $1$. Figure~\ref{fig:AE2_Binary1} contains the continuation curves $(x_2,\dot{y}_1)$ colored according to the total energy and the total angular momentum of the system. And Table~\ref{tab:AE2} collects information about some special solutions of the continuation curves.

Although this kind of configurations differs from the ones in the last two sections in the sense that now we have two central bodies, we find some similarities between the cases. First, we observe a similar profile of the continuation curves to the ones found in Section~\ref{sec:Poincare} for Poincar\'e solutions. Again, the minimum of the total angular momentum is given in the closest approach between bodies $1$ and $2$, and, of course, $\mathcal{H}<0$ for all the solutions. Besides, we observe that meanwhile $B_1$ and $B_2$ are far from each other, they describe circular stable orbits. When \textbf{resonance} $\theta=\pi$ is crossed (special solution (a)) we find a range of complex saddle solutions, that is wider as $m_2$ is larger. We also observe a \textbf{turning point}, special solution (b), in the continuation curve when bodies $1$ and $2$ come too close. Before the turning point the eccentricity of $B_2$ becomes non-zero, reaching a minimum shortly after it. Notice that the \textbf{minimum of eccentricity} $e_{x_2}$ is significantly pronounced (special solution (c)), again softer as $m_2$ is larger. From this situation, eccentricity and velocity of body $1$ start growing and the variation of the angle $\theta$ becomes slow. 

Also, following the same reasoning made in Section~\ref{sec:Poincare}, the fact that velocity of $B_1$ remains constant meanwhile bodies $1$ and $2$ are not close, has its explanation through Equation~\eqref{eq:ecvec}. In these solutions we keep on considering as initial condition $x_1=1$, however, relative position among bodies $0$ and $1$ is now $x_{1,0}\approx2$, and not close to $1$ as it happened in the Poincar\'e solutions. Therefore, for body $1$ to describe circular orbits, $e_1=0$, relative velocity of $B_1$ with respect to $B_0$ must be, following Equation~\eqref{eq:rel-x-vy}, $\dot{y}_{1,0}\approx\sqrt{1/2}$. Now, taking into account that bodies forming the binary, $B_0$ and $B_1$, move at the same velocity in opposite direction, $\dot{y}_1=\dot{y}_{1,0}/2\approx0.3536$. That corresponds to the constant value found.

This affects the positions at which body $2$ is expected to be found at resonances $\theta=\pi$ and $\theta=2\pi$, that are included as vertical lines in the Figure~\ref{fig:AE2_Binary}. Following Remark~\ref{rmk:resonances}, since now semi-major axis of the relative orbit $a_{1,0}\approx2$, resonance $\theta=\pi$, related with $T_2=3T_1$, for bodies $1$ and $2$ describing circular or almost circular orbits, is expected to happen at $x_2=\sqrt[3]{3^22^3} \approx 4.16$. Similarly, resonance $\theta=2\pi$, related with $T_2=2T_1$, for bodies $1$ and $2$ describing circular or almost circular orbits, is expected to happen at $x_2=\sqrt[3]{2^22^3} \approx 3.17$. Close to this position --but without reaching it-- we find again a turning point.

In spite of the similarities, we also find some differences with respect to the previous cases in Section~\ref{sec:Poincare}. The first one is related to the loss of stability of the RPOs to the Three-Body Problem before the turning point. Notice that at these unstable solutions --before the turning point (and not on it)-- is where the maximum of velocity of $B_2$ is found, look at the second row of images in Figure~\ref{fig:AE2_Binary}. Besides, here, we do not only find sets of unstable solutions with $2$ real eigenvalues (partially hyperbolic, EH), but also we find a wide set of solutions with $4$ real eigenvalues (totally hyperbolic, TH). Another difference is that we have not been able to compute the solution of these three bodies based on an alignment at $\theta=2\pi$. Considering this, the special solution (d) collected in Table~\ref{tab:AE2} corresponds to the last computed orbit that happens to an angle extremely close to $2\pi$.

\paragraph{Some trajectories}
 
\begin{figure}[h!]
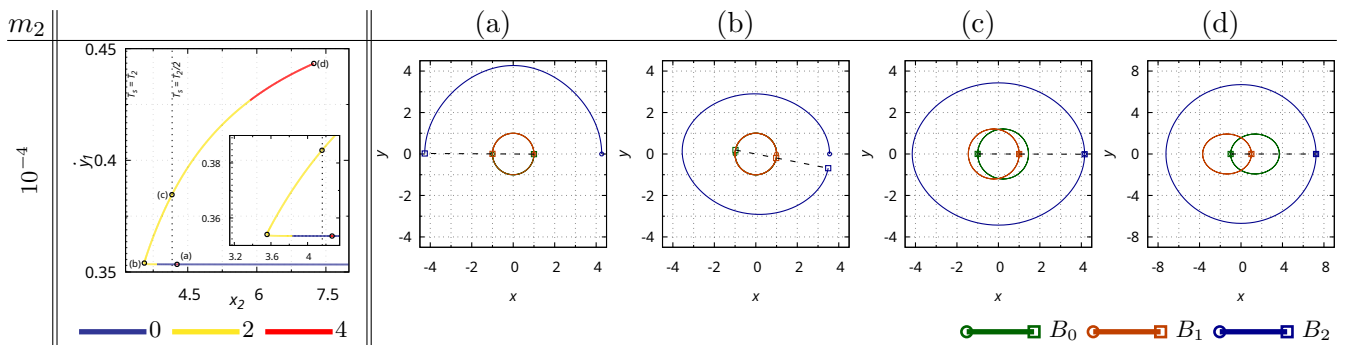

    \centering
    \setlength{\tabcolsep}{2pt}
    \renewcommand{\arraystretch}{0.75}
    \setlength{\arrayrulewidth}{0.4pt}
    \setlength{\doublerulesep}{1.2pt}

    \begin{adjustbox}{max totalsize={\textwidth}{0.9\textheight},center}
    \begin{tabular}{@{}
        >{\centering\arraybackslash}m{0.03\textwidth}||
        >{\centering\arraybackslash}m{0.225\textwidth}||
        *{4}{>{\centering\arraybackslash}m{\mainfigw}}
        @{}}

        $m _2$ & &
        (a) & (b) & (c) & (d)
        \\ \hline

\rowlab{$10^{-4}$}
        &
        \maincurveplot{20}{AE2}{AE2_mapa_all}
        &
        \mainplot{6}{AE2}{copbin2}
        &
        \mainplot{3}{AE2}{copbin2}
        &
        \mainplot{1}{AE2}{copbin2}
        &
        \mainplot{5}{AE2}{copbin2}
        \\

        & 
    {\footnotesize
      \gpshortline[gpzero]{TE }%
      \gpshortline[gptwo]{EH }%
      \gpshortline[gpfour]{TH }%
      \gpshortline[gpsix]{CS}%
    }%
    & \multicolumn{4}{r}{\footnotesize\gplinemarked[gpsun]{$B_0$ }\gplinemarked[gpjup]{$B_1$ }\gplinemarked[gpsat]{$B_2$ }}
    \end{tabular}
    \end{adjustbox}
    
\caption{Binary solutions for the third case in Figure~\ref{fig:AE2_Binary}; $m_0=m_1=(1-m_2)/2$, $m_2=10^{-4}$. Columns (a)--(d) correspond to the selected solutions marked in the continuation curves of the second row. Images show the trajectories in the $(x,y)$ plane of the three bodies; in blue body $2$, in orange body $1$, and in green body $0$.}
\label{fig:AE2_cop2}
\end{figure}

Figure~\ref{fig:AE2_cop2} shows the trajectories in $(x,y)$ of the three bodies computed for the set of masses $m_0=m_1=(1-m_2)/2$, with $m_2=10^{-4}$, that corresponds to the third case in Figure~\ref{fig:AE2_Binary}.

Looking at these solutions we observe how meanwhile body $2$ is far from the binaries, as in (a) when resonance $\theta=\pi$ is crossed, the three bodies describe circular orbits. In fact, bodies $0$ and $1$ move in the same orbit, a circle of radius unity centered at the origin, with a phase shift of $180^\circ$. In panel (b), configuration of the turning point is included, displaying a solution that is approaching the periodic orbit. There, eccentricity of $B_2$ is already significant, as shown in Table~\ref{tab:AE2}, meanwhile $B_0$ and $B_1$ still share the same circular unitary orbit. After the turning point the orbits of the binaries decouple. When $B_2$ reaches the minimum of its $x$ component eccentricity vector (maximum eccentricity element), panel (c), we can easily distinguish the orbits of the two more massive bodies and see how $B_2$ is really close to resonance $\theta=2\pi$. The closest computed solution to this resonance is included in panel (d), where we observe a solution for quite eccentric and large orbits. In agreement with the orbital parameters collected in Table~\ref{tab:AE2}.

\subsubsection{Circumstellar planet}\label{subsec:AE4}

Initial conditions for the kind of solutions discussed in this section are the same as in previous section, collected in Figure~\ref{fig:ic-binary}, except for the initial velocity of $B_2$, that now is lower; $\dot{y}_2=0.3$ instead of $\dot{y}_2=0.7$. Besides, the set of masses are maintained;  $m_2\in\{10^{-6},10^{-5},10^{-4},10^{-3}\}$ and $m_0=m_1=(1-m_2)/2$. With such an apparent small change, configuration found differs significantly from before, since the smallest body, the planet $B_2$ is found to orbit only one of the stars, $B_1$, instead of orbiting the binary system formed by bodies $0$ and $1$.

Similarly as in previous section, Figures~\ref{fig:AE4_Binary} and \ref{fig:AE4_Binary1} show the continuation curves in $(x_2,\dot{y}_1)$ colored according to different magnitudes; the stability type (totally elliptic, TE, partially hyperbolic, EH, totally hyperbolic, TH, or complex saddle, CS), velocity of body $2$ $\dot{y}_2$, the normalized angle $\theta/\pi$, total energy and total angular momentum of the system. Besides, Figure~\ref{fig:AE4_Binary} also includes the initial $x$ component of the eccentricity vector of bodies $1$ and $2$ computed with respect to body $0$ and body $1$, respectively. 

Again, we start placing $B_2$ far away from $B_1$, and the solutions are continued as $x_2$ is diminished. Notice that the profile of the continuation curves has changed from previous case. In fact, first solutions in the continuation are found for a value of $\theta/\pi\approx2$, that is, the first solutions computed are close to resonance $\theta=2\pi$. At this state, when $B_2$ is far from the two stars, velocity of $B_1$ is high and starts to descend as body $2$ approaches. 

Then, the continuation curves in $(x_2,\dot{y}_1)$ show a descending profile for $\dot{y}_1$ as $x_2$ gets smaller up to a position at which velocity of $B_1$ stabilizes at a value close to $0.35$, see Figure~\ref{fig:AE4_Binary}. As already explained, this value is related with bodies $0$ and $1$ describing circular orbits of semi-major axis equal to $1$. In fact, we can check this if we look at the eccentricity plot at the last row of Figure~\ref{fig:AE4_Binary}.

When this region of \textit{plateau} for $\dot{y}_1$ is reached, the value of the angle between alignments, $\theta$, starts to descend fluidly as $B_2$ approaches $B_1$, as shown in the third row of Figure~\ref{fig:AE4_Binary}. In the same figure, row two, we can check that for all this part of the continuation, solutions have $\dot{y}_2$ low but positive. It is remarkable that eventually, when $B_2$ is close enough to $B_1$, its velocity $\dot{y}_2$ gets negative, meaning that $B_2$ revolves clockwise, although in all previous cases $B_2$ revolves counter-clockwise, as $B_1$ does. As $B_2$ keeps on approaching position of $B_1$, its eccentricity tends to diminish. Only for solutions at which bodies $1$ and $2$ are very close, we find circular stable solutions. 

Differently from the smooth variation of trajectories found in previous families, the ones included in this section, show an extreme dynamical richness in the sense that solutions of the same family display very different configurations, not always corresponding to any astronomical known system. Therefore in this last section we include a more extensive description of some of the computed trajectories for one of the set of masses.

\begin{figure}[h!]
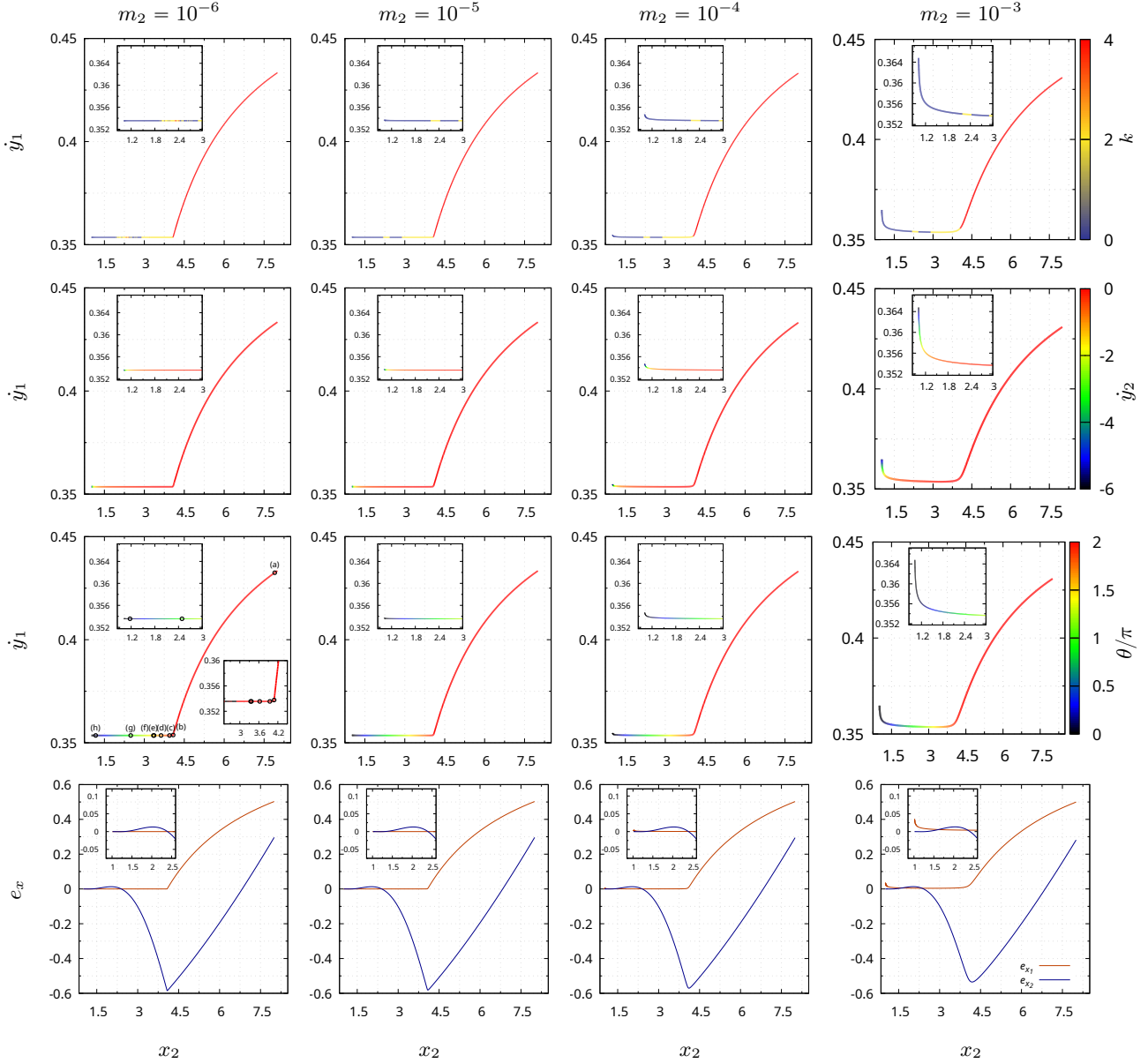

    \centering
    \footnotesize
    \setlength{\tabcolsep}{3pt}
    \renewcommand{\arraystretch}{1.5}

    \begin{adjustbox}{max totalsize={\textwidth}{0.98\textheight},center}
    \begin{tabular}{@{}
        >{\centering\arraybackslash}m{0.025\textwidth}
        *{3}{>{\centering\arraybackslash}m{\maincurvefigw}}
        >{\centering\arraybackslash}m{\rightmaincurvefigw}
        @{}
        >{\centering\arraybackslash}m{0.025\textwidth}}

        &
        {\footnotesize $m_2=10^{-6}$} &
        {\footnotesize $m_2=10^{-5}$} &
        {\footnotesize $m_2=10^{-4}$} &
        {\footnotesize $m_2=10^{-3}$}
        \\

        \rowlab{$\dot{y}_1$}
        
        &
        \maincurveplot{5}{AE4}{AE4_mapa_all}
        &
        \maincurveplot{12}{AE4}{AE4_mapa_all}
        &
        \maincurveplot{19}{AE4}{AE4_mapa_all}
        &
        \rightmaincurveplot{26}{AE4}{AE4_mapa_all}
        &
        
        \\

        \rowlab{$\dot{y}_1$}
        &
        \maincurveplot{1}{AE4}{AE4_mapa_all}
        &
        \maincurveplot{8}{AE4}{AE4_mapa_all}
        &
        \maincurveplot{15}{AE4}{AE4_mapa_all}
        &
        \rightmaincurveplot{22}{AE4}{AE4_mapa_all}
        & \rowlab{$\dot{y}_2$}
        
        \\
        
        \rowlab{$\dot{y}_1$}
        &
        \maincurveplot{4}{AE4}{AE4_mapa_all}
        &
        \maincurveplot{11}{AE4}{AE4_mapa_all}
        &
        \maincurveplot{18}{AE4}{AE4_mapa_all}
        &
        \rightmaincurveplot{25}{AE4}{AE4_mapa_all}
        &
        \rowlab{$\theta/\pi$}
        
        \\

        \rowlab{$e_x$}
        &
        \eccentrycurveplot{7}{AE4}{AE4_mapa_all}
        &
        \eccentrycurveplot{14}{AE4}{AE4_mapa_all}
        &
        \eccentrycurveplot{21}{AE4}{AE4_mapa_all}
        &
        \eccentrycurveplot{28}{AE4}{AE4_mapa_all}
        \\
        
        &
        {\footnotesize $x _2$} &
        {\footnotesize $x _2$} &
        {\footnotesize $x _2$} &
        {\footnotesize $x _2$} &
        \\
    \end{tabular}
    \end{adjustbox}

    \caption{Binary families of solutions computed for $m_0=m_1=(1-m_2)/2$ and $m_2\in\{10^{-6},10^{-5},10^{-4},10^{-3}\}$ for each of the four columns. First three rows show the curve continuation in $(x_2,\dot{y}_1)$, colored according to the stability type, to the velocity of the second body $\dot{y}_2$, and to the angle $\theta/\pi$. Last row shows the initial $x$ component of the eccentricity vectors of the bodies $1$ and $2$, with respect to $B_0$ and $B_1$, respectively, as the second approaches the first. }
\label{fig:AE4_Binary}
\end{figure}

\begin{figure}[h!]
    \centering
    \footnotesize
    \setlength{\tabcolsep}{3pt}
    \renewcommand{\arraystretch}{1.5}

    \begin{adjustbox}{max totalsize={\textwidth}{0.98\textheight},center}
    \begin{tabular}{@{}
        >{\centering\arraybackslash}m{0.025\textwidth}
        *{2}{>{\centering\arraybackslash}m{\maincurvefigw}}
        >{\centering\arraybackslash}m{\rightmaincurvefigw}
        >{\centering\arraybackslash}m{\rightmaincurvefigw}
        @{}
        >{\centering\arraybackslash}m{0.025\textwidth}}

        &
        {\footnotesize $m_2=10^{-6}$} &
        {\footnotesize $m_2=10^{-5}$} &
        {\footnotesize $m_2=10^{-4}$} &
        {\footnotesize $m_2=10^{-3}$}
        \\

        \rowlab{$\dot{y}_1$}
        &
        \maincurveplot{2}{AE4}{AE4_mapa_all}
        &
        \maincurveplot{9}{AE4}{AE4_mapa_all}
        &
    \multicolumn{1}{
        >{\centering\arraybackslash}m{\rightmaincurvefigw}||
    }{
        \rightmaincurveplot{16}{AE4}{AE4_mapa_all}
    }
        &
        \rightmaincurveplot{23}{AE4}{AE4_mapa_all}
        & 
        \rowlabten{$\mathcal{H}\times10^{-2}$}
        
        \\
        
        \rowlab{$\dot{y}_1$}
        &
        \maincurveplot{3}{AE4}{AE4_mapa_all}
        &
        \maincurveplot{10}{AE4}{AE4_mapa_all}
        &
    \multicolumn{1}{
        >{\centering\arraybackslash}m{\rightmaincurvefigw}||
    }{
        \rightmaincurveplot{17}{AE4}{AE4_mapa_all}
    }
        &
        \rightmaincurveplot{24}{AE4}{AE4_mapa_all}
        &
        \rowlab{$L$}
        \\
        
        &
        {\footnotesize $x _2$} &
        {\footnotesize $x _2$} &
        {\footnotesize $x _2$} &
        {\footnotesize $x _2$} &
        \\
    \end{tabular}
    \end{adjustbox}

    \caption{Binary families of solutions computed for $m_0=m_1=(1-m_2)/2$ and $m_2\in\{10^{-6},10^{-5},10^{-4},10^{-3}\}$ for each of the four columns. The rows show the curve continuation in $(x_2,\dot{y}_1)$, first colored according to according to the value of the total energy and second according to the value of the total angular momentum. }
    \label{fig:AE4_Binary1}
\end{figure}

\begin{table}[h!]
\centering
{\footnotesize
\begin{tabular}{@{}|l|c|c|c||c|c|c|c|c|c|@{}}
\hline & \multicolumn{3}{c||}{\textbf{Solution of TBP}} & \multicolumn{6}{c|}{\textbf{Approximation of keplerian elements}} \\ \hline
   &   $x_2$  & $\theta$     & $T$     & $e_{x_1}$      & $a_1$ & $T_1$  & $e_{x_2}$      & $a_2$  & $T_2$\\ \hline

(a) $\theta\approx2\pi$ & 7.9000 & 6.2832 & 49.7286 & $4.96\times 10^{-1}$ & 1.3270 & 9.6047 & $2.70\times10^{-1}$ & 7.4421 & 180.4004\\ \hline

(b) minimum $e_{x_2}$ & 4.0827 & 6.2799 & 17.7960 & $1.26\times 10^{-3}$ & 1.0000 & 6.2832 & $-5.85 \times 10^{-1}$ & 2.2971 & 30.9364\\ \hline

(c) Intermediate & 3.9468 & 5.9555 & 16.8452 & $1.43\times 10^{-5}$ & 1.0000 & 6.2832 & $-5.11 \times 10^{-1}$ & 2.3376 &  31.7578 \\ \hline

(d) First $\dot{y}_2<0$ & 3.6282 & 5.2654 & 14.8930 & $5.80\times 10^{-6}$ & 1.0000 & 6.2832 & $-3.42\times 10^{-1}$ & 2.3529 & 32.0697\\ \hline

(e) $\theta\approx3\pi/2$ & 3.3658 & 4.7591 & 13.4608 & $4.56\times 10^{-6}$ & 1.0000 & 6.2832 & $-2.26\times 10^{-1}$ & 2.2511 & 30.0118\\ \hline

(f) $\theta=3\pi/2$ & 3.3412 & 4.7133 & 13.3314 & $4.49\times 10^{-6}$ & 1.0000 & 6.2832 & $-2.16\times 10^{-1}$ & 2.2369 & 29.7274\\ \hline

(g) $\theta=\pi$ & 2.4880 & $\pi$ & 8.8823 & $3.93\times 10^{-6}$ & 1.0000 & 6.2832 & $-8.15\times 10^{-3}$ & 1.4879 & 16.1271\\ \hline

(h) circumstellar & 1.1714 & 0.2151 & 0.6085 & $1.05\times 10^{-5}$ & 1.0000 & 6.2832 & $-2.36\times 10^{-4}$ & 0.1714 & 0.6307\\ \hline
\end{tabular}
}
\caption{Some relevant data for the Binary type of RPOs found for the first case shown in Figure~\ref{fig:AE4_Binary} at eight special positions of the family. Solution data corresponds to the initial position of body $2$, $x_2$, angle between the two alignments, $\theta$ and period of the solution, $T$. Last six columns contain eccentricity, semi-major axis and period for body $1$ ($e_{x_1}$, $a_1$, $T_1$) and for body $2$ ($e_{x_2}$, $a_2$, $T_2)$, by approximating their motion to a Keplerian orbit around $B_0$ and around the binary formed by $B_0$ and $B_1$, respectively.}
\label{tab:AE4}
\end{table}

\paragraph{Some trajectories}

In Figure~\ref{fig:AE4_cop0} eight relative periodic solutions for the first case ($m_2=10^{-6}$) in Figure~\ref{fig:AE4_Binary} are included, large images for the time propagation equal to one period $T$, and small ones for ten periods $10T$. The solutions are designated by letters (a)-(h) that are also marked in the continuation curve included in the second column. Some relevant data for the eight selected solutions of this case is collected in Table~\ref{tab:AE4}.

Column (a) in Figure~\ref{fig:AE4_cop0} contains one of the first solutions of the continuation curve, for $\theta\approx 2\pi$, since it is very close to being a periodic orbit, we do not appreciate differences among the time propagation for $T$ or $10T$. Next solution, in column (b) corresponds to the minimum of $e_{x_2}$, that is found significantly displaced from previous solution, however, the angle of alignment is again close to $2\pi$. In this case, the unstable behavior of the RPO and the large eccentricity of $B_2$ makes its trajectory move further away than the expected pattern for a large time propagation. Column (c) shows an RPO where body $2$ describes loops around the binary, similarly as in column (d) that corresponds to the first solution with $\dot{y}_2<0$. From this solution, $B_2$ revolves clockwise.

In the lower part of Figure~\ref{fig:AE4_cop0}, column (e) contains a solution defined by an angle of alignment close to $3\pi/2$. It can be observed how the solution in (e) oscillates around the solution in column (f), that corresponds to $\theta=3\pi/2$, and therefore this an absolute periodic orbit. Also solution in column (g) is an absolute periodic orbit, since it corresponds to resonance $\theta=\pi$, in this situation we see body $2$ describe an orbit that resembles a peanut. Finally, for a position of $B_2$, the planet, very close to $B_1$, we observe in column (h) what could correspond to a circumstellar planet in a binary star system. This last relative periodic solution would be the only configuration along this family of solution that has a meaning in the known astronomical systems.

\begin{figure}[h!]
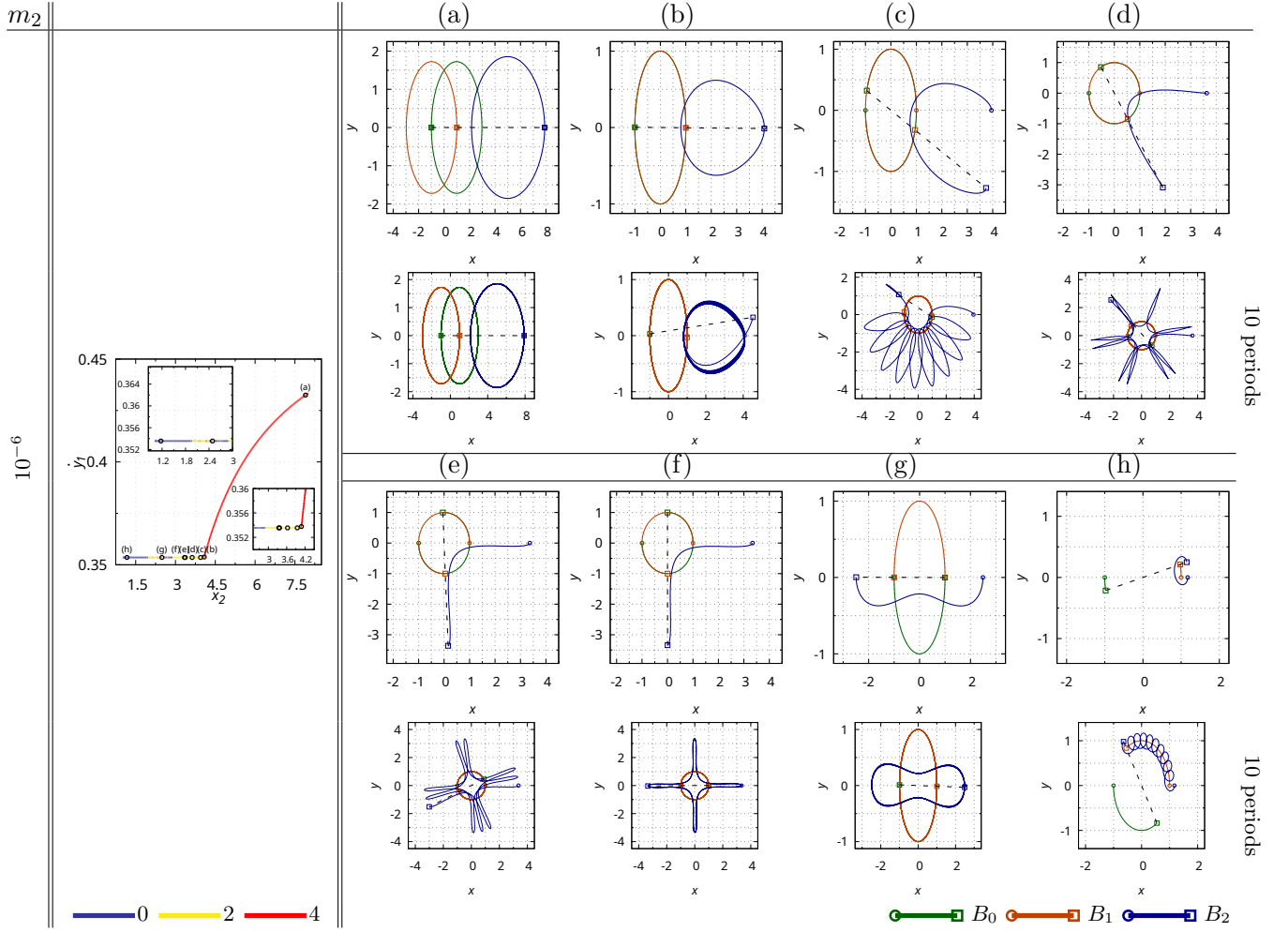

    \centering
    \setlength{\tabcolsep}{2pt}
    \renewcommand{\arraystretch}{0.75}
    \setlength{\arrayrulewidth}{0.4pt}
    \setlength{\doublerulesep}{1.2pt}

    \begin{adjustbox}{max totalsize={\textwidth}{0.9\textheight},center}
    \begin{tabular}{@{}
        >{\centering\arraybackslash}m{0.03\textwidth}||
        >{\centering\arraybackslash}m{0.225\textwidth}||
        *{4}{>{\centering\arraybackslash}m{\mainfigw}}
        >{\centering\arraybackslash}m{0.01\textwidth}
        @{}}

        $m _2$ & &
        (a) & (b) & (c) & (d)
        \\ \hline

         \multirow{5}{*}[-10em]{\rowlab{$10^{-6}$}}
        &
        \multirow{4}{*}[-8em]{\maincurveplot{6}{AE4}{AE4_mapa_all}}
        &
        \mainplot{1}{AE4}{copbin0_vy2b}
        &
        \mainplot{4}{AE4}{copbin0_vy2b}
        &
        \mainplot{5}{AE4}{copbin0_vy2b}
        &
        \mainplot{6}{AE4}{copbin0_vy2b}
        \\

        &
        &
        \iterTplot{1}{AE4}{copbin0_vy2b}
        &
        \iterTplot{4}{AE4}{copbin0_vy2b}
        &
        \iterTplot{5}{AE4}{copbin0_vy2b}
        &
        \iterTplot{6}{AE4}{copbin0_vy2b}
        & \rotatebox[origin=r]{-90}{\footnotesize $10$ periods}
        \\

        \cline{3-6}

         & 
         &
        (e) & (f) & (g) & (h) \\ 
        \cline{3-6}
        &
        &
        \mainplot{7}{AE4}{copbin0_vy2b}
        &
        \mainplot{8}{AE4}{copbin0_vy2b}
        &
        \mainplot{2}{AE4}{copbin0_vy2b}
        &
        \mainplot{3}{AE4}{copbin0_vy2b}
        \\

        &
        &
        \iterTplot{7}{AE4}{copbin0_vy2b}
        &
        \iterTplot{8}{AE4}{copbin0_vy2b}
        &
        \iterTplot{2}{AE4}{copbin0_vy2b}
        &
        \iterTplot{3}{AE4}{copbin0_vy2b}
        & \rotatebox[origin=r]{-90}{\footnotesize $10$ periods}
        \\

        & 
    {\footnotesize
      \gpshortline[gpzero]{TE }%
      \gpshortline[gptwo]{EH }%
      \gpshortline[gpfour]{TH }%
      \gpshortline[gpsix]{CS}%
    }%
    & \multicolumn{4}{r}{\footnotesize\gplinemarked[gpsun]{$B_0$ }\gplinemarked[gpjup]{$B_1$ }\gplinemarked[gpsat]{$B_2$ }}
    \end{tabular}
    \end{adjustbox}
    \caption{Binary solutions for the first case in Figure~\ref{fig:AE4_Binary}; $m_0=m_1=(1-m_2)/2$, $m_2=10^{-6}$. Columns (a)--(h) correspond to the selected solutions marked in the continuation curves of the second row. Large images show the trajectories during one period in the $(x,y)$ plane of the three bodies; in blue body $2$, in orange body $1$, and in green body $0$. Small images show the trajectories for a time propagation of ten periods. }
    \label{fig:AE4_cop0}
\end{figure}

\section{Conclusions and further work}\label{sec:conclusions}

In the present work we introduce a new methodology for the systematic numerical computation of solutions to the general planar Three-Body Problem. This implies a big step in the construction of restricted four-body problem in which to study the motion of a small body under the gravitational effect of three massive bodies in a coherent solution of the TBP, something rarely found in the literature due to the difficulties involved.

The methodology is based on two consecutive alignments of the three bodies, both in positions and in velocities, for what our solutions constitute Relative Periodic Orbits (RPO) of the TBP. Features of the solutions are exploited for reducing the numerical costs and error accumulation in their stability analysis.

Illustrative examples discussed in this paper show the great potential for constructing TBP solutions of many different natures, displaying phenomena compatible with known real systems.

The fact that the stable solutions in these families do not appear as isolated solutions, but they cover wide ranges of the continuation curves, means that there is a robust skeleton for such configurations. Besides, the fact that these solutions are stable means that there are solutions of higher dimension around the computed ones. All these results in a whole robust structure of stable solutions of the TBP for different configurations. Last statement implies that other features of the system, as small inclinations, not exact alignments or slightly different values of the masses and orbital parameters due to uncertainties, would lead to solutions that are part of this structure.

As we mention, although the present work has resulted in a wide analysis of solutions of the TBP, their stability and the diversity of RPO one can find, something not trivial itself, our final goal is the study of the small bodies in restricted four-body problems.

For this reason, we are already advancing in two specific lines. The first one is to reproduce real systems with our methodology to know whether they belong to our Solar system or not. Making comparisons among our solutions and the parameters and phenomena of real systems strengthens the use of the procedure introduced in the present work. The  second line deals with the introduction of the massless particle to our RPO. This allows to construct the corresponding restricted four-body problems and analyses the effect of the third massive body to the stability of triangular points.




\section*{Statements and Declarations}
The authors declare that they have no conflict of interest.

\section*{Acknowledgments}

The authors are deeply grateful to their mentor, Prof. \`A. Jorba, for his invaluable guidance during the initial and intermediate stages of the project, for encouraging them to pursue this work from the very beginning, and for suggesting further directions for exploration.

\smallskip

The author gratefully acknowledges Prof. A. Celletti for her comments and suggestions on an earlier version of the manuscript, which helped improve its clarity and presentation, and Prof. de la Llave for his helpful remarks on improving the introduction.

\medskip

The project has been supported by the Spanish grant
PID2021-125535NB-I00, PID2024-158570NB-100, PID2025-174015NB-I00,  (MICINN/AEI / FEDER, UE), the Catalan grant 2021
SGR 01072, and the Air Force Office of Scientific
Research under award number FA8655-24-1-7059.  
The project that gave rise to these results also received
the support of the fellowship from ``la Caixa'' Foundation (ID
100010434), the fellowship code is LCF/BQ/PR23/11980047. This work
has also been funded through the Severo Ochoa and Mar\'ia de Maeztu
Program for Centers and Units of Excellence in R\&D
(CEX2020-001084-M).

\bibliographystyle{alpha}
\addcontentsline{toc}{section}{References}
{\footnotesize \bibliography{ds, ase}}
\end{document}